\documentclass[reqno]{amsart}
\UseRawInputEncoding

\usepackage{amsmath,amsfonts,amssymb,amsthm,amscd,latexsym,cite,comment}
\usepackage{mathrsfs}
\usepackage{color}
\usepackage{enumerate}
\usepackage{comment}
\newtheorem{thm}{Theorem}[section]
\newtheorem{theorem}[thm]{Theorem}

\newtheorem{lemma}[thm]{Lemma}

\newtheorem{corollary}[thm]{Corollary}
\newtheorem{proposition}[thm]{Proposition}
\newtheorem{definition}[thm]{Definition}
\newtheorem{remark}[thm]{Remark}

\newtheorem{question}{Question}

\newtheorem{conjecture}[thm]{Conjecture}
\newcommand{\cA}{{\mathcal A}}

\newcommand{\cE}{{\mathcal{E}}}
\newcommand{\cF}{{\mathcal F}}

\newcommand{\cH}{{\mathcal H}}

\newcommand{\cM}{{\mathcal M}}

\newcommand{\cP}{{\mathcal P}}

\newcommand{\cS}{{\mathcal S}}
\newcommand{\cT}{{\mathcal T}}

\usepackage{color}
 
\newcommand{\norm}[1]{\left\lVert#1\right\rVert}

\newcommand{\tnorm}[1]{\left\vert\kern-0.25ex\left\vert\kern-0.25ex\left\vert #1 
	\right\vert\kern-0.25ex\right\vert\kern-0.25ex\right\vert}

\usepackage{hyperref}					% Reference
\hypersetup{colorlinks,%				% Reference color setup	
	linkcolor=blue,%
	citecolor=blue}
\usepackage{float}
\begin{document}
	\title[The Daugavet property  in symmetric
	operator spaces]{The Daugavet property   in  	%symmetric spaces of measurable operators
		symmetric  operator spaces 
	}

	\author[Jinghao Huang]{J. Huang}\thanks{J. Huang was supported by the NNSF of China  (No. 12031004, 12301160, 12471134 and 12671159).}
    \address{Institute for  Advanced Study in  Mathematics, Harbin Institute of Technology, Harbin, 150001, China}
    \email{{\color{blue}jinghao.huang@hit.edu.cn}}\email{{\color{blue}xxurann@stu.hit.edu.cn}}

\author[Yerlan Nessipbayev]{Y. Nessipbayev} 
    
	\author[Fedor  Sukochev]{F.  Sukochev}\thanks{F.   Sukochev was supported by the ARC (DP230100434)}
	
	\author[Ran Xu]{R. Xu}

\address{School of Mathematics and Statistics, University of New South Wales, Sydney, Australia}\email{{\color{blue}y.nessipbayev@unsw.edu.au}} \email{{\color{blue}f.sukochev@unsw.edu.au}}

	%Isometric theory of Banach spaces;Spaces of measurable functions;  General theory of von Neumann algebras; 46L52 Noncommutative function spaces;   Version~: \today


	\begin{abstract} 
		
		Under mild assumptions,  the main results of this paper characterize  a  (real or complex) symmetric operator space 
		affiliated with  a semi-finite atomless von Neumann algebra $(\cM, \tau)$   possessing the Daugavet property as  $L_1(\cM,\tau)$ or $\cM$  (up to equivalent norms, or even proportional norms). 
		Our results are new even  for complex symmetric function spaces, which extend  
		results of   \cite{AKM12, KMMW13, AKM15} concerning  
		the Daugavet property for real symmetric  function spaces.

	\end{abstract}
	\maketitle

\section{Introduction}

A Banach space $X$ has the \emph{Daugavet property} if the \emph{Daugavet equation}
\begin{align}\label{Daugavetequation}
\|I+T\|=1+\|T\|
\end{align}
holds for every rank-one bounded linear operator $T\colon X\to X$, where $I$ is the identity operator on $X$ (see, e.g., \cite[Definition 3.1.1]{KMZW25}). By \cite[Lemma 2.1]{AAB1991}, it suffices to verify \eqref{Daugavetequation} for rank-one operators of norm one. The equation arises in problems of best approximation \cite{S71} and was first studied by Daugavet \cite{D63}, who showed in 1963 that every compact operator $T\colon C[0,1]\to C[0,1]$ satisfies \eqref{Daugavetequation}. Classical examples of spaces with the Daugavet property include $C(K)$ for every compact Hausdorff space $K$ without isolated points \cite{FS65}, the spaces $L_1(\mu)$ \cite{Lo66} and $L_\infty(\mu)$ over a non-atomic $\sigma$-finite measure space (the latter being a special case of $C(K)$), and the $\ell_1$-, $\ell_\infty$-, and $c_0$-sums of arbitrary families of such spaces (cf. \cite[Theorem 1]{W92} and \cite[Subsection 4.6]{KMZW25}). Noncommutative examples are given by non-atomic $C^*$-algebras and by preduals of non-atomic von Neumann algebras \cite[Theorem 2.1]{O02}.

The Daugavet property enjoys several notable inheritance properties: (i) it passes to $M$-ideals \cite[Proposition 2.10]{KSSW00} and $L$-summands \cite[Lemma 2.15]{KSSW00}; (ii) if $X^*$ has the property then so does $X$, but the converse fails, for instance for $X=C[0,1]$ \cite[Theorem 3.3.3]{KMZW25}; (iii) it need not pass even to $1$-complemented subspaces, since $L_1$ has the Daugavet property while its $1$-complemented subspace $\ell_1$ does not \cite[page 16]{AKM15}, \cite[Example (c)]{Werner01}.

We distinguish the following notation. If $X$ is a complex normed space, then $X_{\mathbb{R}}$ denotes its realification: it consists of the same elements and has the same norm as $X$, but scalar multiplication is restricted to the field $\mathbb{R}$. By contrast, a superscript $\mathbb{R}$ will be used for genuinely real spaces, for
instance $L_1^{\mathbb{R}}(0,1)$ denotes the real $L_1$-space.

We shall use the standard identification
$$(X_{\mathbb{R}})^*\cong (X^*)_{\mathbb{R}},\qquad \Phi\mapsto \operatorname{Re}\Phi,$$
see \cite[page 61]{KMZW25}. Thus, every bounded real-linear functional on
$X_{\mathbb{R}}$ is the real part of a unique bounded complex-linear
functional on $X$.

In the commutative setting, the Daugavet property has been studied extensively for rearrangement-invariant (equivalently, symmetric) function spaces. The following conjecture is recorded in \cite[Section 4.4, page 122]{KMZW25}; see also \cite[page 4077]{AKM15} for the corresponding discussion on $(0,\infty)$.

\begin{conjecture}\label{conj}
The only real (respectively, complex) rearrangement-invariant function spaces on $(0,1)$ with the Daugavet property are, up to isometric lattice isomorphism, $L_1^{\mathbb{R}}(0,1)$ or $L_\infty^{\mathbb{R}}(0,1)$ (respectively, $L_1(0,1)$ or $L_\infty(0,1)$).
\end{conjecture}

The first substantial progress toward Conjecture \ref{conj} is due to Acosta, Kami\'nska, and Masty\l o \cite{AKM12}. Let $X$ be a real symmetric function space over a finite atomless measure space $(\Omega,\mathcal{S},\mu)$, and suppose that $X$ has the Daugavet property. They proved two complementary results. First, if $X$ is a KB-space, equivalently, if $X$ has the Fatou property and an order-continuous norm, then $X$ coincides with $L_1^{\mathbb{R}}$ as a set \cite[Proposition 2.5]{AKM12} (see also \cite[Corollary 4.4.7]{KMZW25}). Second, if $\mu$ is separable, $X$ has the Fatou property, and the K\"othe dual $X^\times$ is either strictly monotone or order continuous, then $X=L_\infty^{\mathbb{R}}$ up to equivalence of norms \cite[Proposition 2.6]{AKM12}.

Later, Kadets, Mart\'in, Mer\'i, and Werner obtained the corresponding isometric statement in the separable case: the only separable real symmetric function space on $(0,1)$ with the Daugavet property is $L_1^{\mathbb{R}}(0,1)$ with its canonical norm \cite[Corollary 4.9]{KMMW13} (recall that, for symmetric function spaces, separability is equivalent to order continuity of the norm \cite[page 102]{KPS}).  Inspired by \cite{KMMW13}, Acosta, Kami\'nska, and Masty\l o showed that a real symmetric function space $X$ over a finite atomless measure space which has a Fatou norm, has the Daugavet property, and satisfies
$$\lim_{t\to0+}\phi_{X^\times}(t)=0,$$
where $\phi_{X^\times}$ denotes the fundamental function of $X^\times$, is isometrically isomorphic to $L_\infty^{\mathbb{R}}$ \cite[Theorem 3.4]{AKM15}. Combining these results gives a complete description in the real case: a real symmetric function space over a finite atomless measure space with the Fatou norm and the Daugavet property is isometrically isomorphic to either $L_1^{\mathbb{R}}$ or $L_\infty^{\mathbb{R}}$ \cite[Theorem 3.6]{AKM15}. This provides strong evidence for Conjecture \ref{conj} in the real case.

The finiteness of the measure is essential in Conjecture \ref{conj}: on $(0,\infty)$ the analogous statement fails. Indeed, the space
$$L_\infty^0(0,\infty)=\{f\in L_\infty(0,\infty): f^*(t)\to0\text{ as }t\to\infty\},$$
equipped with the norm $\|f\|=\|f\|_\infty$ (here $f^*$ is the decreasing rearrangement of $f$), is a rearrangement-invariant Banach function space with the Daugavet property \cite[Theorem 2.1]{O02} that is neither $L_1(0,\infty)$ nor $L_\infty(0,\infty)$. Positive results nonetheless exist in the infinite setting; for example, a uniformly monotone real symmetric function space on $(0,\infty)$ with the Daugavet property must be isometrically isomorphic to $L_1^{\mathbb{R}}(0,\infty)$ \cite[Theorem 4.4]{AKM15}.

The present paper studies the noncommutative counterpart of these results. Let $(\mathcal{M},\tau)$ be a semi-finite von Neumann algebra with identity $\mathbf{1}$ and a faithful normal semi-finite trace $\tau$. We write $S(\mathcal{M},\tau)$ for the space of $\tau$-measurable operators affiliated with $\mathcal{M}$, and $S(\mathcal{M},\tau)_h$ for its self-adjoint part \cite[Section 2]{DPS}. To each symmetric function space $E$ on $(0,\tau(\mathbf{1}))$ one associates, through the generalized Calkin correspondence \cite{Kalton_S}, the symmetric operator space $E(\mathcal{M},\tau)$: an operator $x$ belongs to $E(\mathcal{M},\tau)$ precisely when its generalized singular value function $\mu(x)$ belongs to $E$. Further details are given in Section \ref{s2}. We are guided by the following question.

\begin{question}\label{Q}
Let $(\mathcal{M},\tau)$ be a semi-finite atomless von Neumann algebra and let $E$ be a non-zero symmetric function space. If $E(\mathcal{M},\tau)$ has the Daugavet property, must $E(\mathcal{M},\tau)$ be, in an appropriate sense, one of $L_1(\mathcal{M},\tau)$ or $\mathcal{M}$?
\end{question}

When $\mathcal{M}=L_\infty(0,\tau(\mathbf{1}))$, the space $E(\mathcal{M},\tau)$ reduces to the function space $E(0,\tau(\mathbf{1}))$, so any answer to Question \ref{Q} also bears on Conjecture \ref{conj}, including its complex version. The noncommutative setting, however, presents substantial difficulties; in particular, one no longer has a lattice structure available. %Moreover, symmetric function spaces and symmetric operator spaces may have very different subspace structures; for example, by a theorem of Arazy and Lindenstrauss \cite[Theorem 6]{AL75}, $L_p(0,1)$ does not embed into the Schatten class $C_p$ for $2<p<\infty$. Related non-embedding phenomena were studied in \cite{GL74,Lewis75,G97,AHS21,HSSZ25}.

We now describe the main results of the paper.

\subsection*{The complex case}

Our first main result is a noncommutative complex analogue of \cite[Propositions 2.5 and 2.6]{AKM12}. It is obtained by combining Theorems \ref{sigmafiniteL1subseteqE}, \ref{infiniteL1subseteqE}, and \ref{assumption=Linfty} below.

\begin{theorem}\label{Thm 1.2}
Let $(\mathcal{M},\tau)$ be a semi-finite atomless von Neumann algebra and let $E$ be a symmetric function space on $(0,\tau(\mathbf{1}))$. Assume that $E(\mathcal{M},\tau)$ has the Daugavet property.
\begin{enumerate}[\rm(i)]
\item If $(\mathcal{M},\tau)$ is $\sigma$-finite and $E$ is KB, then $L_1\subseteq E$ on $(0,\tau(\mathbf{1}))$.
\item If $(\mathcal{M},\tau)$ is non-$\sigma$-finite, $E$ is KB, and $E\cap L_\infty\neq L_1\cap L_\infty$, then $L_1\subseteq E$ on $(0,\infty)$.
\item If $(\mathcal{M},\tau)$ is non-$\sigma$-finite and $E$ is uniformly monotone, then $L_1\subseteq E$ on $(0,\infty)$.
\item If $\tau(\mathbf{1})<\infty$ and $E$ is KB, then $E=L_1$ on $(0,\tau(\mathbf{1}))$, up to equivalence of norms.
\item If $\tau(\mathbf{1})<\infty$, $E$ has a Fatou norm, and either $E^\times$ is strictly monotone or $E^\times$ is order continuous, then $E=L_\infty$ on $(0,\tau(\mathbf{1}))$, up to equivalence of norms.
\end{enumerate}
\end{theorem}

Theorem \ref{Thm 1.2} is a structural result concerning the underlying function space: under natural assumptions, the Daugavet property of $E(\mathcal{M},\tau)$ forces the commutative space $E$ to lie at one of the two expected endpoints. In the finite case it recovers  $E=L_1$ and $E=L_\infty$. Since the commutative case is contained in the noncommutative framework, this already yields new information for complex rearrangement-invariant function spaces.

The proof combines several ingredients. The first is an inheritance principle for the Daugavet property in reduced symmetric operator spaces, see Theorem \ref{EpDPinherit} below.  The second is noncommutative Kadec--Pe\l czy\'nski lemma \cite{Ran03}. A further key step is Theorem ~\ref{dsl1}, which relates asymptotically isometric copies of $\ell_1$ in $E(\mathcal{M},\tau)$ to disjointly supported asymptotically isometric copies of $\ell_1$ in $E$; this lets us pass from noncommutative geometry back to the K\"othe dual $E^\times$. Finally, we use the fact that a Banach space with the Daugavet property cannot have the Radon--Nikod\'ym property \cite[Corollary 1]{W92}, together with recent characterizations of the WCG property in noncommutative symmetric spaces \cite{HNPS24Tran}.

\subsection*{The real case}

The real theory is sharper. The results of \cite{KMMW13,AKM15} identify real symmetric function spaces with the Daugavet property not merely as sets, but up to proportional norms, i.e.
$$\norm{\cdot}_E = \lambda\norm{\cdot}_{L_1} \quad\text{or}\quad \norm{\cdot}_E = \lambda\norm{\cdot}_{L_\infty}$$
for some constant $\lambda>0$; note that passing to a proportional norm preserves the Daugavet property.

Crucially, the proofs in \cite{KMMW13,AKM15} are essentially real. As observed in \cite[Remark 4.11]{KMMW13} and \cite[page 124, Question (4.6)]{KMZW25}, the corresponding complex argument would require an analogue of the real-valued estimate with $\operatorname{Re} f$ in place of $f$, and no such substitute is known. Motivated by this, Sections \ref{sec5} and \ref{sec6} study real symmetric operator spaces consisting of self-adjoint $\tau$-measurable operators.

In the finite case,  we obtain a noncommutative analogue of \cite[Theorem 3.6]{AKM15}.

\begin{theorem}\label{mainreal}
Let $(\mathcal{M},\tau)$ be a finite atomless von Neumann algebra with a finite faithful normal trace $\tau$. If a symmetric operator space $\mathcal{E}\subseteq S(\mathcal{M},\tau)_h$ has a Fatou norm %$\norm{\cdot}_{\mathcal{E}}$ 
and the Daugavet property, then $\mathcal{E}$ is isometrically equal to either $L_1(\mathcal{M},\tau)_h$ or $\mathcal{M}_h$. In particular, the corresponding real symmetric function space coincides with $L_1^{\mathbb{R}}$ or $L_{\infty}^{\mathbb{R}}$ on $(0,\tau(\mathbf{1}))$ up to proportional norms.
\end{theorem}

The infinite case is treated in Section \ref{sec6}, where we prove that a uniformly monotone real symmetric operator space with the Daugavet property is isometrically isomorphic to $L_1(\mathcal{M},\tau)_h$; consequently, the corresponding real symmetric function space is $L_1^{\mathbb{R}}$ on $(0,\infty)$ up to a proportional norm. The proofs in the real self-adjoint setting rest on the geometric slice characterization of the Daugavet property for real Banach spaces, the ideas of \cite{KMMW13,AKM15}, and several estimates based on Hardy--Littlewood--P\'olya submajorization in the noncommutative setting.

The final section (Section~\ref{sec7}) is devoted to establishing 
that  the predual of an arbitrary atomless (not necessarily semi-finite) von Neumann algebra,  as a complex Banach space, has the operator Daugavet property (see Theorem \ref{thm:ODP} below).
This extends \cite[Proposition 4.4]{RTV},  where     the operator Daugavet property of $L_1(0,1)$ was proved. 
It should be noted that the operator Daugavet property implies the Daugavet property; however, there does not  exist any    example   satisfying the Daugavet property  but not the operator Daugavet property \cite[Remark 4.2]{RTV}. 
In particular, we do not know whether all atomless  von Neumann algebras possess the operator Daugavet property. 
Note   that the abelian von Neumann algebra $L_\infty(0,1)$ has the operator Daugavet property\cite[Remark 4.5(2)]{RTV}.

	\section{Preliminaries}
	\label{s2}
	
	General information on von Neumann
	algebras, Banach function lattices  and symmetric operator spaces  can be found in
	\cite{LSZ,KPS,LT2,Kalton_S,DPS}.  
	We now collect some of the basic facts and notation that will be used in this
	paper.

	If $(X, \norm{\cdot}_{X})$ is a
	Banach space, then by $B_X$ and $S_X$ we denote the closed unit ball and the unit sphere  of $X$,
	respectively.
	As usual by $\mathbb{R}$, $\mathbb{R}_{+}$
	and $\mathbb{N}$ we denote the set of real, non-negative real and natural numbers, respectively.
	
	%	Let $(X, \leq)$ be a partially ordered set. If $U$ is a non-empty subset of $X$	for which the least upper bound (or, supremum) exists, then this least upper bound	is denoted by $\vee U$. 	

	Let \((I, \Sigma, m)\) denote the measure space  \( I = (0, a) \) for some \( a \in (0, \infty] \), endowed with the Lebesgue measure \( m \).
	Denote by \( L(I, m) \) the space of   all (equivalence classes of) $\Sigma$-measurable complex valued 
	functions on \( I \). 
	%where functions that agree almost everywhere are identified. 
	Denote by $L_p^{\mathbb{R}} 
	$ (respectively, $L_p
	$)
	the space of all $p$-integrable real  (respectively, complex) valued 
	$\Sigma$-measurable functions on $I
	$ ($1\leq p<\infty$)
	and by  $L_{\infty}^{\mathbb{R}}
	$  (respectively, $L_\infty
	$)
	the space
	of all $m$-essentially bounded real (respectively, complex) valued  $\Sigma$-measurable functions on $I%(0, \infty)
	$, 
	equipped with the usual norms
	$$\norm{f}_{L_p}=\left( \int_{I} |f|^{p} dm \right)^{\frac{1}{p}}, \quad 1 \le p<\infty, \quad   \norm{f}_{L_{\infty}}= \operatorname*{ess\,sup}_{t\in I}|f(t)|.   
	$$
	
	Throughout this paper, we denote by $\cM$ a semi-finite von Neumann algebra  with identity ${\bf 1}$, the operator norm on the underlying Hilbert space $\cH$ by $\norm{\cdot}_{\infty}$ and a semi-finite  faithful normal trace on $\mathcal{M}$ by $\tau$.  
	Let  $\cP(\mathcal{M})$ 
	be the complete lattice of all  (orthogonal) projections in $\mathcal{M}$. 		
	
	\subsection{$\tau$-measurable operators}

	A closed and densely defined operator $x$, affiliated with    $\mathcal{M}$, is said to be {\it $\tau$-measurable},  if
	$\tau(e ^{|x|}(s,\infty))<\infty$ for sufficiently large $s$, where $e ^{|x|}(\cdot)$ denotes  the {\it spectral measure} of $|x|$.
	Denote by 	$S(\mathcal{M},\tau)$  the set of all $\tau$-measurable operators \cite{LSZ,DPS}.
	
	If $U$ is a subset of 	$S(\mathcal{M},\tau)$, then we write $U_{h} := \{x\in U: x = x^*\} $ and $U^{+} := \{x\in U: x \geq 0 \}.$
	If $\{x_{\alpha}\} $ is an increasing net in $S(\cM, \tau)_h$ and $x = \sup{x_{\alpha}}\in S(\cM, \tau)_h$,
	we write $x_{\alpha}\uparrow x$. In the case of a decreasing net $\{x_{\alpha}\} $ with infimum 0,  we write
	$x_{\alpha}\downarrow 0$.

	For every $x\in S(\mathcal{M},\tau),$  the {\it singular value function} $\mu(x)$ is defined by setting 
	$$
	\mu(t; x)
	=\inf\{s\geq0:\ \tau(e ^{|x|}(s,\infty))\leq t \},\quad t\ge0.$$
	The extended trace $\tau : S(\cM, \tau)^+ \to [0, \infty]$ is given by \cite[Proposition 3.3.3]{DPS}
	\[
	\tau(x) = \int_0^\infty \mu(t; x)  dt, \quad x \in S(\cM, \tau)^+. 
	\]

	For $x \in S(\cM, \tau)$, the right and left support projections of $x$ are denoted by $r(x)$ and $l(x)$,  respectively \cite[Subsection 1.4]{DPS}.
	The two-sided ideal $\cF(\cM, \tau)$  in $\cM$ is defined by setting \cite[page 74]{DPS} 
	\[
	\cF(\cM, \tau) := \{ x \in \cM : \tau(r(x)) < \infty \}= \{ x \in \cM : \tau(l(x)) < \infty \}. 
	\]

	For all $x, y\in S(\mathcal{M}, \tau)$,  $x$ is said to be {\it submajorized} by $y$ (in the sense of Hardy, Littlewood and P\'olya) and written  
	$x\prec\prec y$  if
	$$\int_{0}^s \mu(t;x)dt\le \int_{0}^s \mu(t;y)dt$$ for all  $s>0$. In particular, $x\prec\prec y$ is   
	equivalent to $\mu(x)\prec \prec \mu(y)$ \cite[Section 3.3]{LSZ}.

	The algebra 
	\( 
	L_\infty (I) 
	\) 
	can 
	be 
	seen 
	as 
	an abelian von Neumann algebra acting via multiplication on the Hilbert space \(  L_2 (I) \), 
	with the trace given by integration with respect to Lebesgue measure \( m \).  
	The algebra of all \( \tau \)-measurable operators affiliated with \(L_\infty (I) \) can be identified with the   linear subspace \( S(I) \)  of \( L(I) \) consisting of all complex valued measurable  functions \( f \) for which there exists some \( s > 0 \) such that
	$
	m(\{ t \in I : |f(t)| > s \}) < \infty
	$~\cite{KPS,LSZ,LT2}.
	%The  real part of   \( S(I) \)   is denoted by $S^{\mathbb{R}}(I)$.  
	
	%We now recall two important topologies  on $S(\cM,\tau)$.
	
	For $0 < \varepsilon, \delta \in \mathbb{R}$, 
	define 	\[
	V(\varepsilon,\delta):=\{x\in S(\cM,\tau):\exists   \, p\in \cP(\cM) \,\text{such that} \, \norm{xp}_{\infty}\leq \varepsilon, \tau(\mathbf{1}-p)\leq \delta\}.
	\] 
	The collection $\{V(\varepsilon,\delta):\varepsilon,\delta>0\}$ is a neighborhood base at $0$  for a complete metrizable Hausdorff vector space topology $\cT_m$ on $S(\cM,\tau)$,  called the {\it  measure topology}, see    e.g., \cite[Proposition 2.5.3]{DPS} and   \cite[Theorem 2]{N74}.    
	A sequence \(\{x_n\}_{n=1}^{\infty}\) in \(S(\cM, \tau)\) converges to $0$ in $\cT_m$, denoted by 
	$x_{n}\stackrel{\cT_m}{\to}0$, if and only if \(\mu(t; x_n) \to 0\) as \(n \to \infty\) for all \(t > 0\) \cite[Proposition 3.2.1]{DPS}.

	For $0 < \varepsilon, \delta \in \mathbb{R}$ and $e \in \cP(\mathcal{M})$ with $\tau(e) < \infty$, set 
	\[
	V(\varepsilon, \delta, e) := \{x \in S(\mathcal{M}, \tau) : exe \in V(\varepsilon, \delta)\}. 
	\]
	The collection
	$
	\bigl\{ V(\varepsilon, \delta, e) : \varepsilon, \delta > 0,\; e \in \cP(\mathcal{M}),\; \tau(e) < \infty \bigr\}
	$ 
	forms a neighbourhood base at $0$ for a Hausdorff vector space topology $\mathcal{T}_{lm}$ in $S(\mathcal{M},\tau)$,   
	called the \emph{local measure topology}. 
	If	a net $\{x_\alpha\}\subseteq S(\mathcal{M}, \tau)$ converges  to $x \in S(\mathcal{M}, \tau)$ in  $\mathcal{T}_{lm}$, we write  
	$
	x_\alpha \stackrel{\mathcal{T}_{lm}}{\to} x 
	$
	\cite[Section 2.7]{DPS}.

	\subsection{Symmetric  function spaces and  fundamental functions}\label{section:symmetric}

	Symmetric function spaces and their fundamental functions   are classical topics   in the theory of  Banach function lattices;   see \cite{BS88, LT2,KPS} for more details.

	A Banach  function space $(E,\norm{\cdot}_E) \subseteq S(I)$ is said to be 
	{\it symmetric} if 
	$x\in E$,  $y\in S(I)$  and  $\mu(y)\leq\mu(x)$ imply that  $y\in E$ and $\left\|y\right\|_E\leq \left\|x\right\|_E.$ 
	If, in addition,   
	\(\| y \|_E \leq \| x \|_E\) whenever  \( x, y \in E \) satisfy  $y\prec \prec x$,  	then 
	$(E, \norm{\cdot}_{E})$ is said to be 	{\it  strongly symmetric}. 
	A symmetric  function space $(E, \norm{\cdot}_{E})$  is said to be 	{\it  fully symmetric}  if  \(\norm{\cdot}_E\) has the additional property that    for any $x\in E$ and $y\in S(I)$,  $y\prec \prec x$  implies that $y\in E$ and $\norm{y}_E\le\norm{x}_E$.

	Suppose that %$(I, m)$ is atomless and  
	\( E\subseteq S(I) \) is a symmetric
	function space.
	The {\it K\"othe dual space} \(E^{\times}\) of \(E\) is the collection of all elements \(y \in S(I)\) such that
	\[
	\|y\|_{E^{\times}} = \sup\left\{\int_I  |xy|\,dm : \|x\|_{E} \le 1\right\} < \infty.
	\]
	The {\it fundamental function $\phi_E$ of  \(E\)} is defined by 
	$\phi_E(t) = \| \chi_A \|_E,$  where \(0 \leq m(A)=t  \leq m(I)\). 
	Some basic properties of $\phi$ are well known (cf. \cite[page 66, Theorem 5.2 and Corollary 5.3]{BS88}, \cite[page 107, Eq(4.39)]{KPS} and \cite{KMMW13}):
	\begin{enumerate}[\rm(i)]
		
		\item $\phi_E(t)=0$ if and only if $t=0$;
		\item  
		$\phi_E$ is  non-decreasing;
		\item  
		$\phi_E(t)/t$ is decreasing;
		\item 	$\phi_E$ is continuous, except perhaps at the origin; 
		\item 
		%\begin{align}	\label{propertiesfundamental}
		$	\phi_E(t+s) \leq \phi_E(t)+\phi_E(s).
		$ 
		\item  The  fundamental function of   \(E^{\times}\) satisfies \begin{align}\label{defEtimes}
			\phi_{E^{\times}}(t) =\frac{t}{\phi_E(t)}, \quad 0 < t \leq m(I).
		\end{align} 
	\end{enumerate}  Define  \(\phi_E(0+) := \mathop{\lim}\limits_{t \to 0^+} \phi_E(t)\).

	\subsection{Construction of symmetric operator spaces}
	
	A Banach operator space \( (\cE, \norm{\cdot}_{\cE})\subseteq S(\cM, \tau) \)  will be  called
	\begin{enumerate}[\rm(i)]
		\item  {\it symmetric} 
		if \( x \in S(\cM, \tau) \), \( y \in \cE \) and \( \mu(x) \leq \mu(y) \) imply that \( x \in \cE \) and \(\| x \|_{\cE} \leq \| y \|_{\cE}\). 
		\item  {\it strongly symmetric}   if  \(\norm{\cdot}_{\cE}\) has the additional property that \(\| x \|_{\cE}\leq \| y \|_{\cE}\) whenever \( x, y \in \cE \) satisfy \( x \prec \prec y \).
		\item  {\it fully symmetric} if    \(\norm{\cdot}_{\cE}\) has the additional property that \( x \in S(\cM, \tau) \), \( y \in \cE \) and \( x \prec \prec y \) imply that \( x \in \cE \) and \(\| x \|_{\cE} \leq \| y \|_{\cE} \).
	\end{enumerate}
	
	If  $\cM$ is  atomless  and  a symmetric operator  space   $\cE\neq \{0\}$, then  the  carrier projection $c_{\cE}:=\vee\{ p : p \in \mathcal{P}(\cM)\cap \cE \} \in \cP(\cM)$      of   $\cE$  \cite[Definition 4.1.5]{DPS}  is ${\bf 1}$, where $\vee$ denotes the supremum \cite[Lemma 4.4.5 (ii)]{DPS}.

	%Numerous symmetric operator spaces associated with	$\cM$ can be derived from concrete symmetric function spaces, see e.g. \cite{LSZ}. 
	Let
	$(E, \norm{\cdot}_{E})$ be a symmetric  function   space on $(0, \tau({\bf 1}) )$.
	Then, we obtain a symmetric operator space (operator counterpart) by the following  construction: 
	\begin{align*}
		E(\mathcal{M},\tau):=\Big\{x\in S(\mathcal{M},\tau):\ \mu(x)\in E \Big\}
		, \quad  
		\left\|x\right\|_{E (\mathcal{M},\tau)}:= \left\|\mu(x)\right\|_{E}.
	\end{align*}
	Many properties 
	carry over from   \( E \) to its operator   counterpart \( E(\mathcal{M}, \tau) \),   see e.g. \cite[Theorem 6.1.7]{DPS}, \cite[Theorems 53  and 54]{DD14} and \cite{DDST05}. 
	When $E=L_p(0, \infty)$ for  $1 \leq p \leq \infty$, the associated space $E(\cM, \tau)$ coincides with the noncommutative $L_p$-space  $L_p(\cM, \tau)$. In particular,  the von Neumann algebra
	$\cM$ is also denoted by 
	$L_\infty (\cM, \tau)$ \cite[Remark 3.4.15]{DPS}. 
	%When  $\mathcal{M}$ is $B(\mathcal{H})$, $L_p(\cM, \tau)$  is the   Schatten-$p$ class \cite{LSZ}. 

	Let \( \cM \) be  atomless and  let \( E(\cM, \tau) \) be a symmetric operator space  associated with a symmetric function space $E$. The {\it fundamental function}  of \(E(\cM, \tau)\) is defined by $$\phi_{E(\cM, \tau)}(t) = \|p\|_{E(\cM, \tau)},$$
	where \(p \in \mathcal{P}(\cM)\), \(\tau(p) = t\) for $0\le t\le \tau(\mathbf{1})$ ~\cite[page 43]{CS94}. Observe that,  \(\phi_{E(\cM, \tau)}(t) = \phi_{E}(t)\). 
	If \( E(\mathcal{M}, \tau)\neq \{0\} \) and \( y \in S(\mathcal{M}, \tau) \), then
	\begin{align}\label{twodefytimes}
		&~\quad \nonumber 
		\sup \left\{ 
		\tau(|xy|): x \in E(\mathcal{M}, \tau),  \|x\|_{E(\mathcal{M}, \tau)} \leq 1 
		\right\}
		\\ &\nonumber= \sup \left\{ 
		|\tau (xy)| : x \in E(\mathcal{M}, \tau),  \|x\|_{E(\mathcal{M}, \tau)} \leq 1 
		\right\} 
		\\ & =
		\sup \left\{ \int_{0}^{\infty} 
		\mu(t; x) \mu(t; y) dt: 
		x \in 
		E(\mathcal{M}, \tau), \|x\|_{E(\mathcal{M}, \tau)} \leq 1 
		\right\}.
	\end{align}
	In particular, \( y  \) belongs to the {\it  K\"{o}the dual} 
	\( E(\mathcal{M}, \tau)^\times \) 
	of 
	\( 
	E(\mathcal{M}, \tau) \)   if and only if the quantity defined in  \eqref{twodefytimes} is finite. Moreover,   $ \| y \|_{E(\mathcal{M}, \tau)^\times}$ is given by that same quantity, 
	%\[\sup \left\{ \int_{0}^{\infty} \mu(t; x) \mu(t; y) dt : x \in E(\mathcal{M}, \tau), \| x \|_{E(\mathcal{M}, \tau)} \leq 1 \right\} < \infty,\]in which case\[\| y \|_{E^\times} = \sup \left\{ \int_{0}^{\infty} \mu(t; x) \mu(t; y) \, dt : x \in E(\mathcal{M}, \tau), \| x \|_{E(\mathcal{M}, \tau)} \leq 1 \right\},\]
	see e.g. \cite[Section 5.2]{DD14} and  \cite{LSZ}.  
	%It is well-known that \( \norm{\cdot}_{E(\mathcal{M}, \tau)^\times} \) is a norm on \( E(\mathcal{M}, \tau)^\times \) if and only if  \( c_{E(\cM, \tau)} = \mathbf{1} \) \cite[Lemma 4.3.3]{DPS}. 
	%Let $\cM$ be a von Neumann algebra which is either non-atomic or atomic with all atoms having equal trace. 
	If $E(\cM, \tau)\neq {0}$ ($\mathcal{M}$ is atomless), then $E(\cM, \tau)^{\times}$ is a fully symmetric space with $c_{E(\cM, \tau)^{\times}}={\bf 1}$ and the Fatou property  \cite[Theorem 4.5.5]{DPS}.  
	%The closure $E(\cM, \tau)^b := \overline{\mathcal{F}(\cM, 	\tau)}^{\norm{\cdot}_{E(\cM, \tau)}}  	$ of   \( \mathcal{F}(\cM, \tau) \) in  \( E(\cM, \tau) \) is a fully symmetric space \cite[Theorem 4.5.8 and  Proposition 5.1.13]{DPS}.

	Observe that, 
	if $\cM$ is atomless and \( \tau(\mathbf{1}) <\infty \),  for a  symmetric operator space $E(\cM, \tau)$, 
	we have   
	$
	\cM \subseteq   E(\cM, \tau) \subseteq L_1(\cM, \tau)  
	$
	and 
	\begin{align}\label{L1tau1phi1}
		\|x\|_{E(\cM, \tau)}  \notag & \leq \|\mathbf{1}\|_{E(\cM, \tau)}  \|x\|_\infty, \quad x \in \cM, 
		\\ \quad 
		\|x\|_{L_1(\cM, \tau)}  & \leq \frac{\tau(\mathbf{1})}{\|\mathbf{1}\|_{E(\cM, \tau)}} \|x\|_{E(\cM, \tau)},\quad
		x\in E(\cM, \tau), 
	\end{align}
	see e.g. \cite[Example 2.6.7]{LSZ} and  \cite[Theorem 4.4.6]{DPS}.

	Assume that  \( \cE \) is a strongly symmetric  operator space.   Then the space \( \cE_n^* \)  of all normal  linear functionals on  \( \cE \) may be identified with the K\"othe dual \( \cE^\times \) by trace duality, see e.g.  \cite[Theorem 5.9]{DD12} and \cite[Theorem 5.2.9]{DPS}. 
	More precisely, 
	if $\varphi \in \cE_n^*$,  there exists $y \in \cE^\times$ with  
	\begin{align}\label{normalKotheduality}
		\|y\|_{\cE^\times} = \|\varphi\|_{\cE^*}  \quad 
		\text{such that} \quad 	\varphi(x) = \tau(xy), \quad \forall x \in \cE. \end{align}
	%Note that, if $c_{\cE_{oc}}= {\bf 1}$, then a functional $\phi$ is singular if and only if $\phi$ vanishes on $\cE_{oc}$ \cite[page 328]{DPS}.

	Suppose that  \( (X, \norm{\cdot}_{X})  \subseteq S(\cM, \tau) \) is a normed   \( \cM \)-bimodule and  \( p \in \cP(\cM) \). 
	%, where $(\cM, \tau)$ is a semi-finite       von Neumann algebra. 
	The reduced von Neumann algebra \( \cM_p \) is equipped with the trace \( \tau_p \) given by \( \tau_p(a_p) = \tau(pap) \),  
	\( a \in \cM^{+} \). Define  the reduced spaces: 
	\[
	p Xp = \{pxp : x \in  X\}, \quad  X_p = \{x_p \in S(\cM_p, \tau_p) : x \in  X\}.
	\]
	Then,   
	the mapping  \( x \mapsto x_p \) is a \(*\)-preserving linear isomorphism from \( p Xp \) onto the   \( \cM_p \)-bimodule \(  X_p \) of \( \tau_p \)-measurable operators   \cite[Lemma 3.7.1]{DPS}. 
	Moreover,  endowing 
	\(  X_p \) 
	with the norm
	\[
	\|x_p\|_{ X_p} := \|pxp\|_ X, \quad x \in  X,   
	\]
	the space 
	\(  X_p \) 
	becomes a normed \( \cM_p \)-bimodule and 
	the mapping  \( x \mapsto x_p \) is an isometry from \( p Xp \) onto \(  X_p \) \cite[Subsection 4.1]{DPS}.

	Suppose that 	\((\cE, \norm{ \cdot }_{\cE}) \) is 
	a  symmetric function/operator space. 
	Then,   \(\norm{ \cdot }_{\cE}\) is    called a \emph{Fatou norm} if for any upward directed net \(\{x_\alpha\}\) in \(\cE^{+}\) and \( x \in \cE^{+}\), it follows from \(0 \leq x_\alpha \uparrow x\) that \(\|x_\alpha\|_{\cE}\uparrow \|x\|_{\cE}\) \cite[Definition 4.4.14]{DPS}. 
	\((\cE, \norm{ \cdot }_{\cE}) \) is  said to have the {\it Fatou property}  if for every upward directed net  \( \{x_\alpha\} \) in \( \cE^{+} \), satisfying \( \sup_\alpha \|x_\alpha\|_{\cE} < \infty \), there exists an element \( x \in \cE^{+} \) such that \( 0\leq x_\alpha \uparrow x \) in \( \cE \) and \( \|x\|_{\cE} = \sup_\alpha \|x_\alpha\|_{\cE} \) \cite[Definition 4.1.21]{DPS}.
	%Examples such as Schatten-von Neumann operator ideals, Lorentz operator ideals, Orlicz operator ideals, etc. all have symmetric norms which have the  Fatou property.
	%Suppose that the von Neumann algebra $\cM$ is either atomless or atomic and all minimal projections have equal trace. 

	\section{Inheritance of the Daugavet property in  symmetric operator spaces}\label{sec3}

	This section is devoted to the inheritance of the Daugavet property by natural subspaces of noncommutative symmetric operator spaces,  thereby extending  \cite[Theorem 2.1]{AKM12} and 	\cite[Theorem 4.9]{AKM15}.

	The first  result (see Theorem \ref{inheriteDP})  serves as  a noncommutative version   of \cite[Theorem 2.1]{AKM12}, which established that for a real Banach function lattice $X$ on  a $\sigma$-finite atomless measure space with  a  Fatou norm and the Daugavet property, its subspace of all elements of order continuous norm (provided it is non-zero) inherits the Daugavet property.

	To formulate our results, let us recall some notations.
	Let $(\cM, \tau)$ be a semi-finite von Neumann algebra and  \( (\cE, \norm{\cdot}_{\cE}) \subseteq S(\cM, \tau) \) be a  symmetric operator space.

	The norm $\norm{\cdot}_{\cE}$  is said to be \textit{order continuous} if \( \|x_\alpha\|_{\cE} \downarrow_\alpha 0 \) whenever a  net  \( x_\alpha \downarrow_\alpha 0 \subseteq \cE \), or equivalently,  \( \|x_n\|_{\cE} \downarrow 0 \) for every decreasing sequence \( \{x_n\}_{n=1}^\infty\subseteq \cE \)  satisfying  \(x_n\downarrow 0 \)  \cite[Definition 5.3.1 and Lemma 5.3.3]{DPS}.

	The subset $\cE_{oc} \subseteq \cE$ is defined by setting
	\[
	\cE_{oc} := \{ x \in \cE : |x| \ge x_n \downarrow 0 \;\Rightarrow\; \|x_n\|_{\cE} \downarrow 0 \text{ as } n \to \infty \},  
	\]
	and the elements of $\cE_{oc}$ are called \emph{elements of order continuous norm} \cite[Definition 5.4.1  and Remark 5.4.2]{DPS}.  
	Recall that 
	$(\cE, \norm{\cdot}_{\cE})$ is  order continuous (i.e., $ \norm{\cdot}_{\cE}$ is an order continuous norm) if and only if $\cE_{oc} = \cE$ \cite[page 6448]{DD12}.  
	%every non-zero order continuous symmetric operator space is fully symmetric when $\cM$ is atomless \cite[Corollaries 5.3.4 and  5.3.6]{DPS}.

	Note that   
	\( \cE_{oc}\) is fully  symmetric  when $\cM$ is atomless and $c_{\cE_{oc}}={\bf 1}$  (i.e.,   $\cE_{oc}\neq \{0\}$)  \cite[Theorem 4.5.8 and Proposition  5.4.8]{DPS}. However, for general semi-finite $\mathcal{M}$ with atoms, $\cE_{oc}$ may fail to be symmetric even  if $\cE$ is  strongly symmetric and $c_{\cE_{oc}} = {\bf 1}$ \cite[Remark 5.4.7]{DPS}.

	\begin{theorem}
		\label{inheriteDP}
		Let  $(\cM, \tau)$  be a  semi-finite  atomless  von Neumann algebra   and  let 
		$ \norm{\cdot}_{E(\cM, \tau)}$ be a Fatou norm  on   a symmetric operator space $E(\cM, \tau)$. 
		Assume that $E(\mathcal{M},\tau)_{\mathrm{oc}}\neq\{0\}$ and  
		$E(\cM, \tau)$ has the Daugavet property,  then $E(\cM, \tau)_{oc}$
        %$E(\cM, \tau)_{oc}=E_{oc}(\cM, \tau):= \{x \in E (\cM, \tau) : \mu(x) \in E_{oc}\}$ 
        also has  the Daugavet property.  
		%	where $E_{oc}(\cM, \tau)=\{ x\in E(\cM, \tau): \mu(x)\in E_{oc}\}$.   

	\end{theorem}

	\begin{proof}
		Since  $\norm{\cdot}_{E(\cM, \tau)}$ is a Fatou norm,  
		it follows from \cite[Corollary  4.5.7]{DPS} that 
		$E(\cM, \tau)$   is  strongly symmetric, and so,  $E(\cM, \tau)_{oc}=E_{oc}(\cM, \tau)$ \cite[Proposition 6.1.9]{DPS}. 
		Since  
		$ \norm{\cdot}_{E(\cM, \tau)}$ is  order continuous on  \( E(\cM, \tau)_{oc} \) and $c_{E(\cM, \tau)_{oc}}={\bf 1}$ \cite[Proposition 5.4.8]{DPS},    
		it follows  that  
		\begin{align}\label{M-bimodule*=times}
			E(\cM, \tau)^{\times}
			& \,\,\,\,\, \stackrel{\tiny \mbox{\cite[Thm.6.1.6]{DPS}
				}
			}{=}	E^{\times}(\cM, \tau) 		\stackrel{\tiny \mbox{\cite[Prop.6.1.9]{DPS}
				}
			}{\cong}
			(E(\cM, \tau)_{oc})^{*} 
			\notag \\
			&	
			 \stackrel{\text{\cite[
		Prop.5.3.2]{DPS}}, \eqref{normalKotheduality}}{\cong}  
	 (E(\cM, \tau)_{oc})^\times 
		\end{align}  
		with equality of  norms.

		Let $T = F \otimes x_0$ be a rank-one operator on     $E(\cM, \tau)_{oc}$, that is, 
		$$T(z) = F (z)x_0, \quad z \in E(\cM, \tau)_{oc}, $$  
		where $F\in (E(\cM, \tau)_{oc})^{*} $ and $x_0\in E(\cM, \tau)_{oc}$.    
		There exists 
		$y\in  	 (E(\cM, \tau)_{oc})^\times = E(\cM, \tau)^{\times}  $  such that
		\begin{align}\label{F'=Ftimes}
			F(z)\stackrel{\eqref{normalKotheduality}}{=}
			\tau(zy), \,  \forall z\in E(\cM, \tau)_{oc}
			\quad  \text{and}\quad	\norm{F}_{(E(\cM, \tau)_{oc})^{*}}\stackrel{  \eqref{M-bimodule*=times}}{=}\norm{y}_{E(\cM,  \tau)^{\times}}.
		\end{align}

		%Since there exists  a linear %injective \cite[page 269]{DPS},isometry from $E(\cM, \tau)_{n}^{*}
		%= E(\cM, \tau)_^{\times}$  onto $(E(\cM, \tau)_{oc})^{*}$ , 

        Define a normal	functional $\widetilde{F} \in E(\cM, \tau)^*$ by $\widetilde{F}(z)=\tau(zy) $ for any $z\in E(\cM, \tau)$. Then
        $$\norm{\widetilde{F}}_{E(\cM, \tau)^{*}}=\norm{F}_{(E(\cM, \tau)_{oc})^{*}}
		\stackrel{\eqref{F'=Ftimes}}{=} \norm{y}_{E(\cM,  \tau)^{\times}}.$$
        
		% By \cite[Proposition 5.4.6]{DPS},  
		% there exists a 
		% normal 
		% linear functional $\widetilde{F}$ on $E(\cM, \tau)$ such that 
		% $\widetilde{F}|_{E(\cM, \tau)_{oc}}=F$ and $\norm{\widetilde{F}}_{E(\cM, \tau)^{*}}
		% = 
		% \norm{F}_{(E(\cM, \tau)_{oc})^{*}}
		% \stackrel{\eqref{F'=Ftimes}}{=} \norm{y}_{E(\cM,  \tau)^{\times}}$. Moreover,  $\widetilde{F}(z)=\tau(zy) $ for arbitrary $z\in E(\cM, \tau)$.

		Consider the  rank-one operator 
		\begin{align*}
			S: E(\cM, \tau) &\to  E(\cM, \tau) 
			\\  
			z & \mapsto S(z) := \widetilde{F} (z)x_0.
		\end{align*}
        Note that 
        		\begin{align}\label{Snorm=Tnorm}
			\norm{S}_{E(\cM,  \tau) \to  E(\cM,  \tau)}=\norm{T}_{E(\cM, \tau)_{oc} \to  E(\cM, \tau)_{oc}}.
		\end{align}
		We show that $\|I+S\|_{E(\mathcal{M},\tau) \to E(\mathcal{M},\tau)} \le \|I+T\|_{E(\mathcal{M},\tau)_{oc}\to E(\mathcal{M},\tau)_{oc}}.$ %In fact, we have equality.
        
  %       Then,
		% \begin{align}\label{Snorm=Tnorm}
		% 	\norm{S}_{E(\cM,  \tau) \to  E(\cM,  \tau)} &\nonumber  \stackrel{\text{\cite[p.2388]{KMM07}}}{=}
		% 	\norm{\widetilde{F}}_{E(\cM,  \tau)^{*}} \norm{x_0}_{E(\cM, \tau)}
		% 	\\ &\notag \quad\,\,\,\,\, 	=
		% 	\norm{F}_{(E(\cM, \tau)_{oc})^{*}} \norm{x_0}_{E(\cM, \tau)}
		% 	\\ & \quad\,\,\,\,\,	=
		% 	\norm{T}_{E(\cM, \tau)_{oc} \to  E(\cM, \tau)_{oc}}.
		% \end{align}

		Let $x\in E(\cM, \tau)$  with $\norm{x}_{E(\cM, \tau)}\leq 1$  and  let $x=v|x|$ be its polar decomposition,  where the partial isometry $v\in\cM$. 
		Note that  $E(\cM, \tau)_{oc}$ is an  order dense order ideal of $E(\cM, \tau)$ \cite[Lemma 4.2.15]{DPS}. 
		Hence, there exists 
		$\{x_n\}_{n=1}^\infty  \subseteq  E(\cM, \tau)_{oc}$  
		such that 
		$0\leq x_n \uparrow |x|\in E(\cM, \tau)$ \cite[Lemma 4.2.15]{DPS},
		%In particular,we
		%have   $\mu(x_n) \uparrow \mu(x)$ \cite[Proposition 3.2.14]{DPS},
		and thus, \begin{align}\label{vxnnorm}
			\norm{vx_n}_{E(\cM, \tau)}
			\leq  
			\norm{v}_{\infty}  \norm{x_n}_{E(\cM, \tau)}
			\leq 
			\norm{x_n}_{E(\cM, \tau)} 
			\stackrel{\text{\cite[Prop.4.1.3]{DPS}}}{\leq}
			\norm{x}_{E(\cM, \tau)}\leq 1. 
		\end{align}
		Furthermore,     
		$	x_n \stackrel{\cT_{lm}}{\to}	|x| $ \cite[Proposition 2.7.6(v)]{DPS} 
		and  $	vx_n \stackrel{\cT_{lm}}{\to}	v|x|=x $  as $ n\to\infty$ \cite[Proposition 2.7.5]{DPS}.
		
		Since $y\in E(\cM, \tau)^{\times}$, it follows 
		that 
		$ yv \in E(\cM, \tau)^{\times}$. 
		Let $yv= \text{Re} (yv) + i \text{Im} (yv)= a_1-a_2+i (a_3-a_4)$, where $0 \leq a_i\in E(\cM, \tau)^{\times}$  satisfy $a_1a_2=a_3a_4=0$ \cite[p.241]{DPS}. 
		In particular, 
		$|x|a_i, x_na_i \in L_1(\cM, \tau)$ for all $i$ and $n$ \cite[Proposition 4.3.12]{DPS}.
		
		Since 
		$a_i^{\frac{1}{2}} x_n a_i^{\frac{1}{2}} \uparrow a_i^{\frac{1}{2}}|x| a_i^{\frac{1}{2}}$ \cite[Proposition 2.2.24(iv)]{DPS}, 
		it follows from   the  normality  of the trace $\tau$  that for every $i \in \{1,2,3,4\}$
		\begin{align*}
			\tau\left( x_n a_i\right) 
			\stackrel{\text{\cite[Prop.3.4.32]{DPS}}}{=}
			\tau\left( a_i^{\frac{1}{2}} x_n a_i^{\frac{1}{2}}\right)
			\to
			\tau\left( a_i^{\frac{1}{2}} |x| a_i^{\frac{1}{2}}\right) 
			\stackrel{\text{\cite[Prop.3.4.32]{DPS}}}{=}
			\tau\left(  |x| a_i
			\right) 
		\end{align*}
        as
		$n\to\infty$. 
		The  linearity of  $\tau$ yields that 
		\begin{align}\label{F=wideF}
			F(vx_n)=	\widetilde{F}(vx_n)= 	\tau\left( vx_n y\right) =
			\tau\left( x_n yv\right) 
			\to 
			\tau\left(  |x|yv \right) 
			=\tau\left(  x y
			\right)
			=\widetilde{F}(x)  
		\end{align}  as
		$n\to\infty, $ 
		where we used the fact that $\{ vx_n\}_{n=1}^{\infty}\subseteq  E(\cM, \tau)_{oc}$. 
		Therefore, 
		\begin{align*}
			vx_n+S(vx_n)        \stackrel{\eqref{F=wideF}}{=} 
			vx_n+F(vx_n) x_0 \stackrel{\cT_{lm}}{\to}
			x+ \widetilde{F}(x) x_0=x+	S(x), \quad n\to\infty.  
		\end{align*}
		Since  $\norm{\cdot}_{E(\cM, \tau)}$ is a Fatou norm and $\cM$ is atomless,  it follows from \cite[Theorem 4.4.16  and Remark 4.4.17]{DPS} that  
		\begin{align*}
			\|x+S(x)\|_{E(\cM, \tau)} &	\,\, \leq \mathop{\lim\inf}\limits_{n\to\infty}
			\norm{
				vx_n+\widetilde{F}(vx_n) x_0}_{E(\cM, \tau)}  
			\\ & 
			\, 	\stackrel{\eqref{vxnnorm}}{\leq} 
			\sup_{z\in E(\cM, \tau)_{oc}, \norm{z}_{E(\cM, \tau)_{oc}} \leq 1} \|z+F(z) x_0\|_{E(\cM, \tau)_{oc}}
			\\ &
			\,\,=\norm{I+T}_{E(\cM, \tau)_{oc}\to E(\cM, \tau)_{oc}}, 
		\end{align*}
		which shows that 
		$\norm{I+S}_{E( \cM, \tau)\to E( \cM, \tau)} \leq \norm{I+T}_{E(\cM, \tau)_{oc}\to E(\cM, \tau)_{oc}}$. 
		Since $E(\cM, \tau)$ has the Daugavet property, 
		it follows   that 
		\begin{eqnarray*}
			1+\norm{T}_{E(\cM, \tau)_{oc}\to E(\cM, \tau)_{oc}}   &\stackrel{\eqref{Snorm=Tnorm}}{=}&
			1+\norm{S}_{E( \cM, \tau)\to E( \cM, \tau)}\\
			&\stackrel{\rm \tiny Daugavet~property}{=}&
			\norm{I+S}_{E( \cM, \tau)\to E( \cM, \tau)}
			\\ & \leq& \norm{I+T}_{E(\cM, \tau)_{oc}\to E(\cM, \tau)_{oc}}
			\\ & \leq& 
			1+\norm{T}_{E(\cM, \tau)_{oc}\to E(\cM, \tau)_{oc}},  
		\end{eqnarray*}
		which  completes the proof. 
	\end{proof}
	\begin{remark}
It may happen that $E(\cM, \tau)_{oc}=\{0\}$, e.g., when $E(\cM, \tau)=L_\infty(\mu)$ for a non‑atomic measure $\mu$. 
The condition that $E(\cM, \tau)_{oc}\ne \{0\}$ is equivalent to  $c_{E(\cM, \tau)_{oc}}=\mathbf 1$ when $\cM$ is non-atomic\cite[Proposition 5.4.8]{DPS}. 
	\end{remark}

	\begin{remark}
		Let  $(\cM, \tau)$  be a  semi-finite  atomless  von Neumann algebra   and  let 
		$ \norm{\cdot}_{\cE}$ be a Fatou norm  on   a symmetric operator space $\cE\subseteq  S(\cM, \tau)_h$ of self-adjoint operators. 
		The argument used in the proof of
		Theorem \ref{inheriteDP} yields that 
		if  
		$\cE$ has the Daugavet property and  satisfies 
		$c_{\cE_{oc}}={\bf 1}$,  then $\cE_{oc}$ also has  the Daugavet property.

	\end{remark}

	Recall   
	that for a uniformly monotone real Banach function  lattice $X$ over an atomless  measure space $(\Omega,  \cS, \mu)$\cite[Theorem 4.9]{AKM15}, the Daugavet property is inherited by compressions of the form \( X(A): = \{f \chi_A : f \in X\} \), where $0<\mu(A)\leq \mu(\Omega)$. 
	Our second result (see Theorem \ref{EpDPinherit} below)  provides  a natural noncommutative analogue of  \cite[Theorem 4.9]{AKM15}.  
	We first recall the definition of uniformly monotone norms and some fundamental properties.

    \begin{definition}\label{def:uniformly_monotone}
    \upshape
	Let \((X,\leq, \norm{\cdot}_X)\) be a partially ordered normed linear space.  
	$(X, \norm{\cdot}_X)$  is   \emph{uniformly monotone} whenever for any \(\varepsilon>0\),  there exists \(\delta(\varepsilon)>0\) such that for all \(0\le x,y\in X\),  we have  
	\[
	\|x+y\|_X > 1+\delta(\varepsilon),
	\]
	whenever \(\|y\|_X\ge\varepsilon\) and \(\|x\|_X=1\). 
    \end{definition}
	  Equivalently, \((X,\leq, \norm{\cdot}_X)\)  is uniformly monotone, if for the sequences $\{x_n\}$, $\{y_n\} \subseteq X$ satisfying $0 \le x_n \le y_n$ for all $n \in \mathbb{N}$, the condition $	\mathop{\lim}\limits_n \|x_n\|_X = 	\mathop{\lim}\limits_n \|y_n\|_X = 1$ implies that $\|y_n - x_n\|_X \to 0$ (cf.\cite[page 503]{C12} and \cite{CK17}). %see equivalent definition \cite[above Theorem 4.4]{AKM15}
    This notion (motivated by \cite[page 371]{B67}) seems to appear in \cite[page 344]{S95} for the first time, where it was stated in another equivalent form: 
 for all \(\varepsilon > 0\) there exists \(\delta(\varepsilon) > 0\) such that if \(f, g \in X\), \(f \geq 0\), \(g \geq 0\), \(\|f\|_X = 1\) and \(\|f + g\|_X \leq 1 + \delta(\varepsilon)\),  then \(\|g\|_X \leq \varepsilon\).

	Recall  that for an atomless semi-finite von Neumann algebra $(\cM,\tau)$, 
	a symmetric function  space \((E, \norm{\cdot}_E)\) is uniformly monotone
	if and only if  its noncommutative counterpart  \((E(\cM,\tau), \norm{\cdot}_{E(\cM,\tau)} ) \) is uniformly monotone, see e.g.  \cite[Lemma 2.5 and Corollary 2.11]{C12},  
	\cite[Corollary 17.3]{CK17} and  \cite[Theorem 3.2]{DDS97}. Moreover, if \((E(\cM,\tau), \norm{\cdot}_{E(\cM,\tau)} ) \) is uniformly monotone, then it  is order continuous \cite[Lemma 2.5]{C12}.

	%For uniformly monotone  real Banach function lattice, 	 	equivalent  characterizations and properties    are given in  \cite[Theorem 6]{HKM00} and 	\cite[Proposition 2.2]{L25}.   

	Before proceeding to the proof of Theorem \ref{EpDPinherit}, we need the following estimate, which 
	follows from  \cite[Theorem 1]{Sukochev16} (see also \cite[Lemma 2.4]{DDS97}).
    % Indeed, set $p=q=2$ in \cite[Theorem 1]{Sukochev16} and set $A=eu|x|^{1/2}, B=|x|^{1/2}f$ (both in $E^{1/2}(\mathcal{M},\tau)$), where $x=u|x|.$ So $AB=exf$ and by Holders inequality (thm 1)
    % $$\|exf\|_E \le \|A\|_{E^{1/2}} \|B\|_{E^{1/2}}=(\||A|^2\|_E)^{1/2} (\||B|^2\|_E)^{1/2}=$$
    % $$=(\||x|^{1/2}u^*eu|x|^{1/2}\|_E)^{1/2} (\|f|x|^{1/2}|x|^{1/2}f\|_E)^{1/2} \le \|x\|_E^{1/2} \|f|x|f\|_E^{1/2}.$$
	
	\begin{lemma}\label{Lem2.4}
		Let $(\cM, \tau)$ be a semi-finite von Neumann algebra, and let $(\cE, \norm{\cdot}_{\cE})\subseteq S(\cM, \tau)$ be a symmetric operator  space. 
		If \( x \in \cE \) and  \( p,q \in \cP(\cM) \), 
		then
		\[
		\|p x q\|_{\cE} \leq \|x\|_{\cE}^{\frac{1}{2}} \, \|q |x| q\|_{\cE}^{\frac{1}{2}}.
		\]
	\end{lemma}

	\begin{theorem} \label{EpDPinherit} 
		Let \( (\cE, \norm{\cdot}_{\cE} )\) be a uniformly monotone symmetric operator space  affiliated with a semi-finite atomless von Neumann algebra $(\cM, \tau)$. 
		If \( \cE \) has the Daugavet property,  then \( \cE_p:= \{px p : x \in \cE\} \) also has the Daugavet property for every \( p \in \cP(\cM)\) with \( 0 < \tau(p) \leq \infty \).
	\end{theorem}
	
	\begin{proof}
		Let \( U: \cE_p \to \cE_p \) be a non-zero finite rank operator. 
		% Let  \( V: \cE \to \cE_p \) be the norm-one linear projection defined by 
		% \[
		% V(z) =p z p, \quad z \in \cE.
		% \]
		Define non-zero finite rank operator \( T: \cE \to \cE \) by 
		\begin{align}\label{defT}
			T(z):=U(pzp), \quad z \in \cE.
		\end{align}
		% Observe that, for all $z\in \cE_p$, we have  $T(z)=U(pzp)=U(z) $  and $\norm{z}_{\cE}=\norm{z}_{\cE_{p}}$. 
		% It follows that
		% \begin{align*}
		% 	\norm{U}_{\cE_p \to \cE_p}
		% 	=\sup_{z\in \cE_p,  \norm{z}_{\cE_p}\leq 1 }\norm{U(z)}_{\cE_p}
		% 	=\sup_{z\in \cE_p,  \norm{z}_{\cE}\leq 1 }\norm{T(z)}_{\cE} 
		% 	\leq \|T\|_{\cE \to \cE}.
		% \end{align*}
		% On the other hand, 
		% $\|T\|_{\cE \to \cE} = \|U\circ V\|_{\cE \to \cE}  \leq \|U\|_{\cE_p \to \cE_p}\|V\|_{\cE \to \cE_p}
		% =\|U\|_{\cE_p \to \cE_p}$.
		% Hence,  \( \|T\|_{\cE \to \cE} = \|U\|_{\cE_p \to \cE_p} \).

Since $pzp\in \mathcal{E}_p$, $\|pzp\|_{\mathcal{E}_p}\leq \|z\|_{\mathcal{E}}$, and $pzp=z$ for $z\in\mathcal{E}_p$, we have
$$\begin{aligned}
\|T\|_{\mathcal{E}\to\mathcal{E}} &= \sup_{\substack{z\in \mathcal{E}\\ 
\|z\|_{\mathcal{E}}\le 1}} \|T(z)\|_{\mathcal{E}} = \sup_{\substack{z\in \mathcal{E}\\ \|z\|_{\mathcal{E}}\le 1}} \|U(pzp)\|_{\mathcal{E}_p}\\
&=\sup_{\substack{z\in \mathcal{E}_p\\ \|z\|_{\mathcal{E}_p}\le 1}}
\|U(z)\|_{\mathcal{E}_p}=\|U\|_{\mathcal{E}_p\to\mathcal{E}_p}.
\end{aligned}
$$

		Since  \( \cE  \) has the Daugavet property, it follows that 
		\begin{align}\label{1+T=1+u}
			\|I + T\|_{\cE \to \cE} = 1 + \|T\|_{\cE \to \cE} = 1 + \|U\|_{\cE_p \to \cE_p}, 
		\end{align}
		which implies that  there is a sequence 
		\( \{x_n\}_{n=1}^{\infty} \) in \( S_{\cE} \)
		such that
		\begin{align}\label{xn+txn=1+u} 
			\lim_{n \to \infty} \|x_n + U(p x_n p)\|_{\cE}
			\stackrel{\eqref{defT}}{=} 	\lim_{n \to \infty} \|x_n + T(x_n)\|_{\cE}
			\stackrel{\eqref{1+T=1+u}}{=}
			1 + \|U\|_{\cE_p \to \cE_p}.
		\end{align}
		Since 
		$\norm{x_n}_\cE=1$ for all $n\geq 1$,   
		it follows from 
		the triangle   
		inequality 
		that  
		\begin{align*}
			1 + \|U\|_{\cE_p \to \cE_p}  &\stackrel{\eqref{xn+txn=1+u}}{=} 
			\lim_{n \to \infty} \|x_n + U(p x_n p)\|_{\cE} 
			\le \liminf_{n \to \infty} (1 + \|U\|_{\cE_p \to \cE_p} \|p x_n p \|_{\cE}) \\ &\,\, \le \limsup_{n \to \infty} 
			(1 + \|U\|_{\cE_p \to \cE_p} \| p x_n p\|_{\cE}) \le 
			1 + \|U\|_{\cE_p \to \cE_p}.
		\end{align*}
		It follows from  \( U \neq 0 \)  that $$ \lim_{n \to \infty} \|p x_n^* p\|_{\cE} \stackrel{\text{\cite[Prop.3.2.7]{DPS}}}{=}  
		\lim_{n \to \infty} \|p x_n p\|_{\cE} = 1
		.  
		$$

		By virtue of 
		\[ 	\|px_n p\|_{\cE}^2 
		\stackrel{\text{Lem.\ref{Lem2.4}}}{\leq }
		\|x_n\|_{\cE}   \|p |x_n| p\|_{\cE}
		=\|p |x_n| p\|_{\cE} 
		\leq  \| x_n\|_{\cE}=1, \quad n\in \mathbb{N}, 
		\]
		we immediately infer	
		$$ \lim_{n \to \infty} \|p |x_n| p\|_{\cE} 
		\stackrel{\text{\cite[Prop.3.2.10(ii)]{DPS}}}{=}
		\lim_{n \to \infty} \| |x_n|^{\frac{1}{2}} p |x_n|^{\frac{1}{2}}\|_{\cE}
		= 1.$$ 
		Since for each $n\geq 1$,  $\| x_n \|_{\cE} = 1$ and 
		\[
		0 \le |x_n|^{\frac{1}{2}} p |x_n|^{\frac{1}{2}} \stackrel{\text{\cite[Prop.2.2.24]{DPS}}}{\leq} |x_n|, 
		\] we deduce by \( (\cE, \norm{\cdot}_{\cE} )  \) is  uniformly monotone that \begin{align}\label{xn-xnpxnto0} 
			\||x_n|-|x_n|^{\frac{1}{2}} p |x_n|^{\frac{1}{2}} \|_{\cE}
			=\||x_n|^{\frac{1}{2}} ({\bf 1}-p) |x_n|^{\frac{1}{2}} \|_{\cE} \to 0,  \quad n \to \infty.
		\end{align}
		
		By the 	similar argument, we deduce that   
		% \begin{align*}
		% 	\lim_{n \to \infty} \|p |x_n^*| p\|_{\cE} 
		% 	\stackrel{\text{\cite[Prop.3.2.10(ii)]{DPS}}}{=}\lim_{n \to \infty} \| |x_n^*|^{\frac{1}{2}} p |x_n^*|^{\frac{1}{2}}\|_{\cE}   
		% 	= 1.
		% \end{align*}
		% and 
		\begin{align}\label{xn*-xn*pxn*to0} 
			\||x_n^*|-|x_n^*|^{\frac{1}{2}} p |x_n^*|^{\frac{1}{2}} \|_{\cE}
			=\||x_n^*|^{\frac{1}{2}} ({\bf 1}-p) |x_n^*|^{\frac{1}{2}} \|_{\cE} \to 0,  \quad n \to \infty.
		\end{align}

		We now have 
		\begin{eqnarray*} 
			\| ({\bf 1}-p)
			x_n ({\bf 1}-p)
			\|_{\cE} & 
			\stackrel{\text{Lem.\ref{Lem2.4}}}{\leq }  &
			\| ({\bf 1}-p)
			|x_n| ({\bf 1}-p)
			\|_{\cE}^{\frac{1}{2}}
			\\ & 
			=&	 \| |x_n|^{\frac{1}{2}} ({\bf 1}-p)
			|x_n|^{\frac{1}{2}} 
			\|_{\cE}^{\frac{1}{2}}
			\stackrel{\eqref{xn-xnpxnto0}}{ \to} 0, \quad n \to \infty, 
		\end{eqnarray*}
		\begin{eqnarray*}
			\| p
			x_n ({\bf 1}-p)
			\|_{\cE} & 
			\stackrel{\text{Lem.\ref{Lem2.4}}}{\leq }  &
			\| ({\bf 1}-p)
			|x_n| ({\bf 1}-p)
			\|_{\cE}^{\frac{1}{2}}
			\\ & 
			=&	 \| |x_n|^{\frac{1}{2}} ({\bf 1}-p)
			|x_n|^{\frac{1}{2}} 
			\|_{\cE}^{\frac{1}{2}}
			\stackrel{\eqref{xn-xnpxnto0}}{ \to} 0, \quad n \to \infty,  	
		\end{eqnarray*} 
		and 
		\begin{eqnarray*}
			\| ({\bf 1}-p)
			x_n  p 
			\|_{\cE}
			& \stackrel{\text{\cite[Prop.3.2.7]{DPS}}}{=}   &
			\| p
			x_n^*  ({\bf 1}-p) 
			\|_{\cE} 	\\
			&  \stackrel{\text{Lem.\ref{Lem2.4}}}{\leq } &
			\| |x_n^*|^{\frac{1}{2}} ({\bf 1}-p)
			|x_n^*|^{\frac{1}{2}} 
			\|_{\cE}^{\frac{1}{2}}  \stackrel{\eqref{xn*-xn*pxn*to0}}{ \to}  0,   \quad  n \to \infty.  	
		\end{eqnarray*} 	
		By \begin{align*}
			\|x_n-
			p x_n p \|_{\cE}
			& \leq 
			\|p x_n ({\bf 1}-p)
			\|_{\cE}
			+ 
			\| ({\bf 1}-p)
			x_n  p  \|_{\cE}
			+
			\| ({\bf 1}-p)
			x_n ({\bf 1}-p)
			\|_{\cE},
		\end{align*}
		we have  
		\begin{align}\label{xn-pxnpnormto0}
			\lim\limits_{n\to\infty}\norm{ x_n-
				p x_n p 	}_{\cE}=0.
		\end{align} 
		Thus, 
		\begin{align*}
			1 + \|U\|_{\cE_p \to \cE_p} &\stackrel{\eqref{xn+txn=1+u}}{=}  \lim_{n \to \infty} \|x_n + U(p x_n p )\|_{\cE} 
			\stackrel{\eqref{xn-pxnpnormto0}}{=}   
			\lim_{n \to \infty} \|p x_n p  + U(p x_n p )\|_{\cE}\\ & \,\,  =\lim_{n \to \infty} \|(I+U)(p x_n p ) \|_{\cE}.
		\end{align*}
		Since \( p x_n p\in B_{\cE_p} \) for each \( n \), we have
		\[
		\|I + U\|_{\cE_p \to \cE_p} = 1 + \|U\|_{\cE_p \to \cE_p},
		\]
		which implies that $\cE_p$ has the Daugavet property.
	\end{proof}
	
It is noted in \cite{AKM15} that Conjecture \ref{conj} would be positively resolved if Theorem \ref{EpDPinherit} is valid without the assumption of uniform monotonicity.

	\section{ 
		The Daugavet property on complex symmetric operator spaces}
	\label{s4}

	In this section, we present the full proof of Theorem \ref{Thm 1.2}.

	Recall that \cite[Corollary 2.2]{AKM12}  states that, under additional assumptions,
    %(Fatou norm and $\operatorname{supp}X_a=\Omega$)
    for a real Banach function lattice \(X\) with  a  Fatou norm and the Daugavet property its K\"{o}the dual \(X^{\times}\) contains a (lattice) isomorphic or (lattice)  isometric copy of \(L_1^{\mathbb{R}}(0,1)\). %{\color{blue}(I changed the statement as Corollary 2.2 gives conclusions only about Kothe dual space $X^\times.$)}
	We first   obtain a  noncommutative  version (see Corollary ~\ref{commutativecontainL1} below) of  \cite[Corollary 2.2]{AKM12}.  
	The proof for \cite[Corollary 2.2]{AKM12} relies on the  study of  asymptotically isometric copies of $\ell_1$ in Banach lattices, which is much less explored in the noncommutative context. 
	In what follows, we need to establish  several auxiliary  lemmas. 
	For 
	definitions for  related terminology, see the Appendix \ref{AppendixA}.

	Let $(\cM, \tau)$ be  a  semi-finite  von Neumann algebra and let $(\cE, \norm{\cdot}_{\cE})\subseteq S(\cM, \tau)$ be a symmetric  operator  space.
	We say that a $\norm{\cdot}_{\cE}$-norm closed subspace $X$ of $\cE$ is {\it strongly embedded}  into $\cE$ if the $\norm{\cdot}_{\cE}$-norm topology and the measure topology on $X$ coincide \cite[Definition  2.1]{Ran03}. 
	A sequence $\{x_n\}_{n=1}^\infty$ in $\cE$ is said to be \textit{both left and right disjointly supported} 
	if 
	there exist 
	two sequences  $\{p_n\}_{n=1}^\infty, \{q_n\}_{n=1}^\infty$ of mutually orthogonal projections in $\mathcal{P}(\mathcal{M})$ such that $x_n = p_n x_n q_n$ for all $n$.
	In particular, if $p_n=q_n$ for each $n\geq 1$, 
	then we call  $\{x_n\}_{n=1}^\infty$   \textit{disjointly supported} (see e.g. \cite[page 213]{CDS97} and \cite{Ran03}).

	We first establish the following proposition, which is an extension of \cite[Proposition 3.3]{Ran03}. 
	An important tool used in the proof of the following proposition is 
	\cite[Theorem 2.5]{CDS97} (see also \cite[Lemma 7.1]{DPS16} and \cite[Lemma 7.2]{HNPS24Tran}):  in an order continuous 
	symmetric operator space, any sequence  
	converging to zero in measure  contains an approximately right disjointly supported subsequence, which can be viewed  as a    noncommutative extension  of the classical Kadec--Pe\l czy\'{n}ski lemma concerning  unconditional basic sequences in $L_p(0,  1)$, $ p > 2$\cite{KP62}.
	
A sequence $\{y_n\}_n \subseteq X$ is called a \emph{block basis} of $\{x_n\}_n$ if there exist a strictly increasing sequence of integers $0 = k_0 < k_1 < k_2 < \cdots$ and scalars $\{a_j^{(n)}\}$ such that \cite[Definition 1.3.4]{AK06}
$$y_n = \sum_{j=k_{n-1}+1}^{k_n} a_j^{(n)} x_j, \quad n \ge 1.$$
If in addition $\|y_n\| = 1$ for all $n \ge 1$, then $\{y_n\}_n$ is called a \emph{normalized block basis}.

	\begin{proposition}\label{eitherorProp}
		Let $(\cM, \tau)$  be   a semi-finite von Neumann algebra  
		and let \( E
		\) be a 
		symmetric   function space on $(0, \tau({\bf 1}))$ with an order continuous norm $\norm{\cdot}_{E}$.  
		If \(X\) is a $\norm{\cdot}_{E(\cM, \tau)}$-closed subspace of   \(E(\mathcal{M}, \tau)\), then one of the following holds:   
		\begin{enumerate}[\rm(a)]
			\item \(X\) is strongly embedded into \(E(\mathcal{M}, \tau)\); 
			\item  there exist a normalized basic 
			sequence \(\{y_n\}_{n=1}^{\infty}\) in \(X\) and two mutually orthogonal  sequences  \(\{p_n\}_{n=1}^{\infty}, \{q_n\}_{n=1}^{\infty}
			\subseteq \cM
			\)   
			of $\tau$-finite projections  such that
			\[
			\lim_{n \to \infty} \| y_n - p_n y_n q_n \|_{E(\cM, \tau)}= 0.
			\] 	
				In particular, \(\{y_n\}_{n=1}^{\infty}\) has a subsequence 
                that is  equivalent to   a disjointly supported basic sequence in \(E(\mathcal{M},\tau)
	\). 
	Moreover, if \(X\) has a basis,  then the sequence \(\{y_n\}_{n=1}^{\infty}\) can be chosen to be a block basis of the basis of \(X\).
		\end{enumerate}
	\end{proposition}
	
	\begin{proof}
		Assume that \( X \) is not strongly embedded into \( E(\mathcal{M}, \tau) \).
		Let 
		$$ j : \left( E(\mathcal{M}, \tau), \norm{\cdot}_{E(\cM, \tau)} \right) \rightarrow \left( S(\cM, \tau), \cT_m \right) $$
		be  the natural inclusion. 
		By \cite[Proposition 4.4.4]{DPS}, the 
		mapping $j$  is
		continuous.
		Since \( X \) is not strongly embedded into \( E(\mathcal{M}, \tau) \), the restriction \( j|_X \) fails to be an  isomorphism. Hence, there exists a  sequence \(\{y_n\}_{n=1}^{\infty}\subseteq S_X \) such that  $y_n \stackrel{\cT_m}{\longrightarrow} 0$   as $n\to\infty$.  
		In particular, we have  $y_n^* \stackrel{\cT_m}{\longrightarrow} 0$ \cite[Proposition 2.5.2(viii)]{DPS}.

	% To use TAMS lemma we need strong symmetry, which is available as we already mentioned: 	Note that   
	% \( \cE_{oc}\) is fully  symmetric  when $\cM$ is atomless and $c_{\cE_{oc}}={\bf 1}$  (i.e.,   $\cE_{oc}\neq \{0\}$)  \cite[Theorem 4.5.8 and Proposition  5.4.8]{DPS}.	
		By \cite[Lemma 7.2]{HNPS24Tran}, there exist a subsequence of \(\{y_n\}_{n=1}^{\infty}\) (still denoted  by \(\{y_n\}_{n=1}^{\infty}\) for simplicity) and a sequence  \(\{q_n\}_{n=1}^{\infty}\) of  mutually orthogonal $\tau$-finite projections   %in \( \mathcal{M} \) 
		such that 
		\begin{align}\label{yn-ynqn=0}
			\mathop{\lim}\limits_{n \to \infty} 
			\Big\|y_n -                           y_n q_n
			\Big\|_{E(\cM, \tau)}
			&\notag  =\mathop{\lim}\limits_{n \to \infty} 
			\Big\|\left(y_n -                           y_n q_n \right)^*
			\Big\|_{E(\cM, \tau)}
			\\ &  =
			\mathop{\lim}\limits_{n \to \infty} 
			\Big\| y_n^* -                           q_ny_n^*  
			\Big\|_{E(\cM, \tau)}
			= 0.
		\end{align}  
		It follows from  $y_n^* \stackrel{\cT_m}{\longrightarrow} 0$ 
		%or Propositions  2.4.4 and  2.5.13
		that $q_ny_n^* \stackrel{\cT_m}{\longrightarrow} 0$ 
		as $n\to\infty$ \cite[Proposition 2.5.2(xi)]{DPS}.  
		Applying \cite[Lemma 7.2]{HNPS24Tran} again, there exist a subsequence of \(\{ q_ny_n^*\}_{n=1}^{\infty}\) (still denoted  by \(\{q_n y_n^*\}_{n=1}^{\infty}\)) and a sequence  \(\{p_n\}_{n=1}^{\infty}\) of  mutually orthogonal $\tau$-finite projections   
		such that 
		\begin{align}\label{qnynpn=0}
			\mathop{\lim}\limits_{n \to \infty} 
			\Big\| q_n y_n^* -                           q_ny_n^* p_n 
			\Big\|_{E(\cM, \tau)} 
			= 0.
		\end{align}
		The sequence 
		\(\{p_n y_n q_n\}_{n=1}^{\infty}\)   is   both
		left and right disjointly supported and   
		\begin{align*}%\label{yn-pnynqnnorm0}
			\mathop{\lim}\limits_{n \to \infty} 
			\Big\|  y_n -                           p_ny_n  q_n
			\Big\|_{E(\cM, \tau)}=	\mathop{\lim}\limits_{n \to \infty} 
			\Big\|  y_n^*  -                           q_ny_n^*   p_n
			\Big\|_{E(\cM, \tau)}
			\stackrel{\eqref{yn-ynqn=0}, \eqref{qnynpn=0}}{=}0. 
		\end{align*}
		
By the principle of small perturbations \cite[Theorem 1.3.9]{AK06}, there is a subsequence $\{y_{n_k}\}_{k=1}^{\infty}$  equivalent to $\{p_{n_k}y_{n_k}q_{n_k}\}_{k=1}^{\infty}$ in $E(\cM,\tau).$
Moreover,    	\(\{p_{n_k} y_{n_k} q_{n_k}\}_{k=1}^{\infty}\)  is  isometrically equivalent to  a
	disjointly supported  basic sequence  
	in \(E\) (see \cite[Proposition  2.4]{Ran03} and  
	\cite[Proposition 2.3]{S96}). 
	Thus, \(\{y_n\}_{n=1}^{\infty}\) has a subsequence that is  equivalent to   a disjointly supported basic sequence in \(E
	%(0, \tau({\bf 1}) )
	\).

	The proof of the last assertion follows the same lines as that of \cite[Proposition 3.3]{Ran03}.  \end{proof}  

	As a corollary of Proposition \ref{eitherorProp}, we obtain the following result. First, let us recall the following definition (equivalent definitions are given in the Appendix, see Lemma \ref{equivalentconditionsail1}).
A sequence $\{x_n\}_{n=1}^\infty$ in a  Banach space $(X, \norm{\cdot})$ is said to be an \emph{asymptotically isometric copy of} $\ell_1$ if there is a decreasing to zero sequence $\{\varepsilon_n\}_{n=1}^{\infty}$ of positive numbers in $(0,1)$ such that
$$\sum_{n=1}^{m}(1-\varepsilon_n)|a_n|\leq\left\|\sum_{n=1}^{m}a_nx_n\right\|\leq\sum_{n=1}^{m}|a_n|$$
for all $\{a_n\}_{n=1}^{m}\in \ell_1$, where $m\in \mathbb{N}\cup \{\infty\}$.
In particular, $\{x_n\}_{n=1}^\infty$ is equivalent to the canonical basis of $\ell_1:$
$$(1-\varepsilon_1)\sum_{n=1}^{m}|a_n|\leq\left\|\sum_{n=1}^{m}a_nx_n\right\|\leq\sum_{n=1}^{m}|a_n|,$$
and hence, is a basic sequence for its closed linear span.
	
	\begin{corollary}\label{cor:finite l1} Let   \( (\cE, \norm{\cdot}_{\cE})
		\) be  an order continuous  
		symmetric operator space affiliated with a  semi-finite   von  Neumann algebra $(\cM, \tau)$.
        If 
	    there exists a $\tau$-finite  projection $p\in \cM$   such that  \(\cE p \)  (or $p \cE$) contains   an  asymptotically isometric copy $\{x_k\}_{k=1}^\infty$  of $\ell_1$, 
        then  there
			exist a   normalized block basic sequence 
			$\{y_n\}_{n=1}^\infty$  taken with respect             to $\{x_k\}_{k=1}^\infty$ and two mutually orthogonal  sequences  \(\{g_n\}_{n=1}^{\infty}, \{h_n\}_{n=1}^{\infty}
			\subseteq \cM
			\)   
			of $\tau$-finite projections  such that  
			$\lim\limits_{n\to\infty} \norm{ y_n-g_n y_n h_n }_{\cE}=0$.   In particular, 
       there exists a  subsequence of $\Big\{
		g_n y_n h_n 
		\Big\}_{n=1}^\infty \subseteq \cE p$ (or $p \cE$),  still denoted by $\Big\{
		g_n y_n h_n \Big\}_{n=1}^\infty $,  such that 
        $\Big\{
		\frac{g_n y_n h_n}{\norm{g_n y_n h_n}_{\cE}}
		\Big\}_{n=1}^\infty $   is    a   both
		left and right disjointly supported,   asymptotically isometric  copy of
		\( \ell_1 \).    
	\end{corollary}
	    
	\begin{proof}
		Note that the  sequence   \(\{x_{k}\}_{k=1}^\infty  \)  is  a basis for its $\norm{\cdot}_{\cE}$-norm closed linear span $X\subseteq \cE p $. 
		By	Proposition \ref{eitherorProp}, one of the following holds:   
		\begin{enumerate}
			\item    $X$  is strongly embedded into $\cE
			$. 
			\item   There
			exist a   normalized block basic sequence 
			$\{y_n\}_{n=1}^\infty$  taken with respect             to $\{x_k\}_{k=1}^\infty$ and two mutually orthogonal  sequences  \(\{g_n\}_{n=1}^{\infty}, \{h_n\}_{n=1}^{\infty}
			\subseteq \cM
			\)   
			of $\tau$-finite projections  such that  
			$\lim\limits_{n\to\infty} \norm{ y_n-g_n y_n h_n  }_{\cE}=0$.   
		\end{enumerate}

		Assume that  $X$   is strongly embedded into $\cE
		$.  That is, for any sequence $\{z_n\}_{n=1}^{\infty}$  in $X$, $z_n \stackrel{\cT_m}{\longrightarrow} 0$
		is equivalent to  $z_n \stackrel{\norm{\cdot}_{\cE}}{\longrightarrow} 0$.     %as $n\to\infty$.
		By  ~\eqref{L1tau1phi1}, we know that   $z_n \stackrel{\norm{\cdot}_{\cE}}{\longrightarrow} 0$  implies 
		$z_n \stackrel{\norm{\cdot}_{L_1(\cM, \tau)}}{\longrightarrow} 0$\footnote{Note that for any sequence $\{z_n\}_{n=1}^{\infty}$  in $\cE p$,	$z_n  \stackrel{\norm{\cdot}_{E(\cM, \tau)}}{\longrightarrow} 0$  implies   $\mu(z_n  ) \stackrel{\norm{\cdot}_{E(0,\tau(p)) }}{\longrightarrow} 0$. Therefore,  $\mu(z_n  ) \stackrel{\norm{\cdot}_{L_1(0,\tau(p)) }}{\longrightarrow} 0$ (since $\tau(p)<\infty$), which implies that 
			$z_n \stackrel{\norm{\cdot}_{L_1(\cM, \tau)}}{\longrightarrow} 0$.}.  	%as $n\to\infty$.
		This together with  the continuity of the embedding  $(L_1 (\cM, \tau), \norm{\cdot}_{L_1(\cM, \tau)}) $ into $ (S (\cM, \tau), \cT_m)$ \cite[Proposition 3.4.11]{DPS} yields  
		that 
		the $\norm{\cdot}_{L_1(\cM, \tau)}$-norm topology and the measure topology coincide  on $X$. This means that 
		$X$  is strongly embedded into $L_1(\cM, \tau)
		$, which   
		is equivalent to  
		$X$ not containing   $\ell_1$ \cite[Theorem  4.3]{Ran03}. This directly contradicts the fact  that  $X$ contains  $\ell_1$, since  $X$ is    a   copy of  $\ell_1$. 
		Thus, it suffices to consider 
		case (2).

		%It suffices to consider the case of  asymptotically isometric copies.
		Since   \(\{x_{k}\}_{k=1}^\infty \subseteq \cE p \) is an asymptotically isometric copy of \(\ell_1\), it follows from 
		Lemma \ref{blockbasisasyl1} that  
		the normalized  block basic sequence  
		$\{
		y_n
		\}_{n=1}^\infty$  taken with respect             to $\{x_k\}_{k=1}^\infty$  is an asymptotically isometric  copy of
		\( \ell_1 \).   	    
		This together  Lemma  \ref{perturbationnormlized}    
		implies that    there exists a  subsequence of $\Big\{
		g_n y_n h_n 
		\Big\}_{n=1}^\infty \subseteq \cE p$ 
		(still denoted by $\Big\{
		g_n y_n h_n \Big\}_{n=1}^\infty $ )  such that 
        $\Big\{
		\frac{g_n y_n h_n}{\norm{g_n y_n h_n}_{\cE}}
		\Big\}_{n=1}^\infty \subseteq \cE p$  is    a normalized,   both
		left and right disjointly supported,   asymptotically isometric  copy of
		\( \ell_1 \).
	\end{proof}

\begin{remark}\label{threecopies}
		The asymptotically
		isometric copy of $\ell_1$ is, in particular, an almost isometric copy of $\ell_1$,  which in turn is a copy of $\ell_1$.  
		Thus,	all results in this paper concerning an asymptotically isometric copy of $\ell_1$ (in particular, those contained in the Appendix \ref{AppendixA}) 
        remain valid 
		in the latter two (weaker) settings.
	\end{remark}

	\begin{lemma}\label{doesnot}
		%Let $(\cM, \tau)$ be a semi-finite   von  Neumann algebra and let  \( E\) be  a symmetric function space on  \((0, \tau({\bf 1}) )\) with  an order continuous norm   \(  \norm{\cdot}_{E}\). 
        Let   \( (\cE, \norm{\cdot}_{\cE})
		\) be  an order continuous  
		symmetric operator space affiliated with a  semi-finite   von  Neumann algebra $(\cM, \tau)$.
		Assume that 	$\cE
		$ contains   an  asymptotically isometric copy  \( \{x_{k}\}_{k=1}^\infty    \) of \(\ell_1\) and  let    $p\in \cP(\cM)$ satisfy  $\tau(p)<\infty$.   
	If   
		$\{x_k p \}_{k=1}^{\infty}$  has  no  subsequence  that is  a copy of $\ell_1$,     
		then    there exists a block sequence 
		$\{y_n\}_{n= 1}^{\infty}:=
		\left\{  \mathop{\sum}\limits_{k= r_n}^{s_n} a_k x_k 
		 \right\}_{n= 1}^{\infty}$
		such that $\mathop{\sum}\limits_{k= r_n}
		^{s_n}   \left|a_k\right| =1$ for all $n\geq 1$ and 
			\begin{align*} \lim\limits_{n\to \infty} \norm{y_n p }_{\cE} =0.		\end{align*}  
	  
	\end{lemma}	
    
	\begin{proof}

		Since $(\cE, \norm{\cdot}_{\cE})$   is 
		order continuous,   it follows from  \cite[Theorem 6.1.7  and Lemma  5.5.2]{DPS} that 
		$
		\cE \stackrel{\mbox{\tiny \cite[Lem.5.5.2]{DPS}}}{\subseteq} 
		S_0(\cM, \tau)$,   
		where $S_0(\cM, \tau)$ is the collection  of all 
		$\tau$-compact operators in $S(\cM, \tau)$.  
		Note that   the spectral projection  	 $ p  := e^{\left|x_{1}\right|}
		\left(N^{-1}, N
		\right] \in \cP(\cM ) $  
		satisfy 
		$ \tau(p)\stackrel{
        {\tiny 
        \mbox{
        \cite[Definition 2.4.1]{DPS}
        }
        }
        }{<}\infty.$	 
		
		By the assumption that      the $\norm{\cdot}_{\cE }$-norm bounded sequence   $\left\{x_k p\right\}_{k=2}^{\infty}$   is not  a  copy of $\ell_1$ and   
		Rosenthal's $\ell_1$-theorem \cite[Theorem 10.2.1]{AK06},  there exists a subsequence of 
		$\left\{x_k p\right\}_{k=2}^{\infty}$   (still denoted by $\left\{x_k p\right\}_{k=2}^{\infty}$  for simplicity) such that  $\left\{x_k p\right\}_{k=2}^{\infty}$  is   a  weak Cauchy sequence. 
		Then, 
		$$\left\{z_l\right\}_{l=1}^{\infty}:=\left\{x_{_{2l}} p-
		x_{_{2l+1}} p \right\}_{l=1}^{\infty}$$  is   a  weak null sequence. 
		By Mazur's  theorem \cite[page 344, F.1]{AK06},  
		there exists a block sequence 
		\begin{align*}
			\left\{
			u_n
			\right\}_{n= 1}^{\infty} & :=
			\left\{  \mathop{\sum}\limits_{l= b_n}
			^{c_n} 
			\lambda_l		z_l
			\right\}_{n= 1}^{\infty} 
			=\left\{  \mathop{\sum}\limits_{l= b_n}
			^{c_n} 
			\lambda_l		x_{2l} p-
			\mathop{\sum}\limits_{l= b_n}
			^{c_n} 
			\lambda_l
			x_{2l+1} p 
			\right\}_{n= 1}^{\infty}  
		\end{align*} 	
		such that   $\lambda_l\geq 0$ for each $l\geq 1$, 
		$\mathop{\sum}\limits_{l= b_n}
		^{c_n}  \lambda_l =1$ for each $n\geq 1$	 and \begin{align}\label{5-1-1} \lim\limits_{n\to \infty} \norm{u_n }_{\cE} =0.
		\end{align}  
		Then, 
		$
		\left\{
		w_n
		\right\}_{n= 1}^{\infty}:=
		\left\{  
		\frac{1}{2}  u_n	\right\}_{n= 1}^{\infty} 		
	 $    is a block sequence   taken with respect           to $\left\{  x_k p \right\}_{k=2}^\infty$. 
		For simplicity, assume that 		\begin{align*} 
			\left\{
			w_n
			\right\}_{n= 1}^{\infty}: =		\left\{ \left(  \mathop{\sum}\limits_{k= r_n}^{s_n} a_k x_k \right)p \right\}_{n= 1}^{\infty}.     		\end{align*} 
            Define
            \begin{align*} 
			\left\{
			y_n
			\right\}_{n= 1}^{\infty}: =		\left\{  \mathop{\sum}\limits_{k= r_n}^{s_n} a_k x_k  \right\}_{n= 1}^{\infty}.     		\end{align*}
		Indeed, $r_n:=2b_n$ and $s_n:=2c_n+1$ for every $n\ge 1.$ 
		In particular, 
		$\left\{
	y_n
		\right\}_{n= 1}^{\infty}$ satisfies 
		\begin{align*}
			\mathop{\sum}\limits_{k= r_n}
			^{s_n}   \left|a_k\right|  &= \frac{1}{2} \left(  \lambda _{b_n} + | - \lambda _{b_n} | + \lambda _{b_{n}+1} + | - \lambda _{b_{n}+1} |+  \cdots + \lambda _{c_n} + | - \lambda _{c_n} | \right)\\ & =\sum_{l=b_n}^{c_n} \lambda_l = 1
		\end{align*}
		for all $n\geq 1$ and   
		\begin{align}\label{5-1} \lim\limits_{n\to \infty} \norm{y_n  p }_{\cE} 
        =
        \lim\limits_{n\to \infty} \norm{w_n  }_{\cE} \stackrel{\eqref{5-1-1}}{=}0,		\end{align} 
		%(passing to a subsequence, we may assume 
%\begin{align*}%\label{wnnormvarepsilon}$	\norm{  w_n^{(1)}}_{E(\cM, \tau)} \leq \frac{1}{2},$  $\forall n\geq 1$).
		%\end{align*}
		% {\color{red}
		% 	(Ok, I finally understood this argument (this is the same argument as you provided): we have
		% 	$$w_n:=\frac12 u_n= \sum_{l=b_n}^{c_n} \frac{\lambda_l}{2}(x_{2l}-x_{2l+1})p_1,$$
		% 	where $b_1<c_1<b_2<c_2< \ldots$ and $\sum_{l=b_n}^{c_n}\lambda_l=1, \lambda_l \ge0$ for every $n \ge 1.$ Set $r_n:=2b_n$ and $s_n:=2c_n+1$ for every $n\ge 1.$ Note that $r_1<s_1<r_2<s_2<\ldots.$
		% 	We claim that one can write 
		% 	$$w_n=\left(\sum_{k=r_n}^{s_n} a_kx_k\right)p_1,\quad \sum_{k=r_n}^{s_n} |a_k|=1, \quad n\ge 1.$$
		% 	Indeed, set
		% 	$$a_{r_n+2m}:=\frac{\lambda_{b_n+m}}{2}, \quad a_{r_n+(2m+1)}:=-\frac{\lambda_{b_n+m}}{2}, \quad \text{for any }  m\ge 0 \ \text{with } b_n+m\le c_n, \ n\ge 1.$$
		% 	Then 
		% 	$$\sum_{k=r_n}^{s_n}|a_k|=\sum_{k=b_n}^{c_n}\lambda_k=1, \quad n\ge 1.$$
		% 	)
		% }
		% {\color{green}From Jinghao: To Yerlan: Please, incorporate whatever details you think will help the presentation. I fully trust your ability to explain. YN: Thanks, I will try to do it later this week.}
        which completes the proof
        \end{proof}

	The following theorem is crucial in the proof of Theorem \ref{assumption=Linfty}. 
	Using 	Corollary ~\ref{cor:finite l1},  Remark ~\ref{threecopies} and  Lemma \ref{doesnot}, it  characterizes, 	from the noncommutative  perspective,   when
	a symmetric function space  contains a disjointly supported  asymptotically isometric copy of \(\ell_1\). 
	
			\begin{theorem}\label{dsl1}
		Let $(\cM, \tau)$ be a semi-finite   von  Neumann algebra 
		and let  \( E
		\) be  a 
		symmetric function space on  \((0, \tau({\bf 1}) )\) with  an 
		order continuous norm   \(  \norm{\cdot}_{E}\).
		Then,    
		$E(\cM, \tau)
		$ contains an asymptotically isometric copy of \(\ell_1\) 
		if and
		only if  
		$E(0, \tau({\bf 1}))$ contains a  disjointly supported asymptotically isometric  copy of
		\( \ell_1 \).  
	\end{theorem}
	\begin{proof}
		
		It suffices to prove the   `only if' part.  Assume that 
			\( \{x_{k}\}_{k=1}^\infty  \subseteq E(\cM, \tau)   \)  is  an  asymptotically isometric copy of \(\ell_1\).
		In view of \cite[Proposition  2.4]{Ran03} and  
		\cite[Proposition 2.3]{S96},  
		the problem reduces to showing that  \(E(\cM, \tau)\) contains   an asymptotically isometric  copy of
		\( \ell_1 \), which is  disjointly supported from both the left and the right.

		We first construct  an asymptotically isometric copy $\{z_n\}_{n= 1}^{\infty}$  of $\ell_1$ disjointly supported from the right.  There are  two possible cases.
		
		{\bf Case 1: 
			There exists a $\tau$-finite  projection $h\in \cM$   such that   \( E(\cM, \tau) h \) %(or $p E(\cM, \tau)  $)  
			contains  an asymptotically isometric copy   of \(\ell_1\). 
		}  
		
		The construction of $\{z_n\}_{n= 1}^{\infty}$  is an immediate consequence of  Corollary \ref{cor:finite l1}.

		{\bf Case 2:  For any   $\tau$-finite  projection $h\in \cM$,  %neither  
			$ E(\cM, \tau) h $ 
			%nor  $p E(\cM, \tau)  $ 
			does not   contain    an  asymptotically isometric copy of \(\ell_1\).}

		%For an arbitrary $p\in \cP(\cM)$ with $\tau(p)<\infty$ and  a   copy $\{z_k\}_{k=1}^{\infty}$ of \( \ell_1 \) in \(E(\cM, \tau)\),  	define $$W  := \overline{\mathrm{span}\left\{ 	z_k  	\right\}_{k=  1}^{\infty} }^{\|\cdot\|_{E(\cM, \tau)}}  p  \subseteq  E(\cM, \tau) p.$$	
		%{\color{red} Then, $\left\{ 	z_k p 	\right\}_{k=  1}^{\infty} $  is a basis for $W$. }
		%	There are two possibilities: 
		%      \begin{itemize}
			%          \item $W $  contains  a  copy  of $\ell_1$;
			%          \item $W $ does not  contain any copy of $\ell_1$.  
			%       \end{itemize}
		
		Before the construction under the assumption that Case 2, it should be noted that 	for a  given copy $\{\xi_k\}_{k=1}^{\infty}$ of
		\( \ell_1 \) in \(E(\cM, \tau)\) and a given projection   $h\in \cP(\cM)$ with $\tau(h)<\infty$,   there are two possible cases: 
		\begin{enumerate}
			\item  There exists a subsequence  of  $\{\xi_k  h \}_{k=1}^{\infty}$ which is a copy of $\ell_1$, still denoted by  $\{\xi_k  h\}_{k=1}^{\infty}$.  
			\item There exists no  subsequence  of  $\{\xi_k h \}_{k=1}^{\infty}$ which is a copy of $\ell_1$. 
		\end{enumerate}

			For simplicity, we present the construction for a specific case
			to illustrate the general inductive construction.

			By \cite[Theorem 5.5.12]{DPS},  we can choose \( N_1 \)   large   enough  and define 	 $ p_1  := e^{\left|x_{1} \right|}
			\left(N_1^{-1}, N_1
			\right]  $  
			so that 
			\begin{align} \label{0-51}
				\tau(p_1)\stackrel{\tiny \mbox{\cite[Definition 2.4.1]{DPS}}}{<}\infty \quad 
				\text{and}\quad \left\|
				x_{1}- x_{1}  p_1  \right\|_{E(\cM, \tau)} 
				\leq  \frac{1}{2^2}.  \end{align}

			{\bf Step $1$.}  Assume that  there exists a subsequence  of  $\{x_k p_1\}_{k=2}^{\infty}=\left\{ x_k  p_1 \right\}_{k=1}^{\infty} \setminus \{x_1 p_1\}$ which is a copy of $\ell_1$, still denoted by  $\{x_k p_1\}_{k=2}^{\infty}$.

			By   Corollary ~\ref{cor:finite l1} 
            and Remark ~\ref{threecopies},   
			there exist  a   sequence  \( \left\{f_n^{(1)}\right\}_{n=1}^{\infty}
			\subseteq \cM
			\)   
			of  mutually orthogonal  $\tau$-finite projections 
			and  a normalized block basis $ \left\{ \gamma_n^{(1)}  \right\}_{n=1}^{\infty}
			:= 	\left\{ \left( 
			\sum\limits_{k= r_n^{(1)}}^{s_n^{(1)}}  a_{k}^{(1)} x_k \right) 
			p_1   
			\right\}_{n=1}^{\infty}$
			of $\left\{ x_k p_1 \right\}_{k=2}^{\infty}
			$    such that    
			$\left\{ \gamma_n^{(1)}  f_n^{(1)}   \right\}_{n=1}^{\infty}	
			$ 
			is  a copy of $\ell_1$  disjointly supported from the right 
			and 
			\begin{align}\label{0-1}\Big\|\zeta_n^{(1)} p_1 -                           \zeta_n^{(1)} f_n^{(1)}
				\Big\|_{E(\cM, \tau)} =	
				\Big\|\gamma_n^{(1)} -                           \gamma_n^{(1)} f_n^{(1)}
				\Big\|_{E(\cM, \tau)} 
			\leq 
				\frac{1}{2^{n+2}},  
			\end{align}  
			where 	$\left\{\zeta_n^{(1)}
			\right\}_{n=1}^{\infty}:= \left\{ 	\sum\limits_{k=r_n^{(1)} \ge 2 }^{s_n^{(1)}}  a_{k}^{(1)} 
			x_k 
			\right\}_{n= 1}^{\infty}  
			$ is a block sequence taken with respect to $\left\{ x_k \right\}_{k=2}^{\infty} = \left\{ x_k \right\}_{k=1}^{\infty} \setminus \{x_1\}
			$   and $f_n^{(1)} \leq p_1$ for all $n\geq 1$.

			Since $\tau(p_1)<\infty$, 
			it follows that $\tau(f_n^{(1)}) \to 0$ as $n\to \infty$.
			%, and thus, $\norm{ f_n^{(1)}}_{E(\cM, \tau)}\to 0$ as $n\to  \infty$ \cite[Lemma 7.1]{HNPS24Tran}. 
			Without loss of generality, we may assume that 
			\begin{align}\label{0-0} \sum\limits_{k=r_n^{(1)}}^{s_n^{(1)}}  |a_{k}^{(1)}| =1 	\end{align}  
			and 
			\begin{align}\label{0-11}\Big\|\zeta_n^{(1)} p_1 -                           \zeta_n^{(1)} f_n^{(1)}
				\Big\|_{E(\cM, \tau)} 
                %=	\Big\|\gamma_n^{(1)} -                           \gamma_n^{(1)} f_n^{(1)}	\Big\|_{E(\cM, \tau)} 
				\leq 
			\frac{1}{2^{n+1}}, \quad \forall n\geq 1.    
			\end{align}      
			Indeed, 
			since 	\( \left\{x_{k}\right\}_{k=1}^\infty  \subseteq E(\cM, \tau)   \)  is  an  asymptotically isometric copy of \(\ell_1\), 
			it follows that  there exists a null sequence \(\{\varepsilon_k\}_{k=1}^{\infty}\) of positive numbers in \((0,1)\)   such that for each  $n\geq 1$,  
			\begin{align*}
				1 =	\norm{\gamma_n^{(1)}}_{E(\cM, \tau) }
				\leq \norm{ \zeta_n^{(1)}}_{E(\cM, \tau) }& \, \,  \quad=
				\norm{	\sum\limits_{k= r_n^{(1)}}^{s_n^{(1)}}  a_{k}^{(1)} x_k }_{E(\cM, \tau) } \\ & \stackrel{\text{Lem.\ref{equivalentconditionsail1}(iii)}}{\leq} 
				\left(
				1+\varepsilon_{r_n^{(1)}}
				\right) \sum\limits_{k= r_n^{(1)}}^{s_n^{(1)}}  |a_{k}^{(1)}|. 
			\end{align*} 
			Define	$ t_n:= \sum\limits_{k= r_n^{(1)}}^{s_n^{(1)}}  |a_{k}^{(1)}|$. 
		Passing to a subsequence of $\left\{ \zeta_n^{(1)} \right\}_{n=1}^{\infty} $ if necessary (still denoted by $\left\{ \zeta_n^{(1)} \right\}_{n=1}^{\infty}  $), we may assume that  $t_n>0$ \footnote{ Note that $t_n\not\to 0$ as $ n\to \infty$.
		Because
		$\norm{\zeta_n^{(1)}}_{E(\cM,\tau)}\le \norm{\sum\limits_{k=r_n^{(1)}  }^{s_n^{(1)}}  a_{k}^{(1)} 
			x_k 
		}_{E(\cM,\tau)} \le t_n \cdot  \max\limits_{r_n^{(1)}  \leq k\leq s_n^{(1)}  }\norm{x_k}_{E(\cM,\tau)} $.     }  for all $n\geq 1$.  
		  Then, \begin{align}\label{0-012}
			 \frac{1}{t_n} \leq  1+\varepsilon_{r_n^{(1)}}< 2, \quad \forall n\geq 1.
			\end{align}   
			Moreover, the right-disjointly supported  sequence  	$\left\{\frac{1}{t_n}   \gamma_n^{(1)} f_n^{(1)}\right\}_{n=1}^{\infty}	
			$ 
			satisfies   
			\begin{align*}	
				\Big\| \frac{1}{t_n}\zeta_n^{(1)} p_1 -                          \frac{1}{t_n} \zeta_n^{(1)} f_n^{(1)}
				\Big\|_{E(\cM, \tau)} 
				% =  \frac{1}{t_n} 	\Big\| \gamma_n^{(1)}  -                           \gamma_n^{(1)} f_n^{(1)}\Big\|_{E(\cM, \tau)}
		\stackrel{\eqref{0-1}}{\leq}  	\frac{1}{2^{n+2} t_n} 
			\stackrel{\eqref{0-012}}{<}
				  \frac{1}{2^{n+1}}    
			\end{align*} 
	and 
			$\sum\limits_{k=r_n^{(1)}}^{s_n^{(1)}}  |\frac{a_{k}^{(1)}}{t_n} |=1 $	  for all $n\geq 1$. 
			
			Without loss of generality,  by \eqref{0-0} and Lemma \ref{blocksum=1asyl1}, we may assume that   $ \left\{\zeta_n^{(1)}\right\}_{n=1}^{\infty}
			$   is an  
			asymptotically isometric  copy of
			\( \ell_1 \). 
			Since $\left\{ f_n^{(1)}\right\}_{n=1}^{\infty}$  is a sequence of  mutually orthogonal projections,  it follows from  \cite[Theorem 5.5.12 or 5.5.13]{DPS} that  \begin{align*}
				\Big\|
				x_1  	f_n^{(1)}   
				\Big\|_{E(\cM, \tau)}   \leq \frac{1}{2^{n+1}}, \quad \forall n\geq 1 . 
			\end{align*} 
	Then,  
		\begin{align}\label{0-002}
			 \sum_{n =2 } ^{\infty} \Big\|
			x_1 	f_n^{(1)}   
			\Big\|_{E(\cM, \tau)} \leq   \sum_{n =2 } ^{\infty}  
            \frac{1}{2^{n+1}}=\frac{1}{2^2}.  
		\end{align}

			Choose 
			$C$  large enough such that  $q :=  e^{|\zeta_2^{(1)}({\bf 1}-p_1)|} (C^{-1}, C] \leq {\bf 1}-p_1 $ satisfies  \begin{align}\label{0-03}
				\tau(q)<\infty \quad 
				\text{and}\quad  	\left\|
				\zeta_2^{(1)}  ({\bf 1}-p_1) - \zeta_2^{(1)} q \right\|_{E(\cM, \tau)} \leq  \frac{1}{2^3}.
			\end{align}
			
			{\bf Step 2}. 
			Assume that there exists no  subsequence  of  $\left\{\zeta_k^{(1)} q  \right\}_{k=3}^{\infty}$ which is a copy of $\ell_1$.

			By 	Lemma ~\ref{doesnot},  
			there exists a block sequence 
			$\{	v_n \}_{n= 1}^{\infty}:=
			\left\{  \mathop{\sum}\limits_{k= t_n\geq 3}^{h_n} b_{k} \zeta_k^{(1)} 
			\right\}_{n= 1}^{\infty}$  
			such that $\mathop{\sum}\limits_{k= t_n}
			^{h_n}   \left|b_{k} \right| =1$ for all $n\geq 1$ and (setting $w_n := v_n q $)
			\begin{align}\label{5-1} \lim\limits_{n\to \infty} \norm{w_n  }_{E(\cM, \tau)}=
				\lim\limits_{n\to \infty} \norm{v_n  q }_{E(\cM, \tau)} =0.		\end{align} 
			Moreover, $\{	v_n \}_{n= 1}^{\infty}$ 
			is  an asymptotically isometric  copy of
			\( \ell_1 \), see  Lemma ~\ref{blocksum=1asyl1}. 
			
			Since 
			$$
			\lim\limits_{n\to \infty} \norm{v_n   -
				\left( 
				v_n  -w_n 
				\right)
			}_{E(\cM, \tau)} 
			=	\lim\limits_{n\to \infty} \norm{w_n }_{E(\cM, \tau)}	\stackrel{\eqref{5-1}}{=}0
			, 
			$$
			it follows from   Lemma  \ref{perturbationnormlized}
			that  there exist
			a  subsequence of $\left\{v_n-w_n\right\}_{n=1}^\infty$, still denoted by $\left\{v_n-w_n\right\}_{n=1}^\infty$,  	 and 
			a sequence $ \left\{ 
			\lambda_n \right\}_{n=1}^\infty $  of  positive  scalars with  $\lim\limits_{n\to \infty} 	\lambda_n=1$  such that  
			\begin{align}\label{1-01}
				\left\{\zeta_n^{(2)}\right\}_{n=1}^\infty
				& \notag :=\left\{ 
				\lambda_n
				\left(   v_n -w_n \right) \right\}_{n=1}^\infty
				\\ & 	\stackrel{\eqref{5-1}}{=} \left\{\lambda_n
				\left( \mathop{\sum}\limits_{k= t_n}^{h_n} b_{k} \zeta_k^{(1)}  
				\right)
				({\bf 1}-q)    \right\}_{n=1}^\infty 
			\end{align}
			is   an 
			asymptotically isometric  copy of
			\( \ell_1 \).  
			%Passing to a subsequence if necessary, we may  assume that   $ \frac{1}{2} = 1-\frac{1}{2}\leq \lambda_n^{(1)} \leq  1, $ for all $ n\geq 1.$ 

			Choose 
			$N_2$  large enough such that  $p_2 :=  e^{|\zeta_2^{(2)}({\bf 1}-p_1-q)|} (N_2^{-1},N_2] \leq {\bf 1}-p_1-q $ satisfies  \begin{align}\label{0-032}
				\tau(p_2)<\infty \quad 
				\text{and}\quad  	\left\|
				\zeta_2^{(2)}  ({\bf 1}-p_1-q) - \zeta_2^{(2)} p_2  \right\|_{E(\cM, \tau)} \leq  \frac{1}{2^4}.
			\end{align}
			
			{\bf Step 3.} 
			Assume that there exists a subsequence  of  $\left\{\zeta_k^{(2)} p_2\right\}_{k=3}^{\infty}$ which is a copy of $\ell_1$, still denoted by  $\left\{\zeta_k^{(2)} p_2\right\}_{k=3}^{\infty}$.

			There exist  a mutually orthogonal  sequence  \( \left\{f_n^{(2)}\right\}_{n=1}^{\infty}
			\subseteq \cM
			\)   
			of $\tau$-finite projections 
			and a    block sequence  $\left\{\gamma_n^{(2)}\right\}_{n=1}^{\infty} 
			:= 	\left\{ \left( 
			\sum\limits_{k= r_n^{(2)}}^{s_n^{(2)}  }  a_{k}^{(2)} \zeta_k^{(2)}  \right) 
			p_2   
			\right\}_{n=1}^{\infty} 
			$  
			such that 
			$\left\{\gamma_n^{(2)}f_n^{(2)}\right\}_{n=1}^{\infty} 
			$  is a copy   of $\ell_1$  disjointly supported from the  right and 
			\begin{align*}%\label{0-00}
				\Big\|\zeta_n^{(3)} p_2 -                           \zeta_n^{(3)} f_n^{(2)}
				\Big\|_{E(\cM, \tau)} 
				=	\Big\|\gamma_n^{(2)} -                           \gamma_n^{(2)} f_n^{(2)}
				\Big\|_{E(\cM, \tau)} 
				\to 0  \,\,\text{as}\,\,\, n\to \infty,  
			\end{align*} where   $\left\{\zeta_n^{(3)}\right\}_{n=1}^{\infty}:= \left\{ 	\sum\limits_{k=r_n^{(2)}}^{s_n^{(2)}}  a_{k}^{(2)} 
			\zeta_k^{(2)}  
			\right\}_{n= 1}^{\infty}   
			$ 	and     $f_n^{(2)} \leq p_2 \leq {\bf 1}-p_1-q $  for all $n\geq 1$.  
			Moreover,   $	\sum\limits_{k= r_n^{(2)}}^{s_n^{(2)}}  |a_{k}^{(2)}|=1$ for all $n\geq 1$  
			and $\tau(f_n^{(2)}) \to 0$ as $n\to \infty$. 
			Note that  \begin{align*}
				\Big\|
				\zeta_2^{(2)}  	f_n^{(2)}   
				\Big\|_{E(\cM, \tau)}  \leq \frac{1}{2^{n+3}}, \quad \forall n\geq 1 , 
			\end{align*}     see \cite[Theorem 5.5.12 or 5.5.13]{DPS}.  
			It follows that \begin{align}\label{0-003}
				\sum_{n =2 } ^{\infty} \Big\|
			\zeta_2^{(2)}  	f_n^{(2)}    
				\Big\|_{E(\cM, \tau)}  \leq \frac{1}{2^4}.  
			\end{align}

			By  Lemma \ref{blocksum=1asyl1}, $ \left\{\zeta_n^{(3)}\right\}_{n=1}^{\infty} 
			$   is an  
			asymptotically isometric  copy of
			\( \ell_1 \). 	
			Choose 
			$N_3$  large enough such that  $p_3 :=  e^{|\zeta_2^{(3)}({\bf 1}-p_1-q-p_2)|} (N_3^{-1},N_3] \leq {\bf 1}-p_1-q-p_2 $ satisfies  \begin{align*}%\label{0-01}
				\tau(p_3)<\infty \quad 
				\text{and}\quad  	\left\|
				\zeta_2^{(3)}  ({\bf 1}-p_1-q-p_2) - \zeta_2^{(3)} p_3  \right\|_{E(\cM, \tau)} \leq  \frac{1}{2^4}.
			\end{align*}

			{\bf Step 4.}
			Assume that there exists a subsequence  of  $\left\{\zeta_k^{(3)}  p_3\right\}_{k=3}^{\infty}$ which is a copy of $\ell_1$, still denoted by  $\left\{\zeta_k^{(3)}  p_3\right\}_{k=3}^{\infty}$ .

			There exist a   mutually orthogonal  sequence  \( \left\{f_n^{(3)}\right\}_{n=1}^{\infty}
			\subseteq \cM
			\)   
			of $\tau$-finite projections 
			and a    block sequence  $\left\{\gamma_n^{(3)}\right\}_{n=1}^{\infty} 
			:= 	\left\{ \left( 
			\sum\limits_{k= r_n^{(3)}}^{s_n^{(3)}}  a_{k}^{(3)} \zeta_k^{(3)}  \right) 
			p_3   
			\right\}_{n=1}^{\infty} 
			$  
			such that 
			$\left\{\gamma_n^{(3)}f_n^{(3)}\right\}_{n=1}^{\infty} 
			$  is a copy   of $\ell_1$  disjointly supported from the  right and 
			\begin{align*}
				\Big\|\zeta_n^{(4)} p_3 -                           \zeta_n^{(4)} f_n^{(3)}
				\Big\|_{E(\cM, \tau)} 
				=	\Big\|\gamma_n^{(3)} -                           \gamma_n^{(3)} f_n^{(3)}
				\Big\|_{E(\cM, \tau)} 
				\to 0  \,\,\text{as}\,\,\,\, n\to \infty,  
			\end{align*}
			where   $\left\{\zeta_n^{(4)}\right\}_{n=1}^{\infty}:= \left\{ 	\sum\limits_{k=r_n^{(3)}}^{s_n^{(3)}}  a_{k}^{(3)} 
			\zeta_k^{(3)}  
			\right\}_{n= 1}^{\infty}   
			$	and     $f_n^{(3)} \leq p_3$  for all $n\geq 1$.  
			Moreover,    $	\sum\limits_{k= r_n^{(3)}}^{s_n^{(3)}}  |a_{k}^{(3)}|=1$ for all $n\geq 1$  and 
			$\tau(f_n^{(3)}) \to 0$  as $n\to \infty$.

			{\bf Step 5.}
			Define \begin{align*}
				y_1 &  :=x_1 
				\left( p_1 - 
				\mathop{\sum}\limits_{k= 2}^{\infty }  f_k^{(1)} 
				\right), \,   
				\\ 	y_2 & :=\zeta_2^{(1)}  f_2^{(1)} + \zeta_2^{(1)} q, \quad 
				\\ 	y_3 & := \lambda_2 	 \mathop{\sum}\limits_{k= t_2}^{h_2} b_{k} \left( 
				\zeta_k^{(1)} f_k^{(1)}   
				\right)+
				\zeta_2^{(2)} 
				\left( 	p_2   	- 
				\sum_{k=2}^\infty  f_k^{(2)} \right)
				,  
				\\ y_4 & := 
				\mathop{\sum}\limits_{k= r_2^{(2)}}^{s_2^{(2)}} a_{k}^{(2)}
				\lambda_k 	 \mathop{\sum}\limits_{m= t_k}^{h_k} b_{m} \left( 
				\zeta_m^{(1)} f_m^{(1)}   
				\right)
				+
				\zeta_2^{(3)} f_2^{(2)}
				+
				\zeta_2^{(3)} \left(
				p_3- 
				\sum_{k=2}^\infty  f_k^{(3)}  \right).   
			\end{align*} 	Observe  that, 
			\begin{align*}%\label{0-07}
	\norm{x_1-y_1}_{E(\cM,\tau)}& \quad 	=	 \notag \left\|x_1 - 
		x_1 \left( 
		p_1 -	 \mathop{\sum}\limits_{k= 2}^{\infty }  f_k^{(1)} 
		\right)
		\right\|_{E(\cM, \tau)} 
		\\  \notag  &  \quad  \leq 
		\Big\|x_1 -
				x_1  p_1 \Big\|_{E(\cM, \tau)} +
				\mathop{\sum}\limits_{k=2}^{\infty} 
				\Big\|
				x_1   f_k^{(1)}  
				\Big\|_{E(\cM, \tau)} 
		\\  \notag  &  	\stackrel{\eqref{0-51}, \eqref{0-002}}{\leq } \frac{1}{2^2} +\frac{1}{2^2}=\frac{1}{2}
				,  
			\end{align*} 
			\begin{align*}
		\norm{\zeta_2^{(1)} - y_2}_{E(\cM,\tau)} 	& \notag  		\quad \,   =\left\|\zeta_2^{(1)} - 
		\left( \zeta_2^{(1)} 	f_2^{(1)}   
				+
				\zeta_2^{(1)}   
				q 
				\right)
				\right\|_{E(\cM, \tau)}
				\\ & \notag
				\quad \, 	\leq  \left\|\zeta_2^{(1)}p_1 - 
				\zeta_2^{(1)} 	f_2^{(1)}   
				\right\|_{E(\cM, \tau)}
				+
				\left\|\zeta_2^{(1)} ({\bf 1}-p_1)- 
				\zeta_2^{(1)} q 
				\right\|_{E(\cM, \tau)}
				\\ & \notag  
				\stackrel{\eqref{0-11},  \eqref{0-03}}{\leq} 	\frac{1}{2^{3}} 	+ 
				\frac{1}{2^{3}}  =
					\frac{1}{2^{2}},  
			\end{align*}	
            and 
			\begin{align*}
	& \notag  	\qquad \quad \norm{\zeta_2^{(2)} - y_3}_{E(\cM,\tau)} 	 
				\\ & \qquad 	  = \left\|\zeta_2^{(2)} - 
				\left( 
				\lambda_2 	 \mathop{\sum}\limits_{k= t_2}^{h_2} b_{k} \left( 
				\zeta_k^{(1)} f_k^{(1)}   
				\right) 
				+
				\zeta_2^{(2)}   
				\left( 	p_2   	- 
			\sum_{k=2}^\infty  f_k^{(2)} \right) 
				\right)
				\right\|_{E(\cM, \tau)} 
				\\ & 	\qquad \notag  
		\leq
				\left\|\zeta_2^{(2)} 
				p_1 - 
				\lambda_2 	 \mathop{\sum}\limits_{k= t_2}^{h_2} b_{k} \left( 
				\zeta_k^{(1)} f_k^{(1)}   
				\right) 
				\right\|_{E(\cM, \tau)} 
				+	\left\|\zeta_2^{(2)} 
				q
				\right\|_{E(\cM, \tau)}
				\\ &  \quad   \qquad 	+
				\left\|
				\zeta_2^{(2)} 
				({\bf 1}-p_1-q) - 
				\zeta_2^{(2)} 
		 	p_2   	
				\right\|_{E(\cM, \tau)}
				+
					\sum_{k=2}^\infty
			\left\|   
			 	\zeta_2^{(2)}  f_k^{(2)}  
			 	\right\|_{E(\cM, \tau)}
				\\ &  
		\stackrel{
		\eqref{1-01}, \eqref{0-032} , \eqref{0-003}}{\leq }
		\lambda_2 
		\mathop{\sum}\limits_{k= t_2}^{h_2}| b_{k}|
		\left\|\zeta_k^{(1)} 	p_1 - 
		\zeta_k^{(1)} f_k^{(1)}   
	\right\|_{E(\cM, \tau)} 
			+	
		\frac{1}{2^4}+	
		\frac{1}{2^4}
		\\ & 
		\quad \,\,\,\, 
		\stackrel{	\eqref{0-11}}{\leq }	\lambda_2 
	\mathop{\sum}\limits_{k= t_2}^{h_2}
		\frac{| b_{k}|}{2^{ k+1}} 
		+	
		\frac{1}{2^3}
			\leq 
				\frac{\lambda_2}{2^{ t_2+1}} 
			+	
			\frac{1}{2^3}	
			\\ &  	\quad \,\,\,\, \,  	\leq  
			\frac{\lambda_2}{2^{4}} 
		+	
		\frac{1}{2^4}+	
		\frac{1}{2^3}	 \quad 
		(\text{since} 
		\,\,\,\,  t_2> t_1\geq 3).   
			\end{align*}

			Arguing inductively, 
			we can construct a sequence $\{y_n\}_{n= 1}^{\infty}$,  which is disjointly supported from the right.	
			By   Lemmas  \ref{perturbationnormlized} and  \ref{blocksum=1asyl1}, %$ \left\{ x_1, \zeta_2^{(1)}, \zeta_2^{(2)},  \cdots, \zeta_2^{(n)}, \cdots\right\} $   is an  asymptotically isometric  copy of \( \ell_1 \). 
			there exists 
			a sequence $ \left\{ 
			\lambda_n \right\}_{n=1}^\infty $  of  positive  scalars with  $\lim\limits_{n\to \infty} 	\lambda_n=1$  such that  $\{z_n\}_{n= 1}^{\infty}:=\{\lambda_n y_n\}_{n= 1}^{\infty}$ is an  
			asymptotically isometric  copy of
			\( \ell_1 \),  which is disjointly supported from the right.	
	Thus, the construction under the Case 2 assumption is finished.

	Applying the same procedure to the left supports of the  right-disjointly supported  asymptotically isometric  copy   $\{z_n\}_{n= 1}^{\infty}$   of
	\( \ell_1 \) in $E(\cM,\tau)$,  one can construct  an asymptotically isometric  copy    of
	\( \ell_1 \) in  $E(\cM,\tau)$,    which is disjointly supported from  the left and the right. 
	\end{proof}

	We say that  a Banach function lattice  \( (X,  \norm{\cdot}) \)  is {\it strictly monotone} if for any $x_1, x_2 \in X$ satisfying 
	$ 
	0 \leq x_1 \leq  x_2 
	$ and $x_1\neq  x_2$, 
	we have 
	$
	\|x_1\| < \|x_2\|
	$\cite{W03}. 
	%The space $L_1$ is a typical example of such a  space\cite{W03}. 
	% Other examples, in the setting of function spaces, are given in other works by   Kami\'{n}ska.

	An operator $T: X\to Y$ between two lattices is said to be a \textit{lattice homomorphism} (\textit{lattice isomorphism}, respectively) whenever it preserves the lattice operations (and, additionally, $T$ is an injection, respectively), i.e.,   
	$
	T(x_1 \lor x_2) = T(x_1) \lor T(x_2)
	$  and $T(x_1\wedge x_2)=T(x_1)\wedge T(x_2)$ 
	for all $x_1, x_2 \in X$;
	if, additionally, $T$ is an isometry then we say that $Y$ contains a \textit{lattice isometric  copy} of $X$ (via $T$) \cite{W03}.

	\begin{corollary}\label{commutativecontainL1}
		Let $(\cM,  \tau)$ be a semi-finite atomless von  Neumann algebra and let   \(E 
		\) be a  
		symmetric function space on  \((0, \tau({\bf 1}) )\) with a	 Fatou norm  $\norm{\cdot}_{E}$. 
		If 	$c_{E_{oc}}
		%=c_{E(\cM, \tau)_{oc}}
		={\bf 1}$ and  
		$E(\cM, \tau)$ has the Daugavet property,  
		then  the following holds:
		\begin{enumerate}[\rm(i)]
			\item   $  E
			^{\times}$  contains an  isometric copy of \(L_1(0,1)\).  
			\item 
			If $ ( E
			^{\times} )_h:=\{  \phi\in E
			^{\times}: \phi=\overline{\phi} \}
			$ is strictly monotone, then $ ( E
			^{\times} )_h$  contains a lattice isometric copy of $L_1^{\mathbb{R}}(0,1)$.  
			
			\item If $ ( E
			^{\times} )_h$ is  KB, then $ ( E
			^{\times} )_h$ contains a lattice isomorphic copy of $L_1^{\mathbb{R}}(0, 1)$.  
		\end{enumerate} 
		
	\end{corollary} 
	
	\begin{proof}

		(i) 
		By Theorem \ref{inheriteDP},  the Banach space \( E(\cM, \tau)_{oc} \) has the Daugavet property. 
		If 
		a real Banach space has the Daugavet property, then it contains an asymptotically isometric copy of real \(\ell_1\)  \cite[Theorem 2.9]{KSSW00}  (this  result can be  extended to the complex case with minor modifications \cite[page 3]{KSSW00}). 
		Thus, \( E(\cM, \tau)_{oc}\) contains an asymptotically isometric copy of complex \(\ell_1\). 
		By Theorem ~\ref{dsl1},  the complex Banach space \( E_{oc}\) contains a disjointly supported  asymptotically isometric copy  of complex \(\ell_1\). 
		It is shown in
		\cite[Theorem 2]{DGH00} 
		that 
		a  complex  Banach space  contains an asymptotically isometric copy of complex  \(\ell_1\)  if and only if its Banach dual space contains an isometric copy of \( L_1(0, 1) \). 	 
		It follows that $ (E_{oc})^{*}\stackrel{\eqref{M-bimodule*=times}}{=}  E
		^{\times} $ contains an isometric copy of \( L_1(0, 1) \).
		
		(ii) 
		By (i), the real  Banach lattice  \(  ( E
		^{\times} )_h \)  contains an isometric copy of \( L_1^{\mathbb{R}}(0,1) \).
		Since	$ ( E
		^{\times} )_h$ is strictly monotone, it follows from 
		\cite[Corollary 2]{W03} that    \( ( E
		^{\times} )_h\) contains  a lattice isometric copy of
		\( L_1^{\mathbb{R}}(0,1) \).

		(iii) 
		Since  
		$ ( E
		^{\times} )_h $ is KB, it follows from  	\cite[Theorem 2.4.12]{MN91} that $ ( E
		^{\times} )_h $
		contains no isomorphic copy of $c_0$. 
		It was shown in \cite[Theorem 3.1]{K79} that if a real Banach lattice \( X \) contains no isomorphic copy of \( c_0 \) and has a subspace isomorphic to \( L_1^{\mathbb{R}}(0,1) \), then \( X \) also has a sublattice which is order isomorphic to \( L_1^{\mathbb{R}}(0,1) \). 
		By (i),  we obtain the required result (an order isomorphism between Banach lattices  is a lattice isomorphism  \cite[Theorem 2.15]{AB06}).
	\end{proof}

Prior to proving the main results of this section, we  first demonstrate two consequences  that follow from the above corollary. 

  Recall from \cite[page 187, Corollary 2.e.4]{LT2}  that 
  if $Y$ is a real symmetric function space on $(0,1)$ which does not contain uniformly isomorphic copies of $\ell_\infty^{(n)}$ for all $n$, and if $Y$ contains a subspace isomorphic to $L_1^{\mathbb{R}}(0,1)$, then $Y$ itself coincides, up to an equivalent norm, with $L_1^{\mathbb{R}}(0,1)$.
 This together with  
  Corollary \ref{commutativecontainL1}  implies  the following. 
  \begin{corollary}
		Let $(\cM,  \tau)$ be a semi-finite atomless von  Neumann algebra and let   \(E 
		\) be a  
		symmetric function space on  \((0, 1) \) with a	 Fatou norm  $\norm{\cdot}_{E}$. 
		Assume that  	$c_{E_{oc}}
		%=c_{E(\cM, \tau)_{oc}}
		={\bf 1}$ and  
		$E(\cM, \tau)$ has the Daugavet property. 
        If  $ ( E
			^{\times} )_h$  does not contain uniformly isomorphic copies of $\ell_\infty^{(n)}$ for all $n$, 
		then  $( E
			^{\times} )_h$  coincides, up to an equivalent norm, with $L_1^{\mathbb{R}}(0,1)$. 
        
 %       the following holds:
%		\begin{enumerate}
%			\item  $( E
%			^{\times} )_h$  coincides, up to an equivalent norm, with $L_1^{\mathbb{R}}(0,1)$.   
%			\item 
%			If $ ( E
%			^{\times} )_h
%			$ is strictly monotone, then $( E
%			^{\times} )_h$  coincides, up to an equivalent norm, with $L_1^{\mathbb{R}}(0,1)$.  
%			
%			\item If $ ( E
%			^{\times} )_h$ is  KB, then $( E
%			^{\times} )_h$  coincides, up to an equivalent norm, with $L_1^{\mathbb{R}}(0,1)$. 
%		\end{enumerate} 
		
	\end{corollary}

	\begin{proposition}
		%\label{carrier=1EL1}
		Let $(\cM, \tau)$ be a semi-finite atomless von  Neumann algebra and let   \( E 
		\) be a 
		symmetric function space on  \((0, \tau({\bf 1}) )\) with the Fatou property.  
		If  
		$E(\cM, \tau)$ has the Daugavet property  and   
		$c_{( E ^{\times}
			) 
			_{oc}}
		=\bf 1$, 
		then   $ E
		$ 
		contains an isometric copy of $L_1(0, 1)$.
		
	\end{proposition}
	\begin{proof}
		Since 	$c_{( E ^{\times}
		) 
		_{oc}}
	=\bf 1$,  it follows that 
	\begin{align}\label{symmetricdua}
		( (E 
		^{\times})_{oc}	 )
		^{\times}
		\stackrel{\text{\cite[page 20, Thm.4.1]{BS88}}}{=} 
		E 
		^{\times\times}
		\stackrel{\text{\cite[page 10, Thm.2.7]{BS88}}}{=} 
		E
	\end{align}
	with equality of norms. 			
	Since 	$ E(\cM, \tau) 
	$ and 	$ E(\cM, \tau) ^{\times}
	 $ are  fully  symmetric   \cite[Corollary 5.1.12 and  Theorem 5.1.5]{DPS}, it follows that  
		\begin{align}\label{timesoc}
			( E(\cM, \tau) ^{\times}
			) 
			_{oc}
			\stackrel{\text{\cite[Thm.6.1.6]{DPS}}}{=} 
			( E^{\times} (\cM, \tau) 
			) 
			_{oc}
			\stackrel{\text{\cite[Prop.6.1.9]{DPS}}}{=} 
			(E^{\times})_{oc}	 
			(\cM, \tau)   
		\end{align}
		with equality of norms, 
		which implies that 
		$c_{ ( E(\cM, \tau) ^{\times}
		) 
		_{oc}}={\bf 1}$.  
	Thus, $( E(\cM, \tau) ^{\times}
	) 
	_{oc}$ is  fully  symmetric   \cite[Proposition  5.4.8]{DPS}
		and
		\begin{align}\label{dualidentities}
			\left( 
			( E(\cM, \tau) ^{\times}
			) 
			_{oc} 
			\right)^{*}
			& \nonumber \stackrel{\text{\cite[%Thm.5.2.9 and 
					Prop.5.3.2]{DPS}}}{=} 
			\left( 
			( E(\cM, \tau) ^{\times}
			) 
			_{oc} 
			\right)^{\times} 
			\stackrel{\eqref{timesoc}}{=} (
			(E^{\times})_{oc}	 
			(\cM, \tau)
			)
			^{\times}\\ &   \stackrel{\text{\cite[Thm.6.1.6]{DPS}}}{=} 
			(
			(E^{\times})_{oc}	 
			)
			^{\times}
			(\cM, \tau)
			\stackrel{\eqref{symmetricdua}}{=}
			E(\cM, \tau)
		\end{align}
		with equality of norms.  
	This implies that $(	( E(\cM, \tau) ^{\times}
		) 
		_{oc} )
		^{*}$ has the Daugavet property. 
		Since a Banach space inherits the Daugavet property from its Banach  dual \cite[Theorem 3.3.3]{KMZW25}, it follows that 
		$	( E(\cM, \tau) ^{\times}
		) 
		_{oc} 	\stackrel{\eqref{timesoc}}{=}
		( E^{\times})_{oc}(\cM, \tau)  
		$ has the Daugavet property.

		Since $\left(  E 
		^{\times},  \norm{\cdot}_{E^{\times}} \right)$ has the Fatou property \cite[Theorem 4.5.5]{DPS}, and so,  $\norm{\cdot}_{E^{\times}}$ is a Fatou norm.  
        %see e.g.  \cite[Lemma 5.2]{Lu63}.  
		Since  $\norm{\cdot}_{E^{\times}}$ is 	order continuous on $(E 
		^{\times})_{oc}$ and  	$c_{( E ^{\times}
			) 
			_{oc}}
		%= c_{( E(\cM, \tau) ^{\times}	) _{oc}}
		=\bf 1$, 
		by Corollary ~\ref{commutativecontainL1}(i), 
		$
		((E 
		^{\times})_{oc})
		^{\times}
		\stackrel{\eqref{symmetricdua}}{=}E
		$   contains an  isometric copy of \(L_1(0,1)\). 	\end{proof}

	We now proceed to establish 
	 main results of this section.  
     Prior to that, we recall some basic results on WCG-spaces.

	Recall that 
	a Banach
	space $(X, \norm{\cdot})$ is said to be a {\it weakly
		compactly generated space},  (briefly, {\it WCG-space}),  if $X$ contains a linearly dense weakly compact
	subset $K$, i.e., $X = \overline{\mathrm{span}}^{\norm{\cdot}}(K)$.
	Any reflexive and any separable Banach space is weakly compactly generated, see e.g.,  \cite[Chapter 13]{CMSBook11} and  \cite[Chapter XIII]{Dis84}.  
	The WCG property is not necessarily inherited by arbitrary closed subspaces--a counterexample was constructed by Rosenthal in 1974 \cite{Ro74}--whereas it is inherited by complemented subspaces.

	Recall that a  symmetric function space  is separable if
	and only if it is  order continuous \cite{KPS,LT2}. 
	This together with \cite[Theorem 4.2]{W99}
	implies that a symmetric function space  is a WCG-space
	if and only if it is order continuous. However, the associated  symmetric operator 
	space $E(\cM, \tau)$ is not necessarily separable when  $\norm{\cdot}_{E(\cM, \tau)}$ is order continuous 
	(even  $L_p(\cM, \tau)$  are not necessarily separable), see e.g. \cite[Theorem 5.6.22]{DPS} or \cite{CK17, S00}.

	The results of \cite{HNPS24Tran} provide a complete characterization of noncommutative symmetric operator spaces possessing the WCG-property. 
	Recall that a von Neumann algebra $(\mathcal{M}, \tau)$ is said to be {\it $\sigma$-finite}  if there exists a sequence
	$\{p_n\}_{n\in\mathbb{N}}$ of $\tau$-finite projections in $\cP(\mathcal{M})$ such that $p_n \uparrow \mathbf{1}$.

	\begin{theorem}\cite[Theorem 4.10]{HNPS24Tran}
		\label{sigmafiniteWCG}
		Let \( \cM \) be a \(\sigma\)-finite von Neumann algebra equipped with a semi-finite faithful normal trace \( \tau \).  Then, a strongly symmetric space \( E(\cM, \tau) \) has order continuous norm if and only if it is a WCG-space.
	\end{theorem} 

Recall that a Banach space $X$ is said to have the \emph{Radon--Nikod\'{y}m property} \cite[p.61, Definition 3]{DisU77} if for every finite measure space $(\Omega,\Sigma,\mu)$ and every countably additive vector measure
$$G:\Sigma\to X$$
of bounded variation which is absolutely continuous with respect to $\mu$, there exists a Bochner integrable function
$$ g\in L_1(\mu,X)$$
such that
$$G(E)=\int_E g\,d\mu \qquad \text{for all } E\in\Sigma.$$

%{\color{red}  For \( \varepsilon > 0 \), \( \phi \in S_{X^*} \),  a slice \( S(\phi, \varepsilon) \) of the unit ball \( B_X \) is defined by  \[ S(\phi, \varepsilon) = \{z \in B_X : \operatorname{Re} \, \phi(z)  > 1 - \varepsilon\}. \]  A point \( x \in S_X \) is said to be a \textit{denting point} if there is a slice of \( B_X \) containing \( x \) with an arbitrary small diameter. Recall that a Banach space $X$ has the \emph{Radon--Nikod\'{y}m property}  if every nonempty bounded closed convex set is the closed convex hull of its denting points \cite[Theorem 2.3.6]{B83}.

%Recall that a Banach space $X$ has the \emph{Radon--Nikod\'{y}m property} if every non-empty bounded subset $V \subseteq X$ is dentable, that is, $V$ contains a slice  with an arbitrary small diameter. }

	Several 
	equivalent characterisations of the Radon--Nikod\'{y}m property, together with illustrative examples are discussed in \cite[page  217-219]{DisU77}. 
	\begin{corollary}\cite[page  83, Corollary 7]{DisU77} \label{WCGRN}
		If \( X \) is a Banach space whose  Banach 
		dual \( X^{*} \) is a subspace of a WCG-space \( Y \), then \( X^{*} \) has the Radon--Nikod\'{y}m property.	 
	\end{corollary}

	We are now in a position to show that Theorems \ref{sigmafiniteL1subseteqE} and \ref{infiniteL1subseteqE} constitute noncommutative analogues of \cite[Proposition 2.5]{AKM12}.
	\begin{theorem}\label{sigmafiniteL1subseteqE}
		Let $(\cM, \tau)$ be  a semi-finite  atomless von Neumann algebra with a faithful, normal, \(\sigma\)-finite trace \(\tau\)  and  let  
		$E$ be  a 
		symmetric  function  space on $(0, \tau({\bf 1}))$.   
		If $E$ is a KB-space and %the operator counterpart 
		$E(\mathcal{M}, \tau)$  
		has the Daugavet property,  
		then  $L_1  \subseteq  E$ with continuous embedding. In particular, if $\tau({\bf 1})<\infty$, then $E = L_1$ (up to equivalence of norms).  
	\end{theorem}

	\begin{proof}

		First, we claim that   the fundamental function satisfies 
		\[
		\phi_{E(\cM, \tau)^{\times}}(0+)
		\stackrel{\text{\cite[Thm.6.1.6]{DPS}}}{=}
		\phi_{E^{\times}}(0+)
		> 0 . 
		\] 
		Assume by  contradiction that  	 \(\phi_{E
			^{\times}}(0+) = 0\), i.e., 	 $ (E^{\times})_{oc}\neq\{0\}$ \cite[page 68, Theorem 5.5]{BS88}. 
		Observe that, 		 
		\begin{align}
			\left( 
			( E(\cM, \tau) ^{\times}
			) 
			_{oc} 
			\right)^{*}
			\nonumber \stackrel{\eqref{dualidentities}  }{=} 
			E(\cM, \tau)
		\end{align}
		with equality of norms  and  
		$E(\cM, \tau)$ 	 is a WCG-space, 
        %{\color{blue}(do we need that $E(\mathcal{M},\tau)$ has order continuous norm? and strongly symmetric?)}
        %{\color{red} Yes, we need, since we used Theorem  \ref{sigmafiniteWCG}. It is natural that $E(\mathcal{M},\tau)$ is strongly symmetric, since $E$ is KB.}, 
        see Theorem  \ref{sigmafiniteWCG} above. 
		By  Corollary \ref{WCGRN}, 
		we deduce that 
		$E(\cM, \tau)$ has the Radon--Nikod\'{y}m property. 
		However,  \cite[Corollary 1]{W92} 
		states that if a Banach space  $X$ has the Radon--Nikod\'{y}m property, then $X$ does not have the Daugavet property. 
		This proves the claim.
		In particular, by the similar argument in \cite[Propositions 2.4 and 2.5]{AKM12},   we obtain that $E^{\times} \subseteq L_\infty
		$, and so, $E^{\times} \subseteq L_\infty
		$   with continuous embedding (see e.g. \cite[Theorem 1.8]{BS88} or \cite{KPS}).     
		%Assume, for contradiction, that    there exists $f \in E^{\times} \subseteq S  
		%(0, \tau({\bf 1}) )$
		%which is unbounded, we can find a sequence $\{I_n\}_{n\geq 1} $ in \((0,  \tau({\bf 1}) )\)  such that $m(I_n) > 0$, $|f(t)| \ge n$ for $t \in I_n$ and $n \in \mathbb{N}$. Hence, we obtain a contradiction: 
		%\begin{align*}%\label{postivelowerbound}
		%\infty >\|f\|_{E^{\times}} \ge \|f\chi_{_{ I_n }}\|_{E^{\times}} \ge n\phi_{E^{\times}}(0+) \to \infty, 
		%\end{align*} 
		%which shows that 
		%$E^{\times} \subseteq L_\infty$. 
		%By the closed graph theorem,  the embedding is continuous, i.e.,  $E^{\times} \hookrightarrow L_\infty$. 
		This implies that 
		$L_1 \subseteq  E^{\times\times}$ with continuous embedding (see e.g. \cite[Remark 4.3.11]{DPS}).
		Since  $E\neq \{0\}$, 
		it follows from \cite[Theorem 5.1.10]{DPS} that  
		$E = E^{\times\times}$ with equality of norms. 
		Hence, $L_1 \subseteq   E$  with continuous embedding.

		% The assumption   $\tau({\bf 1})<\infty$ yields  that  $(\cM, \tau)$ is a $\sigma$-finite von Neumann algebra. 
		% Hence,   $L_1 \subseteq   E$ with continuous embedding. 
		Recall that  
		$E \subseteq  L_1  $ with continuous embedding  whenever  $\tau({\bf 1})<\infty$ 
		\cite[Theorem 4.4.6]{DPS}. The proof is complete.  	
	\end{proof}

	%A complex Banach lattice has the Fatou property if and only if the same is true for its real part.

	It is noted  that 
	if 
	$X$ is a     
	Banach  lattice with  a uniformly monotone norm, then $X$ is a   
	KB-space \cite[page 371, Chap.XV.14]{B67}. 
	
	%A real Banach lattice $X$ is uniformly monotone  if and only if its  complexification  is uniformly complex convex \cite[Theorem 3.5]{L25}, and so,  its complexification  is  finite cotype \cite{D84} (by Maurey-Pisier Theorem, it does not include $c_0$ ), it follows from \cite[Theorem 5.9.6]{DPS} that   its complexification  is KB. 

	Using  Theorem \ref{EpDPinherit}, we obtain  the following result, which   complements  Theorem~\ref{sigmafiniteL1subseteqE} by addressing the case of  non-$\sigma$-finite, semi-finite   atomless von Neumann algebras.  
	%{\color{red}Due to the limitation of the condition of WCG property, we can deal with the non-$\sigma$-finite case under the condition (1) below only. It will be welcome to remove this conditinon $E\cap L_{\infty} \neq  L_1 \cap  L_{\infty}$. Let us leave this condition.}

	\begin{theorem}\label{infiniteL1subseteqE}
		
		Let $(\cM, \tau)$ be a   non-$\sigma$-finite, semi-finite  atomless von Neumann algebra.   %with a     normal faithful  trace $\tau$. 
		If a      symmetric  operator  space  $E(\mathcal{M}, \tau)$  
		has the Daugavet property and the associated      symmetric  function  space  $E$ on $(0, \tau({\bf 1}))$ satisfies one of the following  conditions: 
		\begin{enumerate}[\rm(a)]
			\item  $E$  is a KB-space and $E\cap L_{\infty} \neq  L_1 \cap  L_{\infty}$,  
			\item $E$ is uniformly monotone,   
		\end{enumerate}
		then  $L_1  \subseteq  E$ with continuous embedding.  
	\end{theorem}
	
	\begin{proof} 
		(a) 	Recall from \cite[Corollary 4.8]{HNPS24Tran} that  for a non-$\sigma$-finite, semi-finite 
		von Neumann algebra $(\cM, \tau)$, if     an order continuous    symmetric function space $E$ on $(0,\tau({\bf 1}))$ satisfies $E\cap L_{\infty} \neq  L_1 \cap  L_{\infty}$,  then    $E(\mathcal{M}, \tau)$  is necessarily a
		WCG-space.
		Hence,  by the assumption that 	$E$  is a KB-space and  the argument used in the proof of Theorem \ref{sigmafiniteL1subseteqE}, we have  $L_1  \subseteq  E$ with continuous embedding.

		(b) By the  discussion before the statement of Theorem \ref{infiniteL1subseteqE}, $E$ is a KB-space. 
		Choose  a positive
		operator $x \in E(\cM, \tau)\stackrel{\mbox{\tiny \cite[Lem.5.5.2]{DPS}}}{\subseteq} S_0 (\cM, \tau)$  such that $\tau(s(x))  = \tau({\bf 1})$.
		In particular, 
		the reduced von Neumann algebra $\cM_{s(x)}$  has a 
		$\sigma$-finite trace $\tau_{s(x)}=\tau|_{\cM_{s(x)}}$ \cite[Lemma 5.5.8]{DPS}. 
		It follows
		from Theorem \ref{EpDPinherit} that $E(\cM_{s(x)}, \tau_{s(x)})$ has the Daugavet property. 
		Since $(\cM_{s(x)}, \tau_{s(x)})$ is $\sigma$-finite and $E(0, \tau(s(x)))=E$ is KB, it follows from  
		Theorem ~\ref{sigmafiniteL1subseteqE} that     $L_1(0, \tau({\bf 1}))  = L_1(0, \tau(s(x)))\subseteq  E(0, \tau(s(x)))=E(0, \tau({\bf 1}))$ with continuous embedding.  
	\end{proof}
	
	 It should be noted that $L_1$ is a typical example of a space with a uniformly monotone norm, but $L_1$ does not satisfy the condition (a) above.

	By the same argument used in  \cite[Proposition  2.6]{AKM12},   substituting \cite[Corollary  2.2]{AKM12} with  Corollary ~\ref{commutativecontainL1} above, we obtain the following noncommutative analogue of \cite[Proposition 2.6]{AKM12}. 
	%{\color{red}To Yerlan and Fedor: In \cite{AKM12}, the condition was required to be the Fatou property. However, we check it many times but still can not see why they need the Fatou property, and we believe that the Fatou norm (the weak Fatou property, I believe that they simply missed the word 'weak') is enough. If you would like to make sure with it, then you may ask Kaminska. }
	
	\begin{theorem}\label{assumption=Linfty}
		Let $(\cM, \tau)$ be a finite 
		atomless von  Neumann algebra  with a finite faithful normal trace $\tau$ 
         and let   \(E\) be a 
		symmetric function space on \(  (0, \tau({\bf 1 }) )\)  with a Fatou property. 
		Assume that 	the operator  counterpart  $E(\cM, \tau)$ has the Daugavet property.
		If $E
		^{\times} $
		is strictly monotone or 
		$  E
		^{\times} $
		is order continuous,   
		then  $ E
		=L_\infty $  (up to equivalent norms).  	 
	\end{theorem}

	\section{The Daugavet property on (real) symmetric operator spaces of self-adjoint operators affiliated with a   finite   atomless von Neumann algebra  
	}
	\label{sec5}

	In \cite[Theorem 3.6]{AKM15}, a complete characterization of real symmetric function spaces on a finite atomless measure space with  a  Fatou norm and the Daugavet property was established, relying on the equivalent geometric characterization of the Daugavet property (see Lemma \ref{DPequivalent}), \cite[Theorem 4.8]{KMMW13}. However, the techniques employed there do not carry over directly to the complex case.

	In this section, we aim to prove Theorem \ref{mainreal}, a noncommutative counterpart of \cite[Theorem 3.6]{AKM15}, using Lemma \ref{DPequivalent} below. Theorem \ref{mainreal} characterizes symmetric operator spaces with the Fatou norm and the Daugavet property, which are subspaces of self-adjoint $\tau$-measurable operators affiliated with a   finite   atomless von Neumann algebra $(\cM, \tau)$ as $L_1(\mathcal{M},\tau)_h$ or $\mathcal{M}_h$ isometrically.
	Before that, we establish noncommutative versions of \cite[Theorem 4.8]{KMMW13} and \cite[Theorem 3.4]{AKM15} (see Theorems ~\ref{finiteordercontinuousL1} and ~\ref{finiteLinfty} below).  
	
	The key obstacle in lifting results from the commutative framework of function spaces to the noncommutative  operator spaces is the failure of usual 
	triangle inequality for the modulus of operators, that is, in general,
	$$|x+y|\not\leq|x|+|y|, \quad x, y\in S(\cM, \tau),$$
	see \cite[Lemma 2.3.15]{LSZ}. 
    Instead, we use several inequalities for self-adjoint operators with respect to 
	  the Hardy--Littlewood--P\'olya submajorization are needed.

	Recall  that the Daugavet property on a real Banach space  is equivalent to the following slice conditions. The complex version can be seen in the monograph \cite[Theorem 3.1.5]{KMZW25}. 
	
	\begin{lemma}\cite[Lemma 2.2]{KSSW00} %\cite[Lemma 3.1]{AKM15} 
		\label{DPequivalent}
		For a  Banach space \( X \) on $\mathbb{R}$, the following conditions are equivalent:
		\begin{enumerate}[\rm(i)]
			\item \( X \) has the Daugavet property.
			\item For every \( x \in S_X \), every \( y^* \in S_{X^*} \), and every \( \varepsilon > 0 \), there exists \( x^* \in S_{X^*} \) such that
			$
			x^*(x) > 1 - \varepsilon $ 
			and $\|x^* + y^*\|_{X^*} > 2 - \varepsilon.
			$
			\item For every \( x \in S_X \), every \( x^* \in S_{X^*} \), and every \( \varepsilon > 0 \), there exists \( y \in S_X \) such that
			$
			x^*(y) > 1 - \varepsilon $ 
			and 
			$\|x + y\|_{X} > 2 - \varepsilon.
			$
		\end{enumerate}
	\end{lemma}
	
	The following results will be fundamental to our analysis. 
	The first one is 
	a noncommutative generalization of \cite[Lemma 2.1]{AKM15} and \cite[Corollary 4.6]{KMMW13}.

	\begin{lemma}\cite[Theorem 4.4.6]{DPS}\label{EsubseteqL1normineq}
		Let $(\cM, \tau)$ be a  semi-finite  atomless von Neumann algebra  
		and let $E(\cM, \tau)$ be a  strongly 
		symmetric operator space.   
		For every $p\in \cP(\cM)$ with $0<\tau(p)<  \infty$, we have    
		\begin{align*}
			\norm{xp}_{L_1(\cM, \tau)}	\leq \frac{\tau(p)}{\phi_{E(\cM, \tau)}(\tau (p))} 
			\norm{xp}_{E(\cM, \tau)}, \quad x\in E(\cM, \tau). 
		\end{align*}
	\end{lemma}
	%\begin{proof} If $\tau(p)=\infty$, the proof is compete. We assume that $\tau(p)<\infty$. Since  $p\in E(\cM, \tau)^{\times}$ \cite[Theorem 4.4.6]{DPS}, it follows  that for every $x\in E(\cM, \tau)$, we have $xp\in L_1(\cM, \tau)$ and \begin{align*}\norm{xp}_{L_1(\cM, \tau)} &=	\norm{xp p}_{L_1(\cM, \tau)}\leq	\norm{p }_{E(\cM, \tau)^{\times}}\norm{xp }_{E(\cM, \tau)} \\ &  =  \frac{\tau(p)}{\phi_{E(\cM, \tau)}(\tau (p))} \norm{xp}_{E(\cM, \tau)}. \end{align*}	\end{proof}

The following lemma is well-known to experts. For the sake of completeness, we include a full proof below. 
	
	\begin{lemma}\label{fullyxp1+p_2prec}
		Let   
		$(\cM, \tau)$ be a	semi-finite   von Neumann algebra. 
		If  
		$p\in \cP(\cM)$  such that 
		$0<\tau(p)<  \infty$, then  
		for arbitrary $x\in L_1(\cM, \tau)+L_\infty(\cM, \tau)$, 
		we have   
		\begin{align}\label{pxpprecprec}
			\frac{	\tau( x p )}{\tau(p)}
			p \prec\prec 
			p x	p.
		\end{align}	
	\end{lemma}
	
	\begin{proof} 
		Note that 
		\begin{align*}
			\left| \tau( x p )	\right|
			&
          \,   \stackrel{\text{\cite[Prop.3.4.30]{DPS}}}{=} 
			\left| \tau(p x p )	\right| 
			\stackrel{\text{\cite[Coro.3.4.6]{DPS}}}{\leq} 
			\tau(|p x p| ) 	
			\\ & \stackrel{\text{\cite[Rem.3.2.6(i)]{DPS}}}{=} 
			\int_{0}^{\tau(p)}\mu( p x p)dt
			.
		\end{align*}
		Then, for arbitrary $s>0$, we have $ \int_{0}^{s}
		\mu(p) dt   \stackrel{\text{\cite[Exam.3.2.2]{DPS}}}{=} 	\min\{ s,\tau(p) \}  $,  and so, 
		\begin{align}\label{sleqtaup2}
			\frac{		\left| \tau( x p )	\right|
			}{\tau(p)}
			\int_{0}^{s}
			\mu(p) dt 
			\leq 	\frac{	\min\{ s,\tau(p) \}
			}{\tau(p)} 
			\int_{0}^{\tau(p)}\mu( p x p)dt
			.
		\end{align}

		For  $0<s\leq  \tau(p)$, 
		the fact that   the function $0<m \mapsto 	\frac{	1
		}{m}	\int_{0}^{m}\mu( p x p)dt$ is non-increasing \cite[page 298]{DPS} implies that 
		\begin{align}\label{p2intleqsint}
			\frac{1
			}{\tau(p)}	\int_{0}^{\tau(p)}\mu( p x p)dt
			&	\leq 
			\frac{1}{s}	\int_{0}^{s}
			\mu(pxp)dt,  
		\end{align}
		which shows that 
		\begin{align*}
			\frac{	\left| \tau( x p )	\right|}{\tau(p)}	\int_{0}^{s}
			\mu(p) dt 
			\stackrel{\eqref{sleqtaup2}}{\leq} 
			\frac{		s
			}{\tau(p)}	\int_{0}^{\tau(p)}\mu( p x p)dt
			\stackrel{\eqref{p2intleqsint}}{\leq}
			\int_{0}^{s}
			\mu(pxp)dt. 	\end{align*}

		If $s>\tau(p)$, 
		then  $	\int_{0}^{\tau(p)}\mu( p x p)dt 
		=
		\int_{0}^{s}\mu( p x p)dt$. 	
		In this case, 
		\eqref{pxpprecprec} follows directly from \eqref{sleqtaup2}.
	\end{proof}

	The following lemma is an immediate consequence of 
	 \cite[Lemma 4.7]{KMMW13}.

	\begin{lemma}\label{xEtauleqphiEtau}
		Let $(\cM, \tau)$ be an   atomless von Neumann algebra with  a faithful
		normal tracial state   $\tau$ and 
		let $E(\cM, \tau)$ be  a 	
		symmetric  operator space with   $\phi_{E(\cM, \tau)}(1)=1$.   
		Assume that $x\in\cM$. 
		Then for arbitrary $\alpha>0$,  we have    
		\begin{align*}
			\norm{x}_{E(\cM, \tau)}\leq 	\alpha +
			\norm{x}_{\infty}	\phi_{E(\cM, \tau)}
			\Big(
			\frac{\norm{x}_{L_1(\cM, \tau)}}{\alpha}
			\Big).
		\end{align*}	
		% In particular, if $ \phi_{E(\cM, \tau)}(0+)=0$, 
		% $\{f_n\}_{n=1}^{\infty} \subseteq  \cM$ satisfies 
		% $\mathop{\sup}\limits_n \norm{f_n}_{\infty} <\infty$ and $\lim\limits_{n\to \infty} \norm{f_n}_{L_1(\cM, \tau)} = 0$, then 
		% $\lim\limits_{n\to \infty} \norm{f_n}_{E(\cM, \tau)} = 0$.

In particular, if $\phi_{E(\cM,\tau)}(0+)=0$ and $\{f_\alpha\}_{\alpha\in A}\subseteq\cM$ is a family such that
$\sup\limits_{\alpha\in A}\|f_\alpha\|_\infty<\infty$ and $\|f_\alpha\|_{L_1(\cM,\tau)}\to0,$
then
$$\|f_\alpha\|_{E(\cM,\tau)}\to0.$$
	\end{lemma}

	The following   properties of self-adjoint $\tau$-measurable  operators are useful in the proof of Theorem \ref{finiteordercontinuousL1}.
	
	\begin{proposition}(see \label{selfproperties}\cite{SC88}, \cite[Proposition 1.2]{CS94}, see also \cite[Lemma A.1, Corollary A.1]{CPS00} for related results) Let $\cM$ be a semi-finite von Neumann algebra and let $x, y\in S(\cM, \tau)_h$. 
		\begin{enumerate}[\rm(i)]
			\item We have 
			$|x+y|\prec\prec |x|+|y|$;  
			\item  	If  $y\geq 0$ and $-y \leq x\leq y$, then 
			$x\prec\prec y$.  
			%\item   if  $x \in S(\cM, \tau)_h$ and $p\in \cP(\cM)$ satisfies $xp=px$, then   $|xp|=|x|p$. 
		\end{enumerate} 
		
	\end{proposition}

		The following lemmas provide  a useful criterion for a real/complex symmetric function  space to coincide with $L_1^{\mathbb{R}}$ (or $L_1$) with equal norms.

		\begin{lemma}\label{Lemma2.2}
			\begin{enumerate}[\rm(i)]
				\item \cite[Lemma 4.3]{KMMW13} Let \( E \) be an order continuous   real/complex  symmetric  function space on \((0,1)\).  If 
				$ 
				\liminf\limits_{\tau \to 0} \frac{\phi_{E}(\tau)}{\tau} = 1,
				$ 
				then 	$ E = L_1^{\mathbb{R}}$  or $L_1$  endowed with its canonical norm. 
				\item \cite[Corollary 5]{SS14}, \cite[Lemma 2.2]{AKM15} Let $E$ be a  real/complex symmetric function  space  on an atomless measure space $(\Omega, \cS, \mu)$. Then, 
				$ E = L_1^{\mathbb{R}}$  or $L_1$ with equality of norms if and only if $\phi_E(t) = t$ for all  $t \in \mathbb{R}_{+}$ with
				$t < \mu(\Omega)$. 
			\end{enumerate}
			
		\end{lemma}

Let $E=E(0,1)$ be the symmetric function space associated with
$\cE=E(\cM,\tau)_h$. Since $\cM$ is atomless and finite, for every
$0<t<1$ there exists a projection $p\in\cP(\cM)$ with $\tau(p)=t$.
Hence
$$\phi_{\cE}(t)=\|p\|_{\cE}=\|\mu(p)\|_E=\|\chi_{(0,t)}\|_E=\phi_E(t).$$
Therefore, once we prove
$$\liminf_{t\to0}\frac{\phi_{\cE}(t)}{t}=1,$$
we also have
$$\liminf_{t\to0}\frac{\phi_E(t)}{t}=1.$$
By Lemma \ref{Lemma2.2}(i), it follows that $E=L_1(0,1)$ with equality of norms. Consequently,
$$\cE=E(\cM,\tau)_h=L_1(\cM,\tau)_h$$
with equality of norms.

	Below, we establish a    noncommutative version
	of \cite[Theorem 4.8]{KMMW13}.

	\begin{theorem}\label{finiteordercontinuousL1}
		Let $ \cM $ be a   finite atomless von Neumann algebra equipped with a finite faithful normal trace $\tau$
		and let \( (\cE, \norm{\cdot}_{\cE} )\subseteq S(\cM, \tau)_h  \)  be a
		symmetric 
		operator  space with order continuous norm.  If  \( \cE \)     has the Daugavet property,   then \( \cE  \)  is isometrically isomorphic to   \(  L_1(\cM, \tau)_h \). 
		More precisely, \( \cE = L_1(\cM, \tau)_h  \) and  $\norm{\cdot}_{\cE}=\lambda \norm{\cdot}_{L_1(\cM, \tau)}$ for some 
		positive number $\lambda >0$.
		In particular, 
		if \( \phi_\cE(1) = 1 \), then \( \cE = L_1(\cM, \tau)_h   \) with equality of norms.
		
	\end{theorem}
	\begin{proof}
		Since 	$\norm{\cdot}_{\cE}$ is an order continuous norm on \( \cE \), it follows 
		that	$\lim\limits_{t\to 0} \phi_{\cE }(t)=0$ \cite[Lemma 5.6.18]{DPS}. 
		Since $\cM$ is    finite and atomless, it follows that 
		${\bf 1}\in \cE$ \cite[Theorem 4.4.6]{DPS}. 
		Without loss of generality, we assume that  $\tau({\bf 1})=1$ and 
		$\phi_{\cE}(1)=1$.  	
		By   Lemma \ref{Lemma2.2}(i), it suffices  to prove 	$\liminf\limits_{t\to 0} \frac{\phi_{\cE}(t)}{t}=1$.

		Take the element $x_0:={\bf 1} \in S_{\cE}$.   Define   a  normal linear  functional $$x_0^*(z):=\tau(-z), \quad  \forall z\in \cE. $$ 
	In particular, $\norm{x_0^*}_{\cE^*}=\norm{-{\bf 1}}_{\cE^{\times}}
		\stackrel{\eqref{defEtimes}}{= }
		\frac{1}{\phi_{\cE}(1)}= 1$ \cite[page 265]{DPS}, i.e.,  $x_0^*$
		is in $ S_{\cE^*}$.

		Since \( \cE \) has  the Daugavet property, 
		it follows from  Lemma \ref{DPequivalent}(iii) that 
		for each \(0<\varepsilon <1\),  
		there exists a self-adjoint operator \( y_\varepsilon \in \cE \) such that  
		\begin{enumerate}
			\item \(\|y_\varepsilon\|_{\cE} = 1\),
			\item  \(x_0^*(y_\varepsilon)=\tau(-y_\varepsilon) > 1 - \varepsilon\),
			\item \(\|y_\varepsilon + x_0\|_{\cE} =\|y_\varepsilon + \mathbf{1}\|_{\cE} > 2 - \varepsilon\).
		\end{enumerate}
		Condition (1) implies that 
		\begin{align}\label{yL1normleq1}
			\|y_\varepsilon\|_{L_1(\cM, \tau)} 	\stackrel{\eqref{L1tau1phi1}	}{\leq} \|y_\varepsilon\|_{\cE} = 1.
		\end{align}

		The spectral projections of $y_\varepsilon$ are denoted as follows:    
		\begin{align*}%\label{4-1}
			A_{\varepsilon} := e^{y_\varepsilon}	(-\infty, 0)
			,    & \notag
			\quad B_{\varepsilon} := e^{y_\varepsilon}	[0, \infty),  
			\\ 	A_{\varepsilon, 1}:=e^{y_\varepsilon}	[-2, 0),  &\notag \quad 
			A_{\varepsilon, 2} :=
			e^{y_\varepsilon}	(-\infty,-2),  
			\\ 
			B_{\varepsilon, 1}:= 
			e^{y_\varepsilon}	[0, 2], 
			& \quad
			B_{\varepsilon, 2}:=
			e^{y_\varepsilon}	(2, \infty) 
			. 
		\end{align*}

		By condition (2),  we have that
		\begin{align}\label{1 - varepsilonx0*yvarepsilon1} 
			1 - \varepsilon & \notag		< x_{0}^{*}(y_\varepsilon) 
			=\tau(-y_\varepsilon)
			=\tau(-y_\varepsilon A_{\varepsilon} )-
			\tau(y_\varepsilon B_{\varepsilon} ) 		 \leq 	\tau(-y_\varepsilon A_{\varepsilon} ) 
			\\ & = \tau(|y_\varepsilon | A_{\varepsilon} ) 
			\leq \tau(|y_\varepsilon | ) = \|y_{\varepsilon}\|_{L_1(\cM, \tau)} \stackrel{\eqref{yL1normleq1}
			}{\leq} 1.
		\end{align}
		By the facts  that    \( 0\leq y_{\varepsilon} B_{\varepsilon, 1}= |y_{\varepsilon}| B_{\varepsilon, 1}\leq 2 B_{\varepsilon, 1}\) and $2 B_{\varepsilon, 2}\leq y_{\varepsilon} B_{\varepsilon, 2}= |y_{\varepsilon}| B_{\varepsilon, 2}$, we obtain 
		\begin{align*}
			\tau(y_\varepsilon  B_{\varepsilon} )  =	\tau(|y_\varepsilon | B_{\varepsilon} ) 
			= \tau(|y_\varepsilon|)- 	\tau(|y_\varepsilon | A_{\varepsilon} )   
			\stackrel{\eqref{1 - varepsilonx0*yvarepsilon1}}{\leq}
			\varepsilon,
		\end{align*}
		which implies that
		\begin{align}\label{yvarepsilonBvarepsilon1norm1}
			0\leq 	\tau(y_\varepsilon  B_{\varepsilon, 1})	 \leq 
			\tau(y_\varepsilon   B_{\varepsilon} ) 
			\leq \varepsilon
			\quad \text{and hence,}
			\quad
			0\leq  \tau(y_\varepsilon B_{\varepsilon, 2})	 \leq 
			\varepsilon.
		\end{align}
		In particular, 
		\begin{align}\label{yvarepsilon Bvarepsilon2} 
			2 \tau( B_{\varepsilon, 2}) 
			\leq 
			\tau(y_\varepsilon B_{\varepsilon, 2})= \tau(|y_\varepsilon| B_{\varepsilon, 2})
			\stackrel{\eqref{yvarepsilonBvarepsilon1norm1}}{\leq}
			\varepsilon,
		\end{align}
		and thus, 
		\begin{align}\label{muBvarepsilon2leqvarepsilon2}
			\tau(B_{\varepsilon, 2}) \leq \frac{\varepsilon}{2}. 
		\end{align}	
		Since  $ \phi_{\cE}(0+)=0$, 	$\sup\limits_{0<\varepsilon<1} \norm{y_{\varepsilon} B_{\varepsilon, 1}}_{\infty}\leq 2$ and  $\lim\limits_{\varepsilon\to0}	\norm{y_\varepsilon   B_{\varepsilon, 1}}_{L_1(\cM, \tau)} \stackrel{\eqref{yvarepsilonBvarepsilon1norm1}}{=}0$, it follows from 
		Lemma \ref{xEtauleqphiEtau} 
		that 
		\begin{align}\label{limvarepsilonto0normyvarepsilon BvarepsilonE(tau)=0}
			\lim\limits_{\varepsilon\to0}	\norm{y_\varepsilon   B_{\varepsilon, 1}}_{\cE}
			=0. 
		\end{align}

		Define 
		\begin{align}
			\label{defP}	P_{\varepsilon} &:= 
			A_{\varepsilon, 1}
			+B_{\varepsilon, 1}   =
			e^{y_\varepsilon}	[-2, 0)
			+ 	 e^{y_\varepsilon}	[0, 2] 
		\end{align} 
		and 
		\begin{align}\label{defQ}
			Q_{\varepsilon} &:=  A_{\varepsilon, 2} +
			B_{\varepsilon, 2}
			=
			e^{y_\varepsilon}	(-\infty,-2)
			+
			e^{y_\varepsilon}	(2, \infty).
		\end{align}

		Since 
		\( - |y_{\varepsilon}| \leq y_{\varepsilon}
		\leq |y_{\varepsilon}| \), 
		it follows that
		\( - |y_{\varepsilon}|A_{\varepsilon, 2  }  
		\leq y_{\varepsilon} A_{\varepsilon, 2  } \leq 
		-2 A_{\varepsilon, 2  } \), 
		and so, 
		\begin{align}\label{yvarepsilonAvarepsilon2}
			|y_{\varepsilon}| A_{\varepsilon, 2  } \geq 
			2 A_{\varepsilon, 2  } .
		\end{align}
		Therefore, 
		\begin{align*}
			2 \tau(Q_{\varepsilon})
			&
			\stackrel{\eqref{defQ}}{	=}
			2\tau(B_{\varepsilon, 2})+
			2\tau(A_{\varepsilon, 2} )
			\stackrel{\eqref{yvarepsilon Bvarepsilon2}, \eqref{yvarepsilonAvarepsilon2}}{	\leq}
			\tau(|y_\varepsilon| B_{\varepsilon, 2}) 
			+
			\tau(|y_\varepsilon| A_{\varepsilon, 2} )
			\\ &  \,\, =
			\tau(|y_\varepsilon| Q_{\varepsilon}) 
			\leq
			\| y_\varepsilon \|_{L_1(\cM, \tau)}  
			\stackrel{\eqref{yL1normleq1}
			}{\leq}   1. 
		\end{align*}
		This shows that
		\begin{align}\label{tauQvarepsilonleq12}
			\tau(Q_{\varepsilon}) \leq \frac{1}{2}, \quad  \text{and so,} \quad \tau(P_{\varepsilon}) \geq \frac{1}{2}. 
		\end{align}
		
		By  the fact that \(- A_{\varepsilon, 1}\leq y_{\varepsilon}A_{\varepsilon, 1}+ A_{\varepsilon, 1}\leq  A_{\varepsilon, 1}\), 
		we obtain 
		\begin{align}\label{yA1prec}
			|y_{\varepsilon}A_{\varepsilon, 1}+ 
			A_{\varepsilon, 1}|
			\stackrel{\text{Prop.\ref{selfproperties}(ii)}}{\prec\prec}
			A_{\varepsilon, 1},  
		\end{align} 
		where we used the fact that $\mu(z)=\mu(|z|)$ for each $z\in S(\cM, \tau)$ \cite[page 129]{DPS}. 
		Due 
		to 
		\(y_{\varepsilon}A_{\varepsilon, 2  }+ 
		A_{\varepsilon, 2  } \leq -A_{\varepsilon, 2  } \), 
		it follows that \begin{align}\label{|yA2|leq} 
			| y_{\varepsilon}A_{\varepsilon, 2  }+ 
			A_{\varepsilon, 2  }| 
			=
			- y_{\varepsilon}A_{\varepsilon, 2  }- 
			A_{\varepsilon, 2} 
			= |y_{\varepsilon}| A_{\varepsilon, 2  }- 
			A_{\varepsilon, 2} 
			\leq 	|y_{\varepsilon}| A_{\varepsilon, 2  }
			,	 
		\end{align}
		and so, \begin{eqnarray}\label{|1+y|leq}
			 &  &
			|	{\bf 1} + y_{\varepsilon}|  \\  & 
			= & \left| \left( A_{\varepsilon, 1} + y_{\varepsilon} A_{\varepsilon, 1} \right) +
			\left( A_{\varepsilon, 2  } + y_{\varepsilon} A_{\varepsilon, 2  } \right) +
			\left( B_{\varepsilon} + y_{\varepsilon} B_{\varepsilon} \right) 
			\right| \notag 
			\\
			 \notag			&\stackrel{\text{Prop.\ref{selfproperties}}(i)}{\prec\prec}&   |A_{\varepsilon, 1}+ y_{\varepsilon}A_{\varepsilon, 1}
			| 
			+ |A_{\varepsilon, 2  }+ y_{\varepsilon}A_{\varepsilon, 2} 
			| 
			+  	
			B_{\varepsilon}  + y_{\varepsilon} B_{\varepsilon} 
			\\
			 \notag	&\stackrel{\text{\cite[Lem.2.3]{CDS97}}, \eqref{yA1prec}, \eqref{|yA2|leq}}{\prec\prec}&
			A_{\varepsilon, 1}+ |y_{\varepsilon}| A_{\varepsilon, 2  }
			+  B_{\varepsilon} + y_{\varepsilon} B_{\varepsilon}   
			\\
			 \notag
			&=&
			A_{\varepsilon, 1}+ |y_{\varepsilon}| A_{\varepsilon, 2  }
			+  B_{\varepsilon, 1}  + y_{\varepsilon} B_{\varepsilon, 1}   
			+   B_{\varepsilon, 2}  + y_{\varepsilon} B_{\varepsilon, 2} 
			\\
			    &\stackrel{\eqref{defP}, \eqref{defQ}}{=} &
			P_{\varepsilon} + |y_{\varepsilon}| Q_{\varepsilon}
			+ y_{\varepsilon} B_{\varepsilon, 1}  
			+ B_{\varepsilon, 2}. 
		\end{eqnarray}
		The monotonicity of  $\norm{\cdot}_\cE$ with respect to $\prec \prec $ (symmetric spaces with order continuous norm are strongly symmetric) and  condition (3)  yield that 
		\begin{align}
			2 - \varepsilon &\,\,<  \| {\bf 1} + y_{\varepsilon} \|_{\cE} =  \| |{\bf 1} + y_{\varepsilon} | \|_{\cE} \notag 
			\\
			&\stackrel{\eqref{|1+y|leq}}{\leq}  
			\| P_{\varepsilon} + |y_{\varepsilon}| Q_{\varepsilon}
			+ 
			y_{\varepsilon} B_{\varepsilon, 1}  
			+ 
			B_{\varepsilon, 2} \|_{\cE} \notag
			\\
			&
			\stackrel{\eqref{muBvarepsilon2leqvarepsilon2}}{\leq}  
			\| P_{\varepsilon} + |y_{\varepsilon}| Q_{\varepsilon}  \|_{\cE} 
			+ 
			\| y_{\varepsilon} B_{\varepsilon, 1} \|_{\cE}  +\phi_{\cE}\Big(\frac{\varepsilon}{2}\Big)   
			\label{P+Qnorm+yB+vare2} 
			\\ 
			& \,\, \leq \phi_{\cE}(\tau(P_{\varepsilon})) 
			+ 
			\| |y_{\varepsilon}| Q_{\varepsilon} \|_{\cE} 
			+\| y_{\varepsilon} B_{\varepsilon, 1} \|_{\cE}  +\phi_{\cE}\Big(\frac{\varepsilon}{2}\Big) 
			\notag \\
			&\,\, \leq		
			\phi_{\cE}(1) 
			+ 
			\| y_{\varepsilon} \|_{\cE} 
			+\| y_{\varepsilon} B_{\varepsilon, 1} \|_{\cE}  +\phi_{\cE}\Big(\frac{\varepsilon}{2}\Big) 	
			\notag  \\
			& \notag \,\, = 2 		+
			\| y_{\varepsilon} B_{\varepsilon, 1}  \|_{\cE}+ \phi_{\cE}
			\Big(\frac{\varepsilon}{2}\Big)  \stackrel{\eqref{limvarepsilonto0normyvarepsilon BvarepsilonE(tau)=0}}{\to} 2 	
            %\label{2-varepsilon<2+vare}
		\end{align} 
                   as $\varepsilon\to 0$. 
		By the facts that  $0<\phi_{\cE}(\tau(P_{\varepsilon})) \leq 1 $, $0 \leq  \| |y_{\varepsilon}| Q_{\varepsilon} \|_{\cE} \leq 1$, 
        %{\color{blue}(why $Q_\varepsilon$ cannot be zero? so the norm is also zero? we can simply write $0\le \||y_\varepsilon|Q_\varepsilon\|$. But later we still prove that $Q_\varepsilon$ is nonzero anyway)}
       %{\color{red} Agree, and  the case of $\lim_{\varepsilon\to 0}\||y_\varepsilon|Q_\varepsilon\|=0$ can  not occur. }  
        we have
% By the facts that
% $$0<\phi_{\cE}(\tau(P_\varepsilon))\le1 \quad\text{and}\quad 0\le \||y_\varepsilon|Q_\varepsilon\|_{\cE}\le1,$$
% and since
% $$\|y_\varepsilon B_{\varepsilon,1}\|_{\cE}+\phi_{\cE}\left(\frac{\varepsilon}{2}\right)\to0,$$
% we obtain
% $$\lim_{\varepsilon\to0}\phi_{\cE}(\tau(P_\varepsilon))=
% \lim_{\varepsilon\to0}\||y_\varepsilon|Q_\varepsilon\|_{\cE}=1.$$
		\begin{align}\label{limphiEtautauPvarepsilon=1}
			\lim_{\varepsilon \to 0} \phi_{\cE}(\tau(P_{\varepsilon})) 
			=
			\lim_{\varepsilon \to 0} \| |y_{\varepsilon}| Q_{\varepsilon} \|_{\cE} 
			= 1.
		\end{align}

		An application of Lemma \ref{EsubseteqL1normineq} yields \begin{align}\label{tau|y|PleqphitauP}
			\frac{\phi_{\cE}(\tau(P_{\varepsilon}))}
			{\tau(P_{\varepsilon})} 
			\norm{|y_{\varepsilon}| P_{\varepsilon}
			}_{L_1(\cM, \tau)}
			\leq 
			\norm{|y_{\varepsilon}| P_{\varepsilon}
			}_{\cE}\leq  1, 
		\end{align}	
		from which it follows that 
		\begin{align*}
			\frac{\tau( |y_\varepsilon|  P_\varepsilon )}
			{\tau(P_{\varepsilon})}	 
			%=		\frac{1}{\tau(P_{\varepsilon})}	\norm{|y_{\varepsilon}| P_{\varepsilon}		}_{L_1(\cM, \tau)}
			\stackrel{\eqref{tau|y|PleqphitauP}}
			{\leq } \frac{1}{\phi_{\cE}(\tau(P_{\varepsilon}))}.
		\end{align*}
		Thus, 
		\begin{align}\label{limsuptyPleq1}
			\limsup_{\varepsilon \to 0}\frac{\tau( |y_\varepsilon|  P_\varepsilon )}
			{\tau(P_{\varepsilon})}	 \stackrel{\eqref{limphiEtautauPvarepsilon=1}}
			{\leq } 
			1. 
		\end{align}

		For any $\varepsilon > 0$,  set	
		\begin{align}\label{def	gvarepsilon}
			g_\varepsilon := |y_\varepsilon| + \left(1 - \frac{\tau( |y_\varepsilon|  P_\varepsilon )}{\tau(P_\varepsilon)} 
			\right) P_\varepsilon \in \cE \subseteq L_1(\cM, \tau)_h. 
		\end{align}
		It is immediate that    	
		$
		0\leq   |g_\varepsilon| Q_\varepsilon = g_\varepsilon Q_\varepsilon \stackrel{\eqref{def	gvarepsilon}}{=} |y_\varepsilon| Q_\varepsilon.    
		$ 
		%the first equality is obtained by the facts that  $Q_\varepsilon g_\varepsilon=g_\varepsilon Q_\varepsilon$ and the uniqueness of the positive square root of a positive  operator \cite[page 61]{DPS}.
        %(i.e., $ (|g_\varepsilon| Q_\varepsilon)^2 =|g_\varepsilon| Q_\varepsilon |g_\varepsilon| Q_\varepsilon =Q_\varepsilon |g_\varepsilon|^2 Q_\varepsilon=Q_\varepsilon (g_\varepsilon)^*g_\varepsilon Q_\varepsilon=g_\varepsilon Q_\varepsilon g_\varepsilon Q_\varepsilon =(g_\varepsilon Q_\varepsilon)^2$). 
		Since	
		\[
		\frac{\tau( 		g_\varepsilon  P_\varepsilon )}{\tau(P_\varepsilon)} 
		= \frac{\tau( |y_\varepsilon|  P_\varepsilon )}{\tau(P_\varepsilon)} 
		+ 
		\Big( 
		1
		- \frac{\tau( |y_\varepsilon|  P_\varepsilon )}{\tau(P_\varepsilon)} \Big) 
		\cdot
		\frac{\tau( P_\varepsilon )}{\tau(P_\varepsilon)} 
		= 1,
		\]
		it follows that 
		\begin{align}\label{QgQ+P=yQ+P}
			|g_\varepsilon| Q_\varepsilon +  \frac{\tau(
				g_\varepsilon  P_\varepsilon )}{\tau(P_\varepsilon)}   P_\varepsilon = |y_\varepsilon| Q_\varepsilon + P_\varepsilon.
		\end{align} 
		On the other hand, we have 
		\begin{align}\label{|g|Q+|tauP|P}
			|g_\varepsilon| Q_\varepsilon +  \frac{ \tau(
				g_\varepsilon  P_\varepsilon )  }{\tau(P_\varepsilon)}   P_\varepsilon
			=
			g_\varepsilon  Q_\varepsilon +  \frac{ \tau(
				g_\varepsilon  P_\varepsilon ) }{\tau(P_\varepsilon)}   P_\varepsilon 
			\stackrel{ \text{Lem.}\ref{fullyxp1+p_2prec}
			}	
			{\prec\prec }		g_\varepsilon  Q_\varepsilon +  	g_\varepsilon  P_\varepsilon  = g_\varepsilon.  
		\end{align}

		The monotonicity of  $\norm{\cdot}_\cE$ with respect to $\prec \prec $  yields that 
		\begin{eqnarray*}%\label{1-frac}
			2 - \varepsilon   &  <&
			\| \mathbf{1} + y_{\varepsilon} \|_{\cE}
			\notag\\ &  \stackrel{\eqref{P+Qnorm+yB+vare2}}{\leq} &
			\| P_{\varepsilon} + |y_{\varepsilon}| Q_{\varepsilon} \|_{\cE}
			+ \phi_{\cE}\Big(\frac{\varepsilon}{2}\Big)  + \| y_{\varepsilon} B_{\varepsilon, 1} \|_{\cE} \notag\\
			&		 \stackrel{\eqref{QgQ+P=yQ+P}}{=}& 
			\left\|
			|g_{\varepsilon}| Q_{\varepsilon} 
			+  
			\frac{\tau( 
				g_\varepsilon P_\varepsilon )}{\tau(P_{\varepsilon})} 
			P_{\varepsilon} \right\|_{\cE} 
			+ \phi_{\cE}\Big(\frac{\varepsilon}{2}\Big)  
			+ \| y_{\varepsilon} B_{\varepsilon, 1} \|_{\cE} \notag\\
			&
			 	\stackrel{\eqref{|g|Q+|tauP|P}}{\leq} &
			\| g_{\varepsilon} \|_{\cE} 
			+
			\phi_{\cE}
			\Big(\frac{\varepsilon}{2}\Big) 
			+ 
			\| y_{\varepsilon} B_{\varepsilon, 1} \|_{\cE} 
			\notag	\\
			\notag 	&\stackrel{\eqref{def	gvarepsilon}}{\leq} &
			\| y_{\varepsilon} \|_{\cE}
			+ 
			\left| 1 - \frac{\tau( |y_\varepsilon|  P_\varepsilon )
			}{\tau(P_\varepsilon)} 		\right| 
			\| P_{\varepsilon} \|_{\cE}
			+ \phi_{\cE}
			\Big(\frac{\varepsilon}{2}\Big)  + 
			\| y_{\varepsilon} B_{\varepsilon, 1} \|_{\cE} \\
			 \notag  & =& 1 + \phi_{\cE}(\tau(P_{\varepsilon})) 
			\left| 1 - \frac{\tau( |y_\varepsilon|  P_\varepsilon )}{\tau(P_\varepsilon)} 
			\right| 
			+ \phi_{\cE}
			\Big(\frac{\varepsilon}{2}\Big)  + \| y_{\varepsilon} B_{\varepsilon, 1} \|_{\cE} \\ 
			 \quad& \leq &
			1 + \left| 1 - \frac{\tau( |y_\varepsilon|  P_\varepsilon ) }{\tau(P_{\varepsilon})} 		
			\right|
			+ 
			\phi_{\cE}
			\Big(\frac{\varepsilon}{2}\Big)  
			+ 
			\| y_{\varepsilon} B_{\varepsilon, 1} \|_{\cE} .   
		\end{eqnarray*}
		By the facts that 
		$\phi_{\cE}(0+)=0$,   \eqref{limvarepsilonto0normyvarepsilon BvarepsilonE(tau)=0} and    \eqref{limsuptyPleq1}, we deduce that  
		\begin{align*}%\label{limvarepsilonto0phiEtautauPvarepsilon1}
			\lim_{\varepsilon \to 0} \frac{1}{\tau(P_{\varepsilon})} 
			\tau( |y_\varepsilon|  P_\varepsilon ) = 0. 
		\end{align*}
		Since \(0 < \tau(P_{\varepsilon}) \leq 1\),  it follows that 	\begin{align*}
			\lim_{\varepsilon \to 0} \tau( |y_\varepsilon|  P_\varepsilon )  = 0, 
		\end{align*}
		and hence,  \(\lim\limits_{\varepsilon \to 0} \tau(|y_{\varepsilon}| A_{\varepsilon, 1  } ) = 0\). In view of \eqref{1 - varepsilonx0*yvarepsilon1},  we know that \(\lim\limits_{\varepsilon \to 0} \tau (|y_{\varepsilon}| A_{\varepsilon} )= 1\). Therefore, 
		\begin{align}\label{limvarepsilonto0normyvarepsilonAvarepsilon1L1tau1}
			\lim_{\varepsilon \to 0} \tau (|y_{\varepsilon}| A_{\varepsilon, 2 }) = 1
			. 
		\end{align} 
		This means that  \(\tau(A_{\varepsilon, 2 }) > 0\) for \(\varepsilon > 0\) small enough, and thus,  by Lemma \ref{EsubseteqL1normineq}, we obtain  
		\begin{align*}
			\frac{\phi_{\cE}(\tau(A_{\varepsilon, 2 }))}{\tau(A_{\varepsilon, 2 })} 
			\tau (|y_{\varepsilon}| A_{\varepsilon, 2 }) 
			%=		\frac{\phi_{\cE}(\tau(A_{\varepsilon, 2 }))}{\tau(A_{\varepsilon, 2 })} \| |y_{\varepsilon}|  A_{\varepsilon, 2 }\|_{L_1(\cM, \tau)} 
			\leq 
			\| |y_{\varepsilon}|  A_{\varepsilon, 2 }\|_{\cE}
			\leq 1. 
		\end{align*}
		In view of \eqref{limvarepsilonto0normyvarepsilonAvarepsilon1L1tau1},  we deduce that
		\begin{align}\label{limsupphiEtautauA2varepsilontauA2varepsilonleq1}
			\limsup_{\varepsilon \to 0} \frac{\phi_{\cE}(\tau(A_{\varepsilon, 2 }))}{\tau(A_{\varepsilon, 2 })} \leq 1. 
		\end{align}

		Since the function \(0<t \mapsto \varphi(t) := \frac{\phi_{\cE}(t)}{t}\) is non-increasing (see Section \ref{section:symmetric}), \(\varphi(1) = 1\) and 
		\begin{align*}%\label{A2leqQleqP}
			\tau(A_{\varepsilon, 2 })\stackrel{\eqref{defQ}}{	\leq}
			\tau(Q_\varepsilon) \stackrel{\eqref{tauQvarepsilonleq12}}{\leq} \tau(P_\varepsilon)\leq \tau({\bf 1})=1,
		\end{align*} 
		it follows that 
		\begin{align*}%\label{1leqlimsupleq1}
			1=\varphi(1)
			\leq 
			\limsup_{\varepsilon \to 0} 
			\frac{\phi_{\cE}(\tau(P_{\varepsilon}))}{\tau(P_{\varepsilon})} 
			\leq  
			\limsup_{\varepsilon \to 0} 
			\frac{\phi_{\cE}(\tau(A_{\varepsilon, 2 }))}{\tau(A_{\varepsilon, 2 })} 
			\stackrel{\eqref{limsupphiEtautauA2varepsilontauA2varepsilonleq1}}{\leq} 
			1, 
		\end{align*}
		which shows that
		\begin{align*}
			\limsup_{\varepsilon \to 0} 
			\frac{\phi_{\cE}(\tau(P_{\varepsilon}))}{\tau(P_{\varepsilon})}=1 .
		\end{align*}
		By \eqref{limphiEtautauPvarepsilon=1}, we have 
		$
		\limsup\limits_{\varepsilon \to 0} \frac{1}{\tau(P_{\varepsilon})} = 1.
		$
		Equivalently, 
		\begin{align*}
			\lim_{\varepsilon \to 0} \tau(P_{\varepsilon}) = 1.
		\end{align*}
		As a consequence,
		\[
		\lim_{\varepsilon \to 0} \tau(A_{\varepsilon, 2 }) \leq \lim_{\varepsilon \to 0} \tau(Q_{\varepsilon}) = 1 - \lim_{\varepsilon \to 0} \tau(P_{\varepsilon}) = 0.
		\]
		By ~\eqref{limsupphiEtautauA2varepsilontauA2varepsilonleq1} and the monotonicity  of 
		\(\varphi(t)= \frac{\phi_{\cE}(t)}{t} \), we obtain  $\liminf\limits_{t\to 0} \frac{\phi_{\cE}(t)}{t}=1$. 	
	\end{proof}

	We now turn to the following generalization of  \cite[Theorem 3.4]{AKM15}.

	\begin{theorem}\label{finiteLinfty}
		Let $ \cM $ be a   finite atomless von Neumann algebra equipped with a finite faithul normal trace $\tau$
		and let \( \cE\subseteq S(\cM, \tau)_h  \)  be a  
		symmetric 
		operator  space.  
		If  $\norm{\cdot}_{\cE}$ is a Fatou norm on  \( \cE \),  \( \cE \)  has  the Daugavet property and 
		\( \phi_{\cE^{\times}}(0+) = 0 \), then  \( \cE  \)  is isometrically isomorphic to   \( \cM_h \). 	More precisely, \( \cE = \cM_h  \) and  $\norm{\cdot}_{\cE}=\lambda \norm{\cdot}_{\infty}$ for some 
		positive number $\lambda >0$.
		In particular,
		if \( \phi_\cE(1) = 1 \), then \( \cE = \cM_h   \) with equality of norms. %Moreover, the isometry is given by an appropriate multiple of the formal identity between both spaces. 
	\end{theorem}
	\begin{proof}
		Without loss of generality, we assume that  $\tau({\bf 1})=1$ and 
		$\phi_{\cE}(1)=1$.  		The element $x_0:=-{\bf 1}$  satisfies $x_0\in S_{\cE}$, and  the normal linear  functional defined by $$x_0^*(z):=\tau(z), \quad \forall z \in \cE,$$  
		is in $ S_{\cE^*}$.

		Since \( \norm{\cdot}_{ \cE} \) is a Fatou norm, the fully symmetric operator space  \(\cE^{\times} \) is a 1-norming (see \cite[Definition 3.2.2]{AK06} or \cite[page 54, Exercise 10]{Dis84}) subspace of the Banach  dual space \( \cE^* \)  \cite[page 268 and Theorem 4.5.6(iii)]{DPS}, i.e., 
		\begin{align} \label{Fatounormequal}	\norm{z}_{\cE}=
			\| z \|_{\cE^{\times\times}}
			=\sup \{\tau (| yz |) : y \in \cE^{\times}, \| y \|_{\cE^{\times}} \leq 1 \}, \quad z\in \cE.  
		\end{align}

		Since \( \cE \) has  the Daugavet property, 
		it follows from   \cite[Lemma 3.1]{AKM15}    that 
		for each \(0<\varepsilon <1\),  there exists a Hermitian (see e.g. \cite[page 253]{DPS}),  normal linear functional  \( x_\varepsilon^* \in \cE^* \) such that  
		\begin{itemize}
			\item[(1)] \(\|x_\varepsilon^*\|_{\cE^*} =1\),
			\item[(2)]  \(x_\varepsilon^*(x_0)=x_\varepsilon^*(-{\bf 1})> 1 - \varepsilon\),
			\item[(3)] \(
			\|x_0^* + x_\varepsilon^*\|_{\cE^*}  > 2 - \varepsilon\).
		\end{itemize}	
		In particular,   by \cite[Theorem 5.2.9]{DPS},  $x_\varepsilon^*$ can be represented as 
		\begin{align*}
			%\label{normalKotheduality}
			x_\varepsilon^*(z) = \tau(zy_\varepsilon), \quad z \in \cE,
		\end{align*}
		for the unique  self-adjoint operator $y_\varepsilon \in \cE^\times$ 
		and 
		$	\norm{y_{\varepsilon}
		}_{\cE^{\times}} =
		\norm{x_\varepsilon^*}_{\cE^{*}} 
		= 1$.

		Since 
		\begin{align*}
			|(x_0^* + x_\varepsilon^*)(z) | =
			| \tau(z+zy_\varepsilon )|
			\leq \norm{{\bf 1}+ y_\varepsilon}_{\cE^{\times}}
			\norm{z}_{\cE}, \quad \forall z \in \cE,
		\end{align*}   it follows from condition (3) that
		\begin{align*}%\label{1+yvare>2-vare}
			2 - \varepsilon
			< 
			\|x_0^* + x_\varepsilon^*\|_{\cE^{*}}  \leq 
			\norm{{\bf 1}+ y_\varepsilon}_{\cE^{\times}}. 
		\end{align*}

		By the same spectral projection decomposition of $y_{\varepsilon}$ and  similar arguments in the proofs of Theorem \ref{finiteordercontinuousL1}   and \cite[Theorem 3.4]{AKM15}, we obtain that 	 	$\phi_{\cE^{\times}}(t)=t$ for arbitrary $t\in (0, 1)$. 
		This implies that 
		\( \cE^{\times} = L_1(\cM, \tau)_h \) endowed with its canonical norm (see Lemma \ref{Lemma2.2}(ii)).     
		Furthermore, 
		we have 
		\begin{align*}
			\cM_h\subseteq 	\cE\subseteq \cE^{\times\times}=(L_1(\cM, \tau)_h)^{\times}
			\stackrel{\text{ \cite[p.253 and Thm.5.2.9]{DPS}}}{=}
			\left(
			L_1(\cM, \tau)^{\times}
			\right)_h
			=\cM_h
		\end{align*}
		and 
		$
		\norm{z}_{\cE}
		\stackrel{\eqref{Fatounormequal}}{=}	\norm{z}_{\cE^{\times\times}}
		= 	\norm{z}_{\infty}
		$ for all $z\in 	\cE=\cM_h$. 	
	\end{proof}

	It should be noted that 
	a complete characterization of the Dunford--Pettis property for   noncommutative symmetric operator 
	spaces affiliated with various types of  von Neumann algebra    is given in  \cite{HPS22}. 
	In particular,
	if the type $I_{\infty}$-summand or the type $II$-summand of   a semi-finite von Neumann algebra $(\cM, \tau)$ is nontrivial, then any symmetric operator space $E(\cM, \tau)$ fails  the Dunford--Pettis property \cite[Theorem 19]{HPS22}.
We establish the following noncommutative analogue of \cite[Corollary 3.5]{AKM15}.

	\begin{corollary}
		Let $(\cM, \tau)$ be an    atomless von Neumann algebra  with 
		$\tau({\bf 1})=1$. 
		If a   
		symmetric 
		operator  space  \(\cE \subseteq S(\cM, \tau)_h\)   has  the Daugavet property, which  is  isomorphic to  \( C(K) \), then \( \cE \) is isometrically isomorphic to \( \cM_h \).	 
Moreover, $\cM= \mathop{\oplus}\limits_{1\le n\le m} \cA_n \otimes \mathbb{M}_n $, where $m<\infty $, $\cA_n$'s are atomless commutative algebras and $\mathbb{M}_n$ is the space of all $n\times n$ matrices.
	\end{corollary}

	\begin{proof}
Since %$\cM$ is atomless and 
\(\cE \)   has  the Daugavet property, it follows that $\cE$ is infinite-dimensional\cite[page 66]{KMZW25}, and therefore, $C(K)$ is also infinite-dimensional. 
Recall that $L_1(\cM,\tau)_h$ is weakly sequentially complete\cite[Corollary 5.2.14]{DPS} while $C(K)$ is not. This implies that $\cE$ is not isomorphic to $L_1(\cM,\tau)_h$.
		Since \( C(K) \) has the Dunford--Pettis property, it follows from \cite[Table 1]{HPS22} (or \cite{HNPS24Tran}) that
        $\cM$ is of the form $\mathop{\oplus}\limits_{1\le n\le m} \cA_n \otimes \mathbb{M}_n $, where $m<\infty $, $\cA_n$'s are atomless commutative algebras and $\mathbb{M}_n$ is the space of all $n\times n$ matrices,  and 
        \( \cE  \) equals \( \cM_h \) up to equivalence of norms\footnote{Even though the results in \cite{HPS22} were stated for complex Banach spaces, the results still hold for real symmetric spaces of self-adjoint $\tau$-measurable operators.}.
        By \cite[Theorem 4.5.8(ii)]{DPS}, the norm $\norm{\cdot}_\cE$ on $\cM_h$ is a Fatou norm. 
	Moreover, 	since 
		\( (\cM_h)^{\times}=L_1(\cM, \tau)_h \) \cite[page 253]{DPS}, 
		it follows that 
		\( \phi_{\cE^{\times}}(0+) = 0 \).
		An
		application of Theorem~\ref{finiteLinfty} completes the proof.
	\end{proof}

	We now proceed to prove the main result of this section.
	
	\begin{proof}[Proof of Theorem \ref{mainreal}]
		
		If \(c_{\cE_{oc}}= {\bf 1}\), then by Theorem \ref{inheriteDP}, \( \cE_{oc} \) has the Daugavet property. 
		Since \( \cE_{oc}\) has an  order continuous norm, it follows from  Theorem \ref{finiteordercontinuousL1}   that \( \cE_{oc} = L_1(\cM, \tau)_h \) isometrically.  	%  In particular, there exists a positive number $\lambda >0$ such that  $\norm{x}_{\cE}=\lambda \norm{x}_{L_1(\cM, \tau)}$  for all $x\in \cE_{oc}$.  
		Furthermore, \( L_1(\cM, \tau)_h=\cE_{oc} \subseteq \cE \subseteq L_1(\cM, \tau)_h \) implies that  \(\cE = L_1(\cM, \tau)_h=\cE_{oc} \) up to isometries.

		If \(\cE_{oc}= \{0\}\), then we have \( \phi_{\cE}(0+) > 0 \). 
		This implies that the symmetric function space corresponding to $\cE$ is $L_\infty (0,\tau({\bf 1}))$ (as sets) (see e.g. \cite[page 118]{LT2}).  
		% the   argument of  \cite[Proposition 2.4]{AKM12}.
		Hence,   
		$\cE = \cM_h$,  and 
		therefore,  \( \cE^{\times} = L_1(\cM, \tau)_h \) up to equivalence of norms. 
		Then,  the condition \( \phi_{\cE^{\times}}(0+) = 0 \), together with Theorem~\ref{finiteLinfty}, implies that \( \cE = \cM_h \) isometrically.  
	\end{proof}

	\section{ The Daugavet property on (real) symmetric operator spaces of self-adjoint operators affiliated with an   infinite   atomless von Neumann algebra  }\label{sec6}
	
	In this section,  we investigate the Daugavet property on  (real) symmetric operator spaces of self-adjoint $\tau$-measurable operators affiliated with an   infinite   atomless von Neumann algebra $(\cM, \tau)$.  
	Recall that 
	$L_1^{\mathbb{R}}$ 
	is the only real  symmetric function  space over an infinite atomless measure space with uniformly monotone norm that has the Daugavet property \cite[Theorem 4.4]{AKM15}.  
	The main result of this section states that   the only symmetric operator  space \(\cE\subseteq S(\cM, \tau)_h \)    with  uniformly monotone norm  that
	has the Daugavet property is \( L_1(\cM, \tau)_h   \), see Theorem \ref{infiniteselfum} below. 
		In particular, the real symmetric function space associated with $\cE$ is precisely $L_1^{\mathbb{R}}(0, \tau({\bf 1}))$.

	Inspired by ideas  from \cite[Theorem 4.1]{AKM15},  we start the proof for Theorem \ref{infiniteselfum} by proving the following auxiliary result. 
	
	\begin{theorem} \label{UM}	Let $(\cM, \tau)$ be an  infinite  semi-finite      atomless von Neumann algebra and let \(\cE\subseteq S(\cM, \tau)_h \) be a  strongly
		symmetric 
		operator   space.  
		Assume  that \( \cE \) has the following properties:
		\begin{enumerate}[\rm(a)]
			\item For any \( p\in\cP(\cM) \) with \( \tau(p) < \infty \), the space \( \cE_p \) is order continuous.
			\item Let \( p\in\cP(\cM)  \) with \(  \tau(p) < \infty \). 
			If \( x \in S_\cE \) and for every \( \varepsilon > 0 \), \( y_\varepsilon \in S_\cE \), then 
			$$
			\lim_{\varepsilon \to 0} \|x + p y_\varepsilon p\|_\cE = 2 $$
			whenever $  	\mathop{\lim}\limits_{\varepsilon \to 0} \|p y_\varepsilon p\|_\cE = 1$  and  
			$	\mathop{\lim}\limits_{\varepsilon \to 0} \|x + y_\varepsilon\|_\cE = 2.
			$ 
		\end{enumerate}
		If \( \cE \) has the Daugavet property, then it is isometrically isomorphic to \( L_1(\cM, \tau)_h \). In particular,
		if \( \phi_\cE(1) = 1 \), then \( \cE = L_1(\cM, \tau)_h   \) with equality of norms.
	\end{theorem}
	\begin{proof}
		Without loss of generality,  we may  assume  that \( \phi_{\cE}(1) = 1 \).	
		Applying Lemma ~\ref{Lemma2.2}(ii),  it then suffices to prove that $\phi_{E}(t)= \phi_{\cE }(t)=t$ for all $0<t< \tau({\bf 1})$.   
% {\color{blue}(Let $E$ be a symmetric function space on atomless $(0,\tau(\mathbf{1}))$ associated with $\cE$, i.e. $\cE=E(\mathcal{M},\tau)$ and $\|x\|_{\mathcal{E}}=\|\mu(x)\|_{E(0,\tau(\mathbf{1}))}$ for $x\in \mathcal{E}.$ Note $\phi_E=\phi_\mathcal{E}.$ We show that $\phi_{\cE }(t)=t$ for all $0<t< \tau({\bf 1})$, equivalently $\phi_{E}(t)=t$ for all $0<t< \tau({\bf 1})$. By Lemma \ref{Lemma2.2}(2), we have $E=L_1^\mathbb{R}$ with equality of norms. Therefore, $\cE=L_1(\mathcal{M},\tau)_h$ with equality of norms.)}
		
		Let \( \varepsilon > 0 \) and \( 0 < a < \tau(\mathbf{1}) \). Since \( (\cM, \tau) \) is atomless, there exists a projection \( p \in \cP(\cM) \) such that \( \tau(p) = a \). 	
		Define 
		\begin{align}\label{defcb}   c := \frac{1}{\phi_{\cE}(a)} ,  \quad  b := \frac{\phi_{\cE}(a)}{a}, \quad  
			x := c p%= \frac{1}{\phi_{\cE}(a)} p 
			\quad \text{   and   }  \quad  y := -b p.
		\end{align} 
		Then,  \(\norm{x}_{\cE}=\frac{1}{\phi_{\cE}(a)} \cdot \phi_{\cE}(a)=1\)  and 
		\(  \norm{y}_{\cE^{\times}}=\frac{\phi_{\cE}(a)}{a}  \cdot  \phi_{\cE^{\times}}(a)
		\stackrel{\eqref{defEtimes}}{= }   
		\frac{\phi_{\cE}(a)}{a}  \cdot  \frac{a}{\phi_{\cE}(a)}
		=1  \), i.e., 
		\( x \in S_{\cE} \) and \( y \in S_{\cE^{\times}} \). Consider the normal linear functional  \( F \in \cE^{\times} \) given by 
		\[
		F(z) = \tau(z y) \stackrel{\eqref{defcb}}{= } \tau(-bp z) 
		\quad \text{for } z \in \cE.
		\]
		Then, 
		$\norm{F}_{\cE^{*}}=\norm{y}_{\cE^{\times}}=1$	\cite[page 265]{DPS}.

		Since \( \cE \) has  the Daugavet property, 
		it follows from    Lemma \ref{DPequivalent} (iii) that 
		for each \(0<\varepsilon <1\),  there exists a  self-adjoint operator  \( y_\varepsilon \in S_\cE \) such that
		\begin{align}\label{DPequivalent2vare}F(y_\varepsilon) >  1 - \varepsilon \quad \text{and} \quad \|x + y_\varepsilon\|_\cE > 2 - \varepsilon.
		\end{align}
		It follows that	
		\begin{align}
			1 - \varepsilon &  \label{F(yvare)>1-vare}
			\stackrel{\eqref{DPequivalent2vare}}{<}  	F(y_\varepsilon) 
			=
			-b \tau(py_{\varepsilon} p)
			\leq \tau(|bpy_{\varepsilon} p| ) 
			\\ & \label{3-5}	\,\, \leq \|p y_\varepsilon p\|_{\cE} \|b p\|_{\cE^{\times}}=\|p y_\varepsilon p\|_{\cE} \|y\|_{\cE^{\times}}
			= \|p y_\varepsilon p\|_{\cE} 
			\leq 1,
		\end{align}
		which implies that 
		\( 1- \varepsilon  < 
		\|p y_\varepsilon p\|_{\cE} \leq 1 \). 
		By  \eqref{DPequivalent2vare} and assumption (b),  without loss of generality,  we  suppose that for any \(\varepsilon > 0\),	
		\begin{align}\label{2-vare<normE}
			2 - \varepsilon < \|x +  p y_\varepsilon p\|_{\cE}.
		\end{align}
		
		Consider the following spectral projections:  \begin{align}
		    A_{\varepsilon} &\nonumber  := e^{ p y_\varepsilon p}(-\infty, 0)
		:=A_{\varepsilon, 1  }+A_{\varepsilon, 2  } \leq p, \\  B_{\varepsilon} & := e^{ p y_\varepsilon p}[0, \infty):=
		B_{\varepsilon, 1}+B_{\varepsilon, 2} \leq p,\end{align} 
        where 
		\begin{align} \label{pypspectralprojs} A_{\varepsilon, 1} \notag  := e^{ p y_\varepsilon p}\left[-\frac{2}{\phi_{\cE}(a)}, 0\right)&, 
			\quad 
			A_{\varepsilon, 2  }  := e^{ p y_\varepsilon p}\left(-\infty, -\frac{2}{\phi_{\cE}(a)}\right),\\ 
			B_{\varepsilon, 1} := e^{ p y_\varepsilon p}\left[0, \frac{2}{\phi_{\cE}(a)}\right] &, \quad   B_{\varepsilon, 2} := e^{ p y_\varepsilon p}\left( \frac{2}{\phi_{\cE}(a)}, \infty\right).
		\end{align}	
		
		Combining  \begin{align*}
			b	\tau( y_{\varepsilon} A_{\varepsilon} ) 
			+ b
			\tau(y_{\varepsilon} B_{\varepsilon} ) \notag   =		b\tau(y_{\varepsilon} p) \stackrel{\eqref{F(yvare)>1-vare}}{< }   -1 + \varepsilon 
		\end{align*} and 
		\begin{align*}
			-
			b	\tau(y_{\varepsilon}   A_{\varepsilon} )+	 b	\tau(y_{\varepsilon}   B_{\varepsilon} )
			&=
			-
			b	\tau(py_{\varepsilon} p  A_{\varepsilon} )+	 b	\tau(py_{\varepsilon} p  B_{\varepsilon} )
			\\ &  =		  b\tau(| py_{\varepsilon} p|A_{\varepsilon} )
			+
			b\tau(| py_{\varepsilon} p|B_{\varepsilon} )
			=	b\tau(| py_{\varepsilon} p|)	\stackrel{\eqref{3-5}}{\leq } 1, 
		\end{align*}
		we deduce  that  
		\begin{align}\label{yBleqvare}
			0\leq b
			\tau( y_{\varepsilon} B_{\varepsilon} )  \leq \varepsilon/2.
		\end{align}
		Since  \( 0\leq \tau( y_{\varepsilon} B_{ \varepsilon, i}) \leq \tau( y_{\varepsilon} B_{\varepsilon} )  \), 	it follows that \begin{align}\label{yvareBvarei}
			\tau( B_{\varepsilon, i } y_{\varepsilon} B_{\varepsilon, i }) =		\tau( y_{\varepsilon} B_{\varepsilon, i })  \to 0 	\end{align}  as \(\varepsilon \to 0\) for each $i=1, 2$, and so,  \( B_{\varepsilon, 1} y_\varepsilon B_{\varepsilon, 1} \stackrel{\cT_m}{\to} 0 \) as \(\varepsilon \to 0\)  \cite[Proposition 3.4.11]{DPS}.  In particular,  \( B_{\varepsilon, 1} y_\varepsilon B_{\varepsilon, 1} \to 0 \)  with respect to
		$\sigma (\cM, L_1 (\cM, \tau))$.
		Since  \( 0\leq   B_{\varepsilon, 1} y_\varepsilon B_{\varepsilon, 1}  \in \cM \), 
		%(since \( B_{\varepsilon, 1} y_\varepsilon B_{\varepsilon, 1}  z \stackrel{\cT_m}{\to} 0 \)  for arbitrary $z\in L_1(\cM, \tau)$  	and   	\(\mu(t;  B_{\varepsilon, 1} y_\varepsilon B_{\varepsilon, 1} z) \leq \mu(t/2;  B_{\varepsilon, 1} y_\varepsilon B_{\varepsilon, 1})\mu(t/2;  z) \leq \frac{2}{\phi_{\cE}(a)}  \mu(t/2;  B_{\varepsilon, 1})\mu(t/2;  z)\leq \frac{2}{\phi_{\cE}(a)}  \mu(t/2;  z) \in L_1 \), 
		%	then   by the Dominated Convergence \cite[Theorem 3.4.21]{DPS}, $\tau(B_{\varepsilon, 1} y_\varepsilon B_{\varepsilon, 1}z)\to 0$ 
		%	as \(\varepsilon \to 0\)). 
		\( \cE_p \) is order continuous and $p\in \cE_p$,  it follows from \cite[Theorem 5.5.15]{DPS}  that  
		\begin{align}\label{yvarepsilonBvarepsilon1to0}
			\|B_{\varepsilon, 1} y_\varepsilon B_{\varepsilon, 1}\|_{\cE} =	\|(B_{\varepsilon, 1} y_\varepsilon B_{\varepsilon, 1}) p\|_{\cE} \to 0, \quad \text{as  }  \varepsilon \to 0.	
		\end{align}
		
		In view of 
		\begin{align}\label{4avarepsilon}
			\frac{2}{a} \tau(B_{\varepsilon, 2})=
			\frac{\phi_{\cE}(a)}{a} 
			\cdot 		\frac{2}{\phi_{\cE}(a)}
			\tau(   B_{\varepsilon, 2}) 	\stackrel{\eqref{defcb}, \eqref{pypspectralprojs} }{\leq }  b
			\tau( y_{\varepsilon} B_{\varepsilon, 2})  \stackrel{\eqref{yBleqvare}}{\leq } 
			\frac{\varepsilon}{2} 
		\end{align} 
		and  the  order continuity of \( \cE_p \) (in particular, $\phi_{\cE}(\varepsilon)\to 0$ as  $\varepsilon\to 0$), 
		we obtain   
		\begin{align}\label{cBvarepsilon2to0}
			0\leq  \left\| c  B_{\varepsilon, 2} \right\|_{\cE} \stackrel{\eqref{defcb}}{= }		\left\| \frac{1}{\phi_{\cE}(a)} B_{\varepsilon, 2} \right\|_{\cE} = \frac{\phi_{\cE}(\tau(B_{\varepsilon, 2}))}{\phi_{\cE}(a)} 
			\stackrel{\eqref{4avarepsilon}}{\leq }  \frac{\phi_{\cE}(\frac{a \varepsilon}{4})}
			{\phi_{\cE}(a)} \to 0  \quad \text{as } \varepsilon \to 0.  
		\end{align}

		Define
		\begin{align}\label{defP3P}P_{\varepsilon} := A_{\varepsilon, 1} + B_{\varepsilon, 1 }=e^{ p y_\varepsilon p}\left[-\frac{2}{\phi_{\cE}(a)}, 0\right) + e^{ p y_\varepsilon p}\left[0, \frac{2}{\phi_{\cE}(a)}\right] 
		\end{align}
		and
		\begin{align}\label{defQ3Q} Q_{\varepsilon} := A_{\varepsilon, 2}+ B_{\varepsilon, 2} =e^{ p y_\varepsilon p}\left(-\infty, -\frac{2}{\phi_{\cE}(a)}\right)+
			e^{ p y_\varepsilon p}\left( \frac{2}{\phi_{\cE}(a)}, \infty\right).
		\end{align}	
		
		Since 	\[
		\frac{2}{a} \tau(Q_{\varepsilon}) = \frac{\phi_{\cE}(a)}{a} \cdot   \frac{2}{\phi_{\cE}(a)} \tau(Q_{\varepsilon}) 	\stackrel{\eqref{defcb}, \eqref{defQ3Q}}{\leq }   
		b \tau( |p  y_{\varepsilon}  p| Q_{\varepsilon})   \leq b\tau(|py_{\varepsilon} p| ) \stackrel{\eqref{3-5}}{\leq }   1,
		\]
		it follows that 
		\begin{align}\label{Qvarepsilona2}
			\tau(Q_{\varepsilon}) \leq \frac{a}{2}, \quad  \text{and so,} \quad \frac{a}{2} \leq \tau(P_{\varepsilon}) \leq \tau(p)=a. 
		\end{align}
		
		Since \begin{align*}%\label{+-cA1}
			- c A_{\varepsilon, 1}\stackrel{\eqref{pypspectralprojs} }{\leq } A_{\varepsilon, 1} (p y_{\varepsilon} p)A_{\varepsilon, 1}
			+ c A_{\varepsilon, 1}	=A_{\varepsilon, 1} y_{\varepsilon}A_{\varepsilon, 1}+ c A_{\varepsilon, 1}\stackrel{\eqref{pypspectralprojs} }{\leq }  c A_{\varepsilon, 1}\end{align*}
		it follows from Proposition \ref{selfproperties} (ii) that 
		\begin{align}\label{yA1prec3} 
			|A_{\varepsilon, 1}		y_{\varepsilon}A_{\varepsilon, 1}+ c 
			A_{\varepsilon, 1}
			| 	\prec \prec 
			c A_{\varepsilon, 1} . 
		\end{align}
		By virtue of 
		\begin{align}\label{a2ya2absolute}
			|A_{\varepsilon, 2  }y_{\varepsilon}A_{\varepsilon, 2  }+ 
			c 	A_{\varepsilon, 2  }| &\notag  =-A_{\varepsilon, 2  } y_{\varepsilon}A_{\varepsilon, 2  }-
			c 	A_{\varepsilon, 2  }=-(p y_{\varepsilon} p) A_{\varepsilon, 2  }-
			c 	A_{\varepsilon, 2  }
			\\ & =
			|p y_{\varepsilon} p | A_{\varepsilon, 2  }-
			c 	A_{\varepsilon, 2  }
			\leq  
			|p y_\varepsilon p|  A_{\varepsilon, 2  },
		\end{align} 
		we infer that 
		\begin{eqnarray}\label{absoluteprec}
			|x + p y_\varepsilon p |   \notag &\stackrel{\eqref{defcb}}{=}  &
			|c p+p y_\varepsilon p |
			\\  \notag	&\stackrel{\text{Prop.\ref{selfproperties} (i)}}{\prec\prec} &  
			|cA_{\varepsilon, 1} + A_{\varepsilon, 1}   y_{\varepsilon} A_{\varepsilon, 1} | + 
			| c A_{\varepsilon, 2  } +  A_{\varepsilon, 2  } y_{\varepsilon} A_{\varepsilon, 2  }| \\ & &\notag  +
			c B_{\varepsilon} + B_{\varepsilon}  y_{\varepsilon} B_{\varepsilon} 
			\\
			\notag &\stackrel{\text{\cite[Lem.2.3]{CDS97}}, \eqref{a2ya2absolute},\eqref{yA1prec3}}{\prec\prec} &
			c A_{\varepsilon, 1} + |p y_\varepsilon p| A_{\varepsilon, 2  }
			+  
			c B_{\varepsilon} + B_{\varepsilon} ( p y_{\varepsilon} p)B_{\varepsilon} 
             \\
			\notag & 
           = &  c A_{\varepsilon, 1} + |p y_\varepsilon p| A_{\varepsilon, 2  }+
			c B_{\varepsilon, 1}  + c B_{\varepsilon, 2} \\ 
            &\, &\notag  +   B_{\varepsilon, 1} (p y_{\varepsilon}  p)B_{\varepsilon, 1}   +   |p y_{\varepsilon} p|  B_{\varepsilon, 2}     
			\\
			&		 
			\stackrel{\eqref{defP3P}, \eqref{defQ3Q}}{= }  &   c P_{\varepsilon} + |p y_{\varepsilon} p| Q_{\varepsilon}
			+ B_{\varepsilon, 1}  y_{\varepsilon} B_{\varepsilon, 1}  
			+ c B_{\varepsilon, 2}.  
		\end{eqnarray}
		By the monotonicity of  $\norm{\cdot}_\cE$ with respect to $\prec \prec $, we obtain    
		\begin{align}
			2 - \varepsilon  & \notag \stackrel{\eqref{2-vare<normE}}{<} 
			\|x + p y_\varepsilon p\|_{\cE}  =	\| | c p + p y_\varepsilon p| \|_{\cE}    \\  &  \label{3-1} \stackrel{\eqref{absoluteprec}}{\leq} \|c P_{\varepsilon} + |p y_{\varepsilon} p| Q_{\varepsilon} \|_{\cE} + \norm{B_{\varepsilon, 1}   y_{\varepsilon} B_{\varepsilon, 1}  }_{\cE} + \norm{c B_{\varepsilon, 2}  }_{\cE}.   
		\end{align}

		Define
		\begin{align}\label{defgvare}g_{\varepsilon} := |p y_{\varepsilon} p| + \left( c - \frac{\tau(|p y_{\varepsilon} p|P_{\varepsilon} ) }{\tau(P_{\varepsilon})} \right) P_{\varepsilon}\in \cE. 
		\end{align}
		Then, 
		\begin{align*}
			\frac{\tau(g_{\varepsilon}P_{\varepsilon}) }{\tau(P_{\varepsilon})} \stackrel{\eqref{defgvare}}{=}
			\frac{\tau(|p y_{\varepsilon} p|P_{\varepsilon}) }{\tau(P_{\varepsilon})} 
			+ 	
			\Big(c
			- 
			\frac{\tau(|p y_{\varepsilon} p|P_{\varepsilon} ) }{\tau(P_{\varepsilon})} 
			\Big) \cdot 
			\frac{\tau(P_{\varepsilon}) }{\tau(P_{\varepsilon})}
			=c,   
		\end{align*}
		which shows that 
		\begin{align}\label{|g|Q=|pyp|Q}
			| g_{\varepsilon}| Q_{\varepsilon} +  \frac{\tau(g_{\varepsilon}P_{\varepsilon}) }{\tau(P_{\varepsilon})}  P_{\varepsilon} = |p y_{\varepsilon} p| Q_{\varepsilon}+ c P_{\varepsilon} .
		\end{align}	
		In particular, we have 
		\begin{align*}%\label{|g|Q+|tauP|2}
			|g_\varepsilon| Q_\varepsilon +  \frac{ \tau(
				g_\varepsilon  P_\varepsilon )  }{\tau(P_\varepsilon)}   P_\varepsilon
			& \notag =
			g_\varepsilon  Q_\varepsilon +  \frac{ \tau(
				g_\varepsilon  P_\varepsilon ) }{\tau(P_\varepsilon)}   P_\varepsilon 
			\stackrel{ \text{Lem.}\ref{fullyxp1+p_2prec}
			}	
			{\prec\prec }		g_\varepsilon  Q_\varepsilon +  	g_\varepsilon  P_\varepsilon
			=
			g_\varepsilon,   
		\end{align*}
		where we used the fact that  
		$
		0\leq   |g_\varepsilon| Q_\varepsilon \stackrel{\eqref{|g|Q=|pyp|Q}}{=} |p y_\varepsilon p| Q_\varepsilon \stackrel{\eqref{defgvare}}{=} g_\varepsilon Q_\varepsilon .   
		$  
		By the monotonicity of  $\norm{\cdot}_\cE$ with respect to $\prec \prec $, we obtain
		\[
		\left\|  |g_{\varepsilon}| Q_{\varepsilon} + \frac{\tau(|g_{\varepsilon}|P_{\varepsilon}) }{\tau(P_{\varepsilon})}  P_{\varepsilon} \right\|_\cE \leq \|g_{\varepsilon}\|_\cE,
		\]
		and thus,  		\begin{eqnarray}\label{3-31}
			& & 2 - \varepsilon  \notag \\ &\stackrel{\eqref{3-1}, \eqref{|g|Q=|pyp|Q}}{<}&  \|g_{\varepsilon}\|_\cE
			+
			\norm{B_{\varepsilon, 1}  y_{\varepsilon} B_{\varepsilon, 1}  }_{\cE} + \norm{c B_{\varepsilon, 2}  }_{\cE}\\ \notag 
		 &\stackrel{\eqref{defgvare}}{\leq}	&	\|p y_{\varepsilon} p\|_\cE
			+
			\left|  c - \frac{\tau(|p y_{\varepsilon} p|P_{\varepsilon} ) }{\tau(P_{\varepsilon})} \right|  \norm{ P_{\varepsilon}
			}_{\cE}
			+
			\norm{B_{\varepsilon, 1} y_{\varepsilon} B_{\varepsilon, 1}  }_{\cE} + \norm{c B_{\varepsilon, 2}  }_{\cE}
			\\ &\leq&		\| y_{\varepsilon} \|_\cE		+			\Big|  c 	- \frac{\tau(|p y_{\varepsilon} p|P_{\varepsilon} )}{\tau(P_{\varepsilon})}   	\Big| 
			\|P_{\varepsilon} \|_\cE
			+\norm{B_{\varepsilon, 1} y_{\varepsilon} B_{\varepsilon, 1}  }_{\cE} + \norm{c B_{\varepsilon, 2}  }_{\cE}.
		\end{eqnarray}

		Now, we claim that  \( c >  \frac{\tau(|p y_{\varepsilon} p|P_{\varepsilon} )}
		{\tau(P_{\varepsilon})}   
		\) for all numbers \( \varepsilon \) which are  small enough.  Assume by  contradiction that  there exists a sequence $\varepsilon_n\downarrow0$ such that
$$c\le\frac{\tau(|p y_{\varepsilon_n}p|P_{\varepsilon_n})}{\tau(P_{\varepsilon_n})}
\qquad n\ge1.$$
		Observe that,
        \begin{align*} \left( \frac{\tau(|p y_{\varepsilon_n} p|P_{\varepsilon_n} )}{\tau(P_{\varepsilon_n})} 
			- c \right) \|P_{\varepsilon_n} \|_\cE
			&\,\,\,\,\stackrel{\eqref{defcb}}{=} \left\|   \frac{\tau(|p y_{\varepsilon_n} p|P_{\varepsilon_n} )}{\tau(P_{\varepsilon_n})}  P_{\varepsilon_n} \right\|_{E} 
			- \frac{\phi_{\mathcal{E}}(\tau(P_{\varepsilon_n}))}{\phi_{\mathcal{E}}(a)}  
			\\ & \,\,\,\,\stackrel{\eqref{Qvarepsilona2}}{\leq}
			\left\|   \frac{\tau(|p y_{\varepsilon_n} p|P_{\varepsilon_n} )}{\tau(P_{\varepsilon_n})}  P_{\varepsilon_n} \right\|_{E} - \frac{\phi_{\mathcal{E}}(a/2)}{\phi_{\mathcal{E}}(a)}
			\\ & 
			\stackrel{\text{Lem.\ref{fullyxp1+p_2prec}}}{\leq}
			\|py_{\varepsilon_n} p\|_{E} - \frac{\phi_{\mathcal{E}}(a/2)}{\phi_{\mathcal{E}}(a)} \\ & \,\,\,\, \,\, \leq  1 - \frac{\phi_{\mathcal{E}}(a/2)}{\phi_{\mathcal{E}}(a)}.
		\end{align*}
		Then,   we have   
		\begin{align*}
			2 -{\varepsilon_n}  &\stackrel{\eqref{3-31}}{\leq}1 + 1 - \frac{\phi_{\mathcal{E}}(a/2)}{\phi_{\mathcal{E}}(a)}  + \norm{B_{{\varepsilon_n}, 1} y_{{\varepsilon_n}} B_{{\varepsilon_n}, 1}  }_{\cE} + \norm{c B_{{\varepsilon_n}, 2}  }_{\cE}. 
		\end{align*}
	However, by \eqref{yvarepsilonBvarepsilon1to0} and \eqref{cBvarepsilon2to0}, the above inequality fails when ${\varepsilon_n}  \to 0$. 
	Hence, the claim is proved. 
%     Thus,   there exists a constant  \( \varepsilon \) such that 
% 		\( c >  \frac{\tau(|p y_{\varepsilon} p|P_{\varepsilon} )}
% 		{\tau(P_{\varepsilon})}   
% 		\). 
% % {\color{blue}
% We claim that, for all sufficiently small $\varepsilon>0$,
% $$c>\frac{\tau(|p y_{\varepsilon}p|P_{\varepsilon})}{\tau(P_{\varepsilon})}.$$
% Indeed, otherwise there exists a sequence $\varepsilon_n\downarrow0$ such that
% $$c\le\frac{\tau(|p y_{\varepsilon_n}p|P_{\varepsilon_n})}{\tau(P_{\varepsilon_n})}
% \qquad n\ge1.$$
% Repeating the following argument along this sequence gives a contradiction.
% }
		Consequently, 
		\begin{align}\label{3-2}
			2 - \varepsilon &\notag \stackrel{\eqref{defcb}, \eqref{3-31}}{\leq}  	
			1 +
			\frac{\phi_{\cE}(\tau(P_\varepsilon)) 
			}{\phi_{\cE}(a)}
			- \frac{\phi_{\cE}(\tau(P_\varepsilon)) 
			}
			{\tau(P_{\varepsilon})} \tau(|p y_{\varepsilon} p|P_{\varepsilon} )  	\\ & \quad \,\,  \qquad 	+
			\norm{B_{\varepsilon, 1} y_{\varepsilon} B_{\varepsilon, 1}  }_{\cE} + \norm{c B_{\varepsilon, 2}  }_{\cE}
			.   
		\end{align}

		Since \( t \mapsto \phi_{\cE}(t)/t \) is decreasing on \( (0, \tau(\mathbf{1})) \) and \( 0<a/2 \stackrel{\eqref{Qvarepsilona2}}{\leq}  \tau(P_{\varepsilon}) \leq a \), we have
		\begin{align}\label{infphi}
			0 < %\frac{\phi_{\cE}(a/2)}{a}\leq
			\frac{\phi_{\cE}(a)}{a}  \leq \frac{\phi_{\cE}(\tau(P_{\varepsilon}))}{\tau(P_{\varepsilon})} \leq 2 \frac{\phi_{\cE}(a/2)}{a} 
			\leq 2 \frac{\phi_{\cE}(a)}{a} < \infty.
		\end{align}
		Passing to the limit 
		as $\varepsilon\to 0$ in \eqref{3-2}, it follows from \eqref{yvarepsilonBvarepsilon1to0}  and \eqref{cBvarepsilon2to0} that 
		\begin{align}\label{phitauP=1ctaupypP=0}		\lim_{\varepsilon \to 0} \phi_{\cE}(\tau(P_{\varepsilon})) = \phi_{\cE}(a) \quad \text{and} \quad 	\lim_{\varepsilon \to 0} \tau(|p y_{\varepsilon} p | P_{\varepsilon}) = 0.
		\end{align}	
		From  \eqref{F(yvare)>1-vare}, \eqref{yvareBvarei}  and \eqref{phitauP=1ctaupypP=0},  we have 
		\begin{align}\label{taupyvarepsilonpAvarepsilon2}
			\lim_{\varepsilon \to 0} \tau(|p y_{\varepsilon} p | A_{\varepsilon, 2  }) =  \frac{1}{b}
			\stackrel{\eqref{defcb}}{=} 
			\frac{a}{\phi_{\cE}(a)} . 
		\end{align}
		Consequently, \( \tau(A_{\varepsilon, 2  }) > 0 \) for small \( \varepsilon \). Since \begin{align}\label{3-4} 0 < \tau(A_{\varepsilon, 2  }) \stackrel{\eqref{defQ3Q}}{\leq} \tau(Q_{\varepsilon}) \stackrel{\eqref{Qvarepsilona2}}{\leq}   a/2 \stackrel{\eqref{Qvarepsilona2}}{\leq}  \tau(P_{\varepsilon}) ,
		\end{align}
		it follows that  \begin{eqnarray}\label{applicationp1p2normE}
			\frac{\phi_{\cE}(\tau(P_{\varepsilon}))}{\tau(P_{\varepsilon})} \tau(|p y_{\varepsilon} p | A_{\varepsilon, 2  })
			 \notag &\stackrel{\eqref{3-4}}{\leq}&  \frac{\phi_{\cE}(\tau(A_{\varepsilon, 2  }))}{\tau(A_{\varepsilon, 2  })} \tau(|p y_{\varepsilon} p | A_{\varepsilon, 2  }) \\
			 \notag  & =& \left\|  \frac{\tau(|p y_{\varepsilon} p | A_{\varepsilon, 2  } )}{\tau(A_{\varepsilon, 2  })}   A_{\varepsilon, 2  } \right\|_\cE
			\\
			&\stackrel{\text{Lem.\ref{fullyxp1+p_2prec} }}{\leq}&
			\| A_{\varepsilon, 2  } |p y_{\varepsilon} p|A_{\varepsilon, 2  } \|_\cE
			\leq \| y_{\varepsilon} \|_\cE \leq  1. 
		\end{eqnarray}
		Thus,  we have 
		\begin{align}\label{tauPvarea}
			\phi_{\cE}(\tau(P_{\varepsilon})) \tau(|p y_{\varepsilon} p | A_{\varepsilon, 2  })		\stackrel{\eqref{applicationp1p2normE}}{\leq}  \tau(P_{\varepsilon}) \le \tau(p)= a .
		\end{align} By \eqref{phitauP=1ctaupypP=0} and \eqref{taupyvarepsilonpAvarepsilon2},
		\[
		\lim_{\varepsilon \to 0} \phi_{\cE}(\tau(P_{\varepsilon})) \tau(|p y_{\varepsilon} p | A_{\varepsilon, 2  }) = \frac{\phi_{\cE}(a)}{b} 
		\stackrel{\eqref{defcb}}{=}  	 a.
		\]
		Combining the above equality with \eqref{tauPvarea}, 
		we obtain 	\begin{align}\label{tauPvarepsilontoa}
			\lim_{\varepsilon \to 0} \tau(P_{\varepsilon}) = a,
		\end{align}
		and hence (recall that  $Q_\varepsilon +P_\varepsilon=p$ and $\tau(p)=a$), 
        Hence, $\tau(Q_\varepsilon)=a-\tau(P_\varepsilon) \to 0$ as $\varepsilon \to 0.$ Since $A_{\varepsilon,2} \le Q_\varepsilon,$ it follows that
		\begin{align}\label{Avare2to0}
			\lim_{\varepsilon \to 0} \tau(A_{\varepsilon, 2  }) = 0.
		\end{align}	
		Again using \eqref{phitauP=1ctaupypP=0},   \eqref{taupyvarepsilonpAvarepsilon2} and \eqref{tauPvarepsilontoa}, we have
		\[
		\lim_{\varepsilon \to 0} \frac{\phi_{\cE}(\tau(P_{\varepsilon}))}{\tau(P_{\varepsilon})} \tau(|p y_{\varepsilon} p | A_{\varepsilon, 2  }) = 1.
		\]
		Thus,  $$\mathop{\lim}\limits_{\varepsilon  \to 0} \frac{\phi_{\cE}(\tau(A_{\varepsilon, 2  }))}{\tau(A_{\varepsilon, 2  })}\frac{1}{b} \stackrel{\eqref{taupyvarepsilonpAvarepsilon2}}{=}
		\mathop{\lim}\limits_{\varepsilon  \to 0}
		\frac{\phi_{\cE}(\tau(A_{\varepsilon, 2  }))}{\tau(A_{\varepsilon, 2  })} \tau(|p y_{\varepsilon} p | A_{\varepsilon, 2  }) \stackrel{\eqref{applicationp1p2normE}}{=}  1, $$ and therefore, 
		\[
		\lim_{\varepsilon \to 0} \frac{\phi_{\cE}(\tau(A_{\varepsilon, 2  }))}{\tau(A_{\varepsilon, 2  })} =b\stackrel{\eqref{defcb}}{=}  \frac{\phi_{\cE}(a)}{a}.
		\]	
		Therefore, by \eqref{infphi} and  \eqref{Avare2to0}, 
		we have  \( \mathop{\lim}\limits_{s \to 0} \frac{\phi_{\cE}(s)}{s} = \frac{\phi_{\cE}(a)}{a} \). 
		Taking \( a = 1 \) gives \( \mathop{\lim}\limits_{s \to 0} \frac{\phi_{\cE}(s)}{s} = 1 \). Hence,  for every \( 0 < a < \tau(\mathbf{1}) \),  \( \frac{\phi_{\cE}(a) }{ a} = \mathop{\lim}\limits_{s \to 0} \frac{\phi_{\cE}(s)}{s} =1 \). 
        By Lemma ~\ref{Lemma2.2}(ii), 
        \( \cE = L_1(\cM, \tau)_h   \) with equality of norms. 
		The proof is complete.  
	\end{proof}

	The next theorem is the main result of this section, which extends \cite[Theorem 4.4]{AKM15} to the context of uniformly monotone real  symmetric operator  spaces.

	\begin{theorem}\label{infiniteselfum}
		Let $(\cM, \tau)$ be an   infinite  semi-finite atomless von Neumann algebra and let
		\( (\cE, \norm{\cdot}_{\cE}) \subseteq S(\cM, \tau)_h \) be 	a uniformly monotone symmetric operator space.  
		If 	$\cE$ 
		has the Daugavet property, then it  is isometrically isomorphic to \( L_1(\cM, \tau)_h \).
	\end{theorem}
	
	\begin{proof}
		It suffices to show that $\cE$ satisfies the assumptions   of Theorem \ref{UM}.  
		Since  \( \cE \) is order continuous \cite[Lemma 2.5]{C12}, it follows that \( \cE \) is strongly symmetric \cite[Corollary 5.3.4]{DPS} and  \( \cE_p \) is order continuous for any \( p \in \cP(\cM) \) with $\tau(p)<\infty$. 
		Now let \( p \in \cP(\cM) \) with \( \tau(p) < \infty \), \( x, y_\varepsilon \in S_{\cE} \) for \( \varepsilon > 0 \), \( \|p y_\varepsilon p\|_{\cE}  \to 1 \) and \( \|x + y_\varepsilon \|_{\cE} \to 2 \) as \( \varepsilon \to 0 \).  
		By   Lemma \ref{Lem2.4}, 
		we obtain 
		\[ 	\|py_\varepsilon p\|_{\cE}^2 \leq \|y_\varepsilon\|_{\cE}   \|p |y_\varepsilon| p\|_{\cE}
		=\|p |y_\varepsilon| p\|_{\cE} 
		\leq  \| y_\varepsilon\|_{\cE}=1, 
		\]
		which implies  that 
		\begin{align*}
			\| |y_\varepsilon|^{\frac{1}{2}} p |y_\varepsilon|^{\frac{1}{2}} \|_{\cE}
			=\|p |y_\varepsilon| p\|_{\cE}  \to 1, \quad  \varepsilon \to 0.  
		\end{align*}

		%Since $|x+y_\varepsilon|\prec\prec |x|+|y_\varepsilon|$, it follows that \begin{align*}\norm{x+y_\varepsilon}_{\cE} \leq \norm{|x|+|y_\varepsilon|}_{\cE} \leq 	\norm{x}_{\cE}+\norm{y_\varepsilon}_{\cE}=2, \end{align*} and so, \begin{align*}\norm{|x|+|y_\varepsilon|}_{\cE}\to 2, \quad  \varepsilon \to 0.   \end{align*}

		Since 	\(\norm{\cdot}_{\cE}\)  is uniformly monotone and  by the similar argument of \eqref{xn-pxnpnormto0} in Theorem ~\ref{EpDPinherit}, we obtain 	\begin{align*}
			\lim\limits_{\varepsilon \to 0}\norm{ y_\varepsilon-
				p y_\varepsilon p 	}_{\cE}=0. 
		\end{align*}
		It follows from 
		\begin{align*}
			2 & \geq \|x + p y_\varepsilon p\|_{\cE}
			= \|x +y_\varepsilon-y_\varepsilon+ p y_\varepsilon p\|_{\cE} \geq 
			\|x +y_\varepsilon\|_{\cE}-
			\|y_\varepsilon- p y_\varepsilon p\|_{\cE}.
		\end{align*} 
		that
		\begin{align*}
			\mathop{\lim}\limits_{\varepsilon \to 0} \|x + p y_\varepsilon p\|_{\cE}=  2 , 
		\end{align*} 
		which completes the proof.  \end{proof}

	It is worth noting that property (b) in Theorem \ref{UM} is not equivalent
	to the uniform monotonicity of $\cE$. The latter one  is stronger than (b)    \cite[Remark 4.5]{AKM15}.

\section{Operator Daugavet property for the predual of an atomless von Neumann algebra}\label{sec7}
In this section, we establish an operator version of the Daugavet property as defined in \cite[Section 4]{RTV} for the predual of an atomless von Neumann algebra, which is a noncommutative analogue of \cite[Proposition 4.4]{RTV}. We will not require that a von Neumann algebra is equipped with a trace $\tau.$ It should be noted that the operator Daugavet property implies the Daugavet property \cite[Remark 4.2]{RTV}.

Throughout this section, $\mathcal{M}$ is a von Neumann algebra. We regard its predual $\mathcal{M}_*$ as the canonical isometric subspace of $\mathcal{M}^*$ consisting of normal linear functionals on $\mathcal{M}$. Thus, for $\omega\in\mathcal{M}_*$ and $a\in\mathcal{M}$, the notation $\omega(a)$ means the usual evaluation of the normal functional $\omega$ at $a$. Equivalently, under the canonical duality
$$\langle \omega,a\rangle=\omega(a),\qquad \omega\in\mathcal{M}_*,\ a\in\mathcal{M},$$
one has $(\mathcal{M}_*)^*=\mathcal{M}$ isometrically \cite[Theorem 1.11.5]{DPS}. We write $\norm{\cdot}$ for the norm of $\mathcal{M}_*$, $\norm{\cdot}_{\mathcal{M}}$ for the operator norm.  
For $\omega\in\mathcal{M}_*$, we write $|\omega|$ for the positive normal functional appearing in the polar decomposition of $\omega$ \cite[Theorem III.4.2]{Takesakibook}. Thus,  there is a partial isometry $v\in\mathcal{M}$ such that
$$\omega=v|\omega|,\qquad\text{that is,}\qquad\omega(x)=|\omega|(xv) \qquad |\omega|(x)=\omega(xv^*),\quad x\in\mathcal{M}.$$
The functional $|\omega|\in\mathcal{M}_*^+$ is called the modulus of $\omega$ and is uniquely determined by $\omega$; moreover $\|\omega\|=|\omega|({\bf 1})$. Hence, for a projection $e\in\mathcal{M}$, the quantity $|\omega|(e)$ means the value of this positive normal functional at $e$.  In particular,  $|\omega|(e)\geq0$.

For $\omega\in\mathcal{M}_*$ and $y\in\mathcal{M}$, we write $\omega y$ and $y\omega$ for the normal functionals \cite[Section 1.12]{DPS}
$$(\omega y)(x):=\omega(yx),\qquad (y\omega)(x):=\omega(xy),\qquad x\in\mathcal{M}.$$
Thus,  $\omega y,y\omega\in\mathcal{M}_*$ and
$$\|\omega y\|\leq\|\omega\|\|y\|_{\mathcal{M}},\qquad \|y\omega\|\leq\|y\|_{\mathcal{M}}\|\omega\|.$$
More generally, we make $\mathcal{M}_*$ into an $\mathcal{M}$-bimodule by
$$(a\omega b)(x):=\omega(bxa),\qquad a,b,x\in\mathcal{M},\ \omega\in\mathcal{M}_*,$$
so that
$$\|a\omega b\|\leq\|a\|_{\mathcal{M}}\|\omega\|\|b\|_{\mathcal{M}}.$$

For projections $f,g\in\mathcal{M}$, the rectangular corner $g\mathcal{M}_*f$ is the predual of the rectangular corner $f\mathcal{M}g$, with the duality
$$\langle g\omega f,m\rangle=\omega(m),\qquad m=fmg\in f\mathcal{M}g.$$
Indeed, this follows from the canonical predual identification $(\mathcal{M}_*)^*=\mathcal{M}$ and the usual reduction of von Neumann algebras by projections \cite[Theorem 1.11.5 and Theorem 1.11.11]{DPS}. In particular, for every $a\in\mathcal{M}$,
$$\|fag\|_{\mathcal{M}}=\sup\{|\omega(a)|:\omega\in g\mathcal{M}_*f,\ \|\omega\|\leq1\}.$$	
	
Recall that a slice $S$ of $B_{\mathcal{M}_*}$ is a set of the form
$$S=S(a,\alpha)=\{\omega\in B_{\mathcal{M}_*}:\mathrm{Re}\,  \omega(a)>1-\alpha\},\quad a\in S_{\mathcal{M}}, \ \alpha>0.$$

\begin{definition}
 \upshape
A complex Banach space $X$ has the \emph{operator Daugavet property} if for every $n\in\mathbb{N}$, all $x_1,\ldots,x_n\in S_X$, every slice $S$ of $B_X$, and every $\varepsilon>0$, there exists $x\in S$ such that for every $x_0\in B_X$ there is a bounded linear operator $T:X\to X$ with
$$\|T\|\leq1+\varepsilon,\qquad Tx=x_0,\qquad \|Tx_i-x_i\|<\varepsilon,\quad i=1,\ldots,n.$$
\end{definition}

We will need the following auxiliary lemmas.
\begin{lemma}\label{lem:domination}
Let $\omega\in\mathcal{M}_*$ and let $e\in\mathcal{M}$ be a projection. Then, 
$$\|\omega e\|^2\leq\|\omega\| |\omega|(e),\qquad \|e\omega\|^2\leq\|\omega\| |\omega^*|(e),$$
where $\omega^*(x):=\overline{\omega(x^*)}$.
\end{lemma}

\begin{proof}
Let $\omega=v|\omega|$ be the polar decomposition of $\omega$ in $\mathcal{M}_*$, where $v\in\mathcal{M}$ is a partial isometry with $v^*v=s(|\omega|)$ and $\omega(x)=|\omega|(xv)$ for $x\in\mathcal{M}$. Let $\|x\|_{\mathcal{M}}\leq1$.

We have $(\omega e)(x)=\omega(ex)=|\omega|(exv)$, and by the Cauchy--Schwarz inequality for the positive normal functional $|\omega|$,
$$|(\omega e)(x)|^2=\bigl||\omega|(e\cdot xv)\bigr|^2\leq|\omega|(e)\,|\omega|(v^*x^*xv)\leq|\omega|(e)|\omega|(v^*v)=\|\omega\| |\omega|(e),$$
using $x^*x\leq{\bf 1}$ and $|\omega|(v^*v)=|\omega|(s(|\omega|))=|\omega|({\bf 1})=\||\omega|\|=\|\omega\|$ \cite[Proposition III.4.6]{Takesakibook}. Taking the supremum over $x$ gives the first inequality.

On the other hand, $(e\omega)(x)=\omega(xe)=|\omega|(xev)$, and
$$|(e\omega)(x)|^2=\bigl||\omega|(x\cdot ev)\bigr|^2\leq|\omega|(xx^*)| \omega|(v^*ev)\leq\|\omega\| |\omega|(v^*ev).$$
Since $|\omega^*|=v|\omega|v^*$, one has $|\omega|(v^*ev)=(v|\omega|v^*)(e)=|\omega^*|(e)$, and the second inequality follows.
\end{proof}

\begin{lemma}\label{lem:corner}
Let $\mathcal{M}$ be atomless,  let $0\neq p\in\mathcal{M}$ be a projection, let $u\in\mathcal{M}$ be a partial isometry with $p\leq u^*u$, and let $\omega_1,\ldots,\omega_n\in\mathcal{M}_*$. Then for every $\delta>0$ there is a non-zero projection $d\leq p$ such that, putting $r:=udu^*$,
$$\|d\omega_i\|<\delta,\qquad \|\omega_i r\|<\delta,\qquad i=1,\ldots,n.$$
\end{lemma}

\begin{proof}
By Lemma \ref{lem:domination}, for each $i$ and every projection $h\leq p$,
$$\|h\omega_i\|^2\leq\|\omega_i\|\,|\omega_i^*|(h),\qquad \|\omega_i\,uhu^*\|^2\leq\|\omega_i\|\,|\omega_i|(uhu^*).$$
Define positive normal functionals $\sigma_i$ and $ \rho_i$ on the corner $p\mathcal{M}p$ by
$$\sigma_i(f):=|\omega_i^*|(f),\qquad \rho_i(f):=|\omega_i|(ufu^*),\qquad f\in p\mathcal{M}p,$$
and set $\Phi:=\sum_{i=1}^n(\sigma_i+\rho_i)\in(p\mathcal{M}p)_*^+$. Choose $\gamma>0$ with $\|\omega_i\|\gamma<\delta^2$ for all $i$.

Since $\mathcal{M}$ is atomless, the corner $p\mathcal{M}p$ is diffuse, hence has no minimal projections, so for every $N\in\mathbb{N}$ the projection $p$ splits as a sum $p=p_1+\cdots+p_N$ of non-zero pairwise orthogonal projections in $p\mathcal{M}p$. Fix $N$ with $N \cdot \gamma>\Phi(p)$. If $\Phi(p_k)\geq\gamma$ for every $k$, summing would give $\Phi(p)\geq N \cdot  \gamma>\Phi(p)$, a contradiction; so there is $k$ with $\Phi(p_k)<\gamma$. Put $d:=p_k$. Then $d\leq p\leq u^*u$, so $r:=udu^*$ is a projection, and
$$\|d\omega_i\|^2\leq\|\omega_i\| \sigma_i(d)\leq\|\omega_i\| \Phi(d)<\|\omega_i\|\gamma<\delta^2,$$
$$\|\omega_i r\|^2=\|\omega_i udu^*\|^2\leq\|\omega_i\| \rho_i(d)\leq\|\omega_i\| \Phi(d)<\delta^2.$$
This completes the proof.
\end{proof}

We quote 
the following fact for later reference, see e.g. \cite[Lemma 2.1]{P02}. 
\begin{lemma}\label{lem:orthogonal}
Let $d,r\in\mathcal{M}$ be projections and put $q_d:={\bf 1}-d$, $q_r:={\bf 1}-r$. If
$$\theta_1\in q_d\mathcal{M}_*q_r,\qquad \theta_2\in d\mathcal{M}_*r,$$
then $\|\theta_1+\theta_2\|=\|\theta_1\|+\|\theta_2\|$.
\end{lemma}

% \begin{proof}
% The inequality $\|\theta_1+\theta_2\|\leq\|\theta_1\|+\|\theta_2\|$ is the triangle inequality. For the reverse, let $\varepsilon>0$ and choose
% $$b_1\in q_r\mathcal{M}q_d,\qquad b_2\in r\mathcal{M}d,\qquad \|b_1\|_{\mathcal{M}}\leq1,\ \|b_2\|_{\mathcal{M}}\leq1,$$
% which, after multiplying by suitable unimodular scalars, satisfy
% $$\mathrm{Re} \theta_1(b_1)>\|\theta_1\|-\varepsilon,\qquad \mathrm{Re}\theta_2(b_2)>\|\theta_2\|-\varepsilon.$$
% Put $b:=b_1+b_2$. Since $b_1^*b_2=0=b_2^*b_1$, we have $b^*b=b_1^*b_1+b_2^*b_2$, and the positive operators $b_1^*b_1\in q_d\mathcal{M}q_d$ and $b_2^*b_2\in d\mathcal{M}d$ live in orthogonal corners, so
% $$\|b\|_{\mathcal{M}}^2=\|b^*b\|_{\mathcal{M}}=\max\{\|b_1^*b_1\|_{\mathcal{M}},\|b_2^*b_2\|_{\mathcal{M}}\}\leq1.$$
% Moreover $\theta_1(b_2)=0$ (since $q_r b_2=0$) and $\theta_2(b_1)=0$ (since $r b_1=0$). Therefore,
% $$\|\theta_1+\theta_2\|\geq\mathrm{Re}(\theta_1+\theta_2)(b)=\mathrm{Re}\theta_1(b_1)+\mathrm{Re}\theta_2(b_2)>\|\theta_1\|+\|\theta_2\|-2\varepsilon.$$
% Letting $\varepsilon\to0$ completes the proof.
% \end{proof}

Now we are ready to prove the main result of this section.
\begin{theorem}\label{thm:ODP}
Let $\mathcal{M}$ be an atomless von Neumann algebra. Then its predual $\mathcal{M}_*$, as a (complex) Banach space, has the operator Daugavet property. %In particular, this holds when $\mathcal{M}$ is of type III.
\end{theorem}

\begin{proof}
Use $(\mathcal{M}_*)^*=\mathcal{M}$. Let $\omega_1,\ldots,\omega_n\in S_{\mathcal{M}_*}$, let
$$S=\{\omega\in B_{\mathcal{M}_*}:\mathrm{Re}\,   \omega(a)>1-\alpha\}$$
be a slice of $B_{\mathcal{M}_*}$, where $a\in S_{\mathcal{M}}$ and $\alpha>0$. 
Assume that $\varepsilon>0$. Choose
\begin{align}\label{etachoose}
0<\eta<\min\{\alpha,\varepsilon/3,1\}.
\end{align}
Let $a=u|a|$ be the polar decomposition of $a$ in $\mathcal{M}$ \cite[Theorem 1.7.3]{DPS}. Since $\|a\|_{\mathcal{M}}=1$, the spectral projection
$$p=e^{|a|}(1-\eta/2,1]$$
is non-zero, and $p\leq s(|a|)=u^*u$.

By Lemma \ref{lem:corner} (with $\delta=\eta$),
there is a non-zero projection $d\leq p$ such that, with $r:=udu^*$,
\begin{align}\label{normdomegai}
    \|d\omega_i\|<\eta,\qquad \|\omega_i r\|<\eta,\qquad i=1,\ldots,n.
\end{align}

Next, we construct $z\in S.$ Since $d\leq p\leq u^*u$, it follows that 
$$rad=udu^*\,u|a|\,d=ud|a|d,\qquad (rad)^*(rad)=d|a|d\,u^*u\,d|a|d=(d|a|d)^2,$$
so $\|rad\|_{\mathcal{M}}=\|d|a|d\|_{\mathcal{M}}$. As $d\leq p=e^{|a|}(1-\eta/2,1]$,
we have $d|a|d\geq(1-\eta/2)d$, and $d\neq0$, hence
$$\|rad\|_{\mathcal{M}}\geq1-\eta/2>1-\eta.$$
Since $d\mathcal{M}_*r$ is the predual of $r\mathcal{M}d$ and $rad\in r\mathcal{M}d$, we have
$$\|rad\|_{\mathcal{M}}=\sup\{|z(rad)|:z\in d\mathcal{M}_*r,\ \|z\|=1\}.$$
As $\|rad\|_{\mathcal{M}}>1-\eta$, there is $z\in d\mathcal{M}_*r$ with $\|z\|=1$ and
$$|z(rad)|>1-\eta.$$
Multiplying $z$ by a suitable unimodular scalar, we may assume
$$\mathrm{Re} \,  z(rad)>1-\eta.$$
%One may rotate the complex number $z_0(rad)$ onto the positive real axis. (Since $|z_0(rad)|>1-\eta$, put $c:=z_0(rad)$. Then $c\neq0$. Replacing $z_0$ by
%$$\lambda z_0,\qquad \lambda:=\frac{\overline{c}}{|c|},$$
%does not change its norm and gives
%$$\Re(\lambda z_0)(rad)=|c|>1-\eta.$$
%Thus, after multiplying $z_0$ by a suitable unimodular scalar, we may assume
%$$\Re z_0(rad)>1-\eta.$$)
%If necessary replacing $z_0$ by $z:=z_0/\|z_0\|$, we obtain $z\in d\mathcal{M}_*r$ such that
%$$\|z\|=1,\qquad \mathrm{Re} z(rad)>1-\eta.$$
Since $z=dzr$, we have $z(a)=z(rad)$, so
$$\mathrm{Re}\, z(a)>1-\eta>1-\alpha,\qquad\text{i.e.}\qquad z\in S.$$

Now we are ready to define $T.$ Fix an arbitrary $x_0\in B_{\mathcal{M}_*}$. By the complex Hahn--Banach theorem,  there is a (complex-linear) functional $\psi\in(d\mathcal{M}_*r)^*$ with
$$\|\psi\|=1,\qquad \psi(z)=1.$$
Put $q_d:={\bf 1}-d$, $q_r:={\bf 1}-r$, and define the complex-linear operator $T:\mathcal{M}_*\to\mathcal{M}_*$ by
$$T\omega=q_d\omega q_r+\psi(d\omega r)x_0,\qquad \omega\in\mathcal{M}_*.$$

Since $z=dzr$ we have $q_d z q_r=0$, and $\psi(dzr)=\psi(z)=1$, hence $Tz=x_0$.

We show $\|T\|\leq1$. Define $P:\mathcal{M}_*\to\mathcal{M}_*$ by $P\omega=q_d\omega q_r+d\omega r$. 
Its adjoint is $P^*:\mathcal{M}\to\mathcal{M}$, $P^*b=q_rbq_d+rbd$.
%Under the identification $(\mathcal{M}_*)^*=\mathcal{M}$, its adjoint is determined by
%$$(P\omega)(b)=\omega(P^*b),\qquad \omega\in\mathcal{M}_*,\ b\in\mathcal{M}.$$
%Using the bimodule convention, we have
%$$(q_d\omega q_r)(b)=\omega(q_rbq_d),\qquad (d\omega r)(b)=\omega(rbd).$$
%Thus
%$$(P\omega)(b)=\omega(q_rbq_d+rbd),\qquad \omega\in\mathcal{M}_*,\ b\in\mathcal{M},$$
%and therefore
%$$P^*b=q_rbq_d+rbd,\qquad b\in\mathcal{M}.$$
Setting the self-adjoint unitaries $s_d:=q_d-d$ and $s_r:=q_r-r$, one has
$$P^*b=\tfrac12(b+s_r b s_d),\qquad b\in\mathcal{M}.$$

Since $b\mapsto s_rbs_d$ is an isometry of $\mathcal{M}$, $\|P^*b\|_{\mathcal{M}}\leq\|b\|_{\mathcal{M}}$, so $\|P\|\leq1$ (see \cite[Theorem 4.10]{Rudin}).
%$\|s_rbs_d\|_{\mathcal{M}}\le \|b\|_{\mathcal{M}},$ and since
%$b=s_r^*(s_rbs_d)s_d^*$ we have
%$$\|b\|_{\mathcal{M}}\leq\|s_r^*\|_{\mathcal{M}}\|s_rbs_d\|_{\mathcal{M}}\|s_d^*\|_{\mathcal{M}}=\|s_rbs_d\|_{\mathcal{M}}.$$
For $\omega\in\mathcal{M}_*$, the two functionals $q_d\omega q_r$ and $d\omega r$ lie in the orthogonal rectangular corners $q_d\mathcal{M}_*q_r$ and $d\mathcal{M}_*r$; and so by Lemma \ref{lem:orthogonal},
$$\|q_d\omega q_r\|+\|d\omega r\|=\|q_d\omega q_r+d\omega r\|=\|P\omega\|\leq\|\omega\|.$$
Therefore, using $\|\psi\|=1$ and $\|x_0\|\leq1$, we have 
$$\|T\omega\|\leq\|q_d\omega q_r\|+|\psi(d\omega r)|\,\|x_0\|\leq\|q_d\omega q_r\|+\|d\omega r\|\leq\|\omega\|,$$
so $\|T\|\leq1$.

Finally, we show that $\|T\omega_i-\omega_i\|<\varepsilon$ for all $i$. For each $i$,
$$T\omega_i-\omega_i=(q_d\omega_i q_r-\omega_i)+\psi(d\omega_i r)x_0.$$
Now $\omega_i-q_d\omega_i q_r=d\omega_i+q_d\omega_i r$, so
$$\|q_d\omega_i q_r-\omega_i\|\leq\|d\omega_i\|+\|q_d\omega_i r\|\leq\|d\omega_i\|+\|\omega_i r\|.$$
Together with $|\psi(d\omega_i r)|\leq\|d\omega_i r\|\leq\|\omega_i r\|$ this gives
$$\|T\omega_i-\omega_i\|\leq\|d\omega_i\|+2\|\omega_i r\|
\stackrel{\eqref{normdomegai}}{<}
3\eta
\stackrel{\eqref{etachoose}}{<}\varepsilon.$$

We have constructed $z\in S$ such that for every $x_0\in B_{\mathcal{M}_*}$ there is a bounded (complex-linear) operator $T$ with $\|T\|\leq1$, $Tz=x_0$ and $\|T\omega_i-\omega_i\|<\varepsilon$ for all $i$. 
Hence, $\mathcal{M}_*$ has the operator Daugavet property.
\end{proof}

	\appendix
	
	\section{ Basic sequences and $\ell_1$-copies  of  Banach spaces }\label{AppendixA}

	We follow the standard terminology and notation from  \cite{
		Dis84, LT1, LT2}. 
	The definitions and results presented herein hold for both   real and complex scalars. For convenience, we include some facts on basic sequences and asymptotically isometric copies of $\ell_1$ used above.
	
	Let  \( (X, \norm{\cdot}) \) be a  Banach   space. 
	A sequence \(\{x_n\}_{n=1}^\infty \subseteq X\)  is called a \emph{Schauder basis} (or \emph{basis}) for \(X\) if for  each \(x \in X\),  there exists a unique sequence \(\{a_n\}_{n=1}^\infty\) of scalars such that
	\[
	x = \lim_{n} \mathop{\sum}\limits_{k=1}^{n} a_k x_k.
	\]
	A \emph{basic sequence} in  \(X\) is a sequence \(\{x_n\}_{n=1}^\infty\) that is a basis for its norm closed linear span \cite[page 32]{Dis84}. 
	Let \(\{x_n\}_{n=1}^\infty \subseteq X \) be a basis, \(\{r_n\}_{n=1}^\infty\) and \(\{s_n\}_{n=1}^\infty\) be intertwining sequences of positive integers (i.e., \(r_1 < s_1 < r_2 < s_2 < \cdots\)), and \(y_n = \mathop{\sum}\limits_{i=r_n}^{s_n} a_i x_i\) be nontrivial linear combinations of the \(x_i\); we call the sequence \(\{y_n\}_{n=1}^\infty\) a \emph{block basic sequence taken with respect to} \(\{x_n\}_{n=1}^\infty\), or simply a \emph{block basic sequence} \cite[page 46]{Dis84}. 
	
	%\cite[Theorem 1 of chapter V]{Dis84} shows that \(\{y_n\}_{n=1}^\infty\) is basic \cite[page 46]{Dis84}.   
	
	% We say that two sequences \(\{x_n\}_{n=1}^\infty \) and \(\{y_n\}_{n=1}^\infty \) in a  Banach space $(X, \norm{\cdot})$ are 	{\it congruent}	 if there is an invertible operator \(T : X \to X\) such that \(T(x_n) = y_n\) for all \(n \in \mathbb{N}\) \cite[Definition 1.3.8]{AK06}. A classical  result dating back to 1940 \cite{KMR40}. It says, roughly speaking, that if $\{x_n\}_{n=1}^\infty\subseteq  X$ is a basic sequence and $\{y_n\}_{n=1}^\infty$ is another sequence in $X$ so that $\|x_n - y_n\| \to 0$ fast enough, then $\{y_n\}_{n=1}^\infty$ and $\{x_n\}_{n=1}^\infty$ are congruent.

	Two   basic sequences \(\{x_n\}_{n=1}^\infty\) and \(\{y_n\}_{n=1}^\infty\) in the respective Banach spaces \((X, \norm{\cdot}_{X})\) and \((Y, \norm{\cdot}_{Y})\) are {\it isometrically equivalent} \cite[Corollary 1.3.3]{AK06},
	if and only if 
	for all finitely non-zero sequences of scalars \(\{a_{n}\}_{n=1}^{\infty}\),  we have	
	\[
	\left\| \mathop{\sum}\limits_{n=1}^{\infty} a_{n} y_{n} \right\|_Y 
	=
	\left\| \mathop{\sum}\limits_{n=1}^{\infty} a_{n} x_{n} \right\|_X . 
	\]	 
	%\end{cor}
	%\begin{theorem}\cite[Lemma  2.1]{J64}\label{isomorphiccopy}If a Banach space  \( (X, \norm{\cdot}) \) 	 contains an isomorphic copy of \( \ell^1 \) and if \( \varepsilon > 0 \), then there exists a sequence \( \{ x_n\}_{n=1}^{\infty} \) in \( X \) so that  	\[	(1 - \varepsilon) \mathop{\sum}\limits_{n=1}^\infty |a_n| \leq \left\| \mathop{\sum}\limits_{n=1}^\infty a_n x_n \right\| \leq \mathop{\sum}\limits_{n=1}^\infty |a_n|,	\]	for all \( \{a_n\}_{n=1}^{\infty} \in \ell^1 \).\end{theorem}
	%Let \(\{x_n\}_{n=1}^{\infty} \) be a basis for \(X\) and \(\{y_n\}_{n=1}^{\infty} \) be a basis for \(Y\). 
	%Moreover, \(\{x_n\}_{n=1}^{\infty} \sim \{y_n\}_{n=1}^{\infty}\)  if and only if there is an isomorphism between \(X\) and \(Y\) that carries each \(x_n\) to \(y_n\) \cite[Theorem 5, page 43]{Dis84}.

	Let $L$ be any  sequence space  or function space. 
	Let $\widetilde{X}$ be a  Banach space. 
	If \( X \) is a closed subspace of $\widetilde{X}$, then \( X \) is said to be a {\it copy of \( L \)} if and only if there exists an (Banach space) isomorphism \( T \) from \( L \) onto \( X \). In other words,  
	we say that   
	%the Banach space 
	$\widetilde{X}$ contains a {\it copy
		of}  $L$. 
	If  the linear operator  \( T : L \to X \) is  an (Banach space) isometric, then we say that 
	\( \widetilde{X} \) contains an {\it isometric copy of \( L \)}.  
	If the isomorphism \( T: L \to X \) satisfies \( T(z) \geq 0 \) whenever \( 0 \leq z \in L \), then \( T \) will be called a {\it positive isomorphism} and the range \( X \) will be termed a {\it positive copy of \( L \)}.
	If, in addition, 
	the positive isomorphism 
	satisfies \( z \geq 0 \) whenever \(  T(z) \geq  0 \), then \( T \) will be called an {\it order isomorphism} and the range \( X \) will be termed an {\it order isomorphism copy of \( L \)}, see e.g. \cite[page 381]{DPS}.

	A bounded sequence \( \{x_n\}_{n=1}^{\infty} \) in a Banach space \( (X, \norm{\cdot}) \) is {\it   equivalent to the usual \( \ell_1 \)-basis} if there is a \( \varepsilon > 0 \) so that for all \( n \) and choices of scalars \( a_1, \ldots, a_n \),
	\begin{align*}%\label{copyl1def}
		\varepsilon \mathop{\sum}\limits_{k=1}^n |a_k| \leq \left\| \mathop{\sum}\limits_{k=1}^n a_k x_k \right\|.
	\end{align*}
	In this case, 
	the norm closed subspace spanned by \( \{x_n\}_{n=1}^{\infty} \) 
	is a copy of $\ell_1$, see e.g.  \cite{R74, H76} and \cite[page 201]{Dis84}.  
	
	The study of asymptotically isometric copies of \(\ell_1\) was initiated by Hagler \cite[page  14]{HaglerThesis}.  
	%\begin{definition}\cite[Definition 1]{DR00}	A Banach space \( (X, \norm{\cdot}) \) is said to contain an asymptotically isometric copy of \(\ell^1\) if there exists a null sequence \(\{\varepsilon_n\}_{n=1}^{\infty} \) in \((0,1)\) and a sequence \(\{x_n\}_{n=1}^{\infty}\) in \(X\) such that \[\mathop{\sum}\limits_{n=1}^{\infty} (1 - \varepsilon_n) |a_n| \leq \Bigl\|\mathop{\sum}\limits_{n=1}^{\infty} a_n x_n\Bigr\| \leq \mathop{\sum}\limits_{n=1}^{\infty} |a_n|,	\]	for every \(\{a_n\}_{n=1}^{\infty} \in \ell^1\).	The closed linear span of $\{x_n\}_{n=1}^{\infty}$ is called an asymptotically isometric copy of $\ell^1$.\end{definition}
	The following lemma, drawn from \cite[Lemma 1]{DGH00}, \cite[Definition 1 and Proposition 1]{DR00}  and \cite[Corollary 4]{DR00two}, gives equivalent conditions for a Banach space to contain asymptotically isometric copies of 
	\(\ell_1\).

	\begin{lemma}%\cite[Lemma 1]{DGH00} \cite{DR00} \cite[Corollary 4]{DR00two}
		\label{equivalentconditionsail1}
		For  a  Banach space \((X, \norm{\cdot})\), the following  are equivalent:	
		\begin{enumerate}[\rm(i)]
			\item There exist a null sequence \(\{ \varepsilon_n\}_{n=1}^{\infty}\) of positive numbers in \((0,1)\)  and a sequence \(\{x_n\}_{n=1}^{\infty}\) in \(X\) such that
			\[
			\mathop{\sum}\limits_{n=1}^{m}(1-\varepsilon_n)|a_n|\leq\left\|\mathop{\sum}\limits_{n=1}^{m}a_nx_n\right\|\leq\mathop{\sum}\limits_{n=1}^{m}|a_n|
			\]
			for all  \(\{a_n\}_{n=1}^{m}\in \ell_1\), where 
			$m\in \mathbb{N}\cup \{\infty\}$.
			
			\item There exist a null sequence \(\{\varepsilon_n\}_{n=1}^{\infty}\) of positive numbers in \((0,1)\)  and a sequence \(\{x_n\}_{n=1}^{\infty}\) in \(X\) such that
			\[
			\mathop{\sum}\limits_{n=1}^{m}|a_n|\leq\left\|\mathop{\sum}\limits_{n=1}^{m}a_nx_n\right\|\leq\mathop{\sum}\limits_{n=1}^{m}(1+\varepsilon_n)|a_n|
			\]
			for all  \(\{a_n\}_{n=1}^{m}\in \ell_1\), where 
			$m\in \mathbb{N}\cup \{\infty\}$.
			
			\item There exist a null sequence \(\{\varepsilon_n\}_{n=1}^{\infty}\) of positive numbers in \((0,1)\) and a sequence \(\{x_n\}_{n=1}^{\infty}\) in \(X\) such that
			\[
			\mathop{\sum}\limits_{n=k}^{m}|a_n|\leq\left\|\mathop{\sum}\limits_{n=k}^{m}a_nx_n\right\|\leq(1+\varepsilon_k)\mathop{\sum}\limits_{n=k}^{m}|a_n|
			\]
			for all  \(\{a_n\}_{n=k}^{m}\in \ell_1\),  where  \(k\in\mathbb{N}\) and 
			$m\in \mathbb{N}\cup \{\infty\}$.
			
		\end{enumerate}
	\end{lemma}  
	%Note that each condition is equivalent to the variant obtained by replacing ‘There exist a’ by ‘For each’ and ‘and’ by ‘there exists.’ The norm closed linear span of $\{x_n\}_{n=1}^{\infty}$ is called an  {\it asymptotically isometric copy of $\ell_1$}. 
	For simplicity, we  call
	a sequence \(\{x_n\}_{n=1}^{\infty}\)  satisfying one of the conditions in the lemma 
	an  \emph{asymptotically isometric copy} of $\ell_1$. Note that every asymptotically isometric copy of $\ell_1$ is also a basic sequence \cite[Proposition 1.1.9]{AK06}.

	For completeness, we provide full proofs of the following results, which are probably well-known to experts. 
	\begin{lemma}\label{blockbasisasyl1}
		If a  basic sequence $\{
		x_k
		\}_{k=1}^\infty$  in a  %real/complex  
		Banach space \((X, \norm{\cdot})\) is an asymptotically isometric  copy of
		\( \ell_1 \), then a normalized  block basic sequence  
		$\{
		y_n
		\}_{n=1}^\infty$ taken with respect             to $\{x_k\}_{k=1}^\infty$  is also an asymptotically isometric  copy of
		\( \ell_1 \). 
	\end{lemma}
	\begin{proof}
		Suppose that 
		\(\{x_{k}\}_{k=1}^\infty\) satisfies Lemma ~\ref{equivalentconditionsail1} (iii) and  
		$y_n:= \mathop{\sum}\limits_{k=p_{n}}^{q_n} b_k x_k$. %where $b_k$ is a  scalar,  \(\{p_n\}_{n=1}^\infty\) and \(\{q_n\}_{n=1}^\infty\) are  intertwining sequences of finite positive integers.  
		There exists a null sequence \(\{\varepsilon_{k}\}_{k=1}^\infty \) of positive numbers  such that 
		\begin{align}\label{znnorm1}
			\mathop{\sum}\limits_{k=p_{n}}^{q_n}|b_{k}|\leq 
			\norm{y_n}
			=
			\left\|
			\mathop{\sum}\limits_{k=p_{n}}^{q_n} b_k x_k
			\right\|=1
			\leq
			(1+\varepsilon_{p_{n}})
			\mathop{\sum}\limits_{k=p_{n}}^{q_n}
			|b_{k}|, \quad \forall n\geq 1. 
		\end{align} 
		%which implies that    \begin{align}\label{sumckineq}	\frac{1}{	1+\varepsilon_{p_{n}}}	\leq\mathop{\sum}\limits_{k=p_{n}}^{q_n}|b_{k}| 	\leq  1.\end{align} 
		For any  sequence  of scalars \(\{a_{n}\}_{n=1}^{m}\in \ell_1\),    we have 
		\begin{align*}
			&\nonumber 
			\mathop{\sum}\limits_{n=1}^{m} \Big(1- \frac{\varepsilon_{p_n}} {	1+\varepsilon_{p_n}}
			\Big)  |a_n| 
			=
			\mathop{\sum}\limits_{n=1}^{m} 	\frac{|a_n| } {	1+\varepsilon_{p_n}} 	
			\stackrel{\eqref{znnorm1}}{\leq} 	
			\mathop{\sum}\limits_{n=1}^{m}|a_n|   	\mathop{\sum}\limits_{k=p_n}^{q_n}|b_{k}|  = \mathop{\sum}\limits_{n=1}^{m}  	\mathop{\sum}\limits_{k=p_n}^{q_n} |a_n|  |b_{k}|
			\\ & \nonumber =
			\mathop{\sum}\limits_{n=1}^{m}  	\mathop{\sum}\limits_{k=p_n}^{q_n} |a_n  b_{k}|
			\stackrel{\text{Lem.\ref{equivalentconditionsail1} (iii)}}{\leq}
			\norm{
				\mathop{\sum}\limits_{n=1}^{m}a_{n}
				\left(
				\mathop{\sum}\limits_{k=p_{n}}^{q_n} b_k x_k
				\right)
			}= 
			\norm{ \mathop{\sum}\limits_{n=1}^{m}a_{n}y_{n}
			}
			\\ & \nonumber
			\,	\leq 
			\mathop{\sum}\limits_{n=1}^{m}|a_{n}| 
			\norm{
				y_n
			}
			\stackrel{\eqref{znnorm1}}{=}	\mathop{\sum}\limits_{n=1}^{m}|a_{n}|.   
		\end{align*}	
		Since $\left\{ \frac{\varepsilon_{p_n}} {	1+\varepsilon_{p_n}}
		\right\}_{n=1}^\infty
		$ is a null sequence of positive numbers, 
		it follows    
		that 
		$\{
		y_n
		\}_{n=1}^\infty$  is an asymptotically isometric  copy of
		\( \ell_1 \) which satisfies Lemma ~ \ref{equivalentconditionsail1} ~(i).  	
	\end{proof}

	Asymptotically isometric $\ell_1$-copies are stable with respect to perturbations by norm null
sequences   
        \cite[Lemma 3.1]{P02}. % A complete proof is provided below. 

	\begin{lemma}\label{perturbationnormlized} 
		Suppose that   $\{
		x_n
		\}_{n=1}^\infty$ and $\{y_n\}_{n=1}^\infty$  are  basic sequences in a %real/complex   
		Banach space $(X, \norm{\cdot})$. 
		If $\{
		x_n
		\}_{n=1}^\infty$  is an  asymptotically isometric  copy of
		\( \ell_1 \) and  $\lim\limits_{n\to\infty} \norm{ x_n-y_n }=0$,  then there exists a  subsequence of $\{y_n\}_{n=1}^\infty$ (still denoted by $\{y_n\}_{n=1}^\infty$) such that 
		the  basic sequence 
		$\{z_n\}_{n=1}^\infty:=\Big\{
		\frac{\norm{x_n}}{\norm{y_n}} y_n
		\Big\}_{n=1}^\infty $  is  an   asymptotically isometric  copy of
		\( \ell_1 \).
		In particular, if 	 $\{
		x_n
		\}_{n=1}^\infty$  is  normalized, then 
		the normalized basic sequence 
		$\{z_n\}_{n=1}^\infty:=\Big\{
		\frac{y_n}{\norm{y_n}} 
		\Big\}_{n=1}^\infty $  is  an   asymptotically isometric  copy of
		\( \ell_1 \).
	\end{lemma}

	\begin{lemma}\label{blocksum=1asyl1}
		Let $(X, \norm{\cdot})$ be a %real/complex  
		Banach space.   
		Suppose that  $\{
		x_k
		\}_{k=1}^\infty \subseteq X $  is an asymptotically isometric  copy of
		\( \ell_1 \) and  
		$\{u_n\}_{n=1}^{\infty}:=\left\{
		\mathop{\sum}\limits_{k= r_n}^{s_n} a_k x_k
		\right\}_{n=1}^{\infty}$ is  a block sequence. 
		If  $\{a_k\}_{k= 1}^{\infty} $ satisfies  $\lim\limits_{n\to 
			\infty}
		\mathop{\sum}\limits_{k=r_n}^{s_n} |a_k|=1$,  
		then  there exists a sequence 	$ \{ 
		\lambda_n \}_{n=1}^\infty $   of  
		positive scalars  with  $\lim\limits_{n\to \infty} 	\lambda_n=1$  such that 
		$\{
		\lambda_n u_n\}_{n=1}^\infty $   is an asymptotically
		isometric copy of $\ell_1$. 
		In particular, 	if  $\{a_k\}_{k= 1}^{\infty} $ satisfies  $\mathop{\sum}\limits_{k=r_n}^{s_n} |a_k|=1$ for all $n\geq 1$,  
		then  
		$\{
		u_n
		\}_{n=1}^\infty $  is  an asymptotically isometric  copy of
		\( \ell_1 \).
		
	\end{lemma}
	\begin{proof}
		There exists a null sequence \(\{\varepsilon_k\}_{k=1}^\infty\) of positive numbers such that  for arbitrary sequence \(\{b_n\}_{n=1}^{m}\in \ell_1\) of  scalars, 
		\begin{align}\label{4-2} 
			&\notag  \quad  \mathop{\sum}\limits_{n=1}^{m}  |b_n|
			\mathop{\sum}\limits_{k= r_n}^{s_n} |  a_k|
			=
			\mathop{\sum}\limits_{n=1}^{m} 
			\mathop{\sum}\limits_{k= r_n}^{s_n} |  b_n a_k|
			\stackrel{\text{Lem.\ref{equivalentconditionsail1}(iii)}}{\leq}  \norm{\mathop{\sum}\limits_{n=1}^{m}  b_n u_n
			}
			=	\norm{\sum_{n=1}^{m}  	 	\sum_{k= r_n}^{s_n} b_n  a_k x_k		} 		\\	&  		\leq 
			\mathop{\sum}\limits_{n=1}^{m}  |b_n|   \norm{ 	u_n	}
			=
			\mathop{\sum}\limits_{n=1}^{m}  |b_n|   \norm{ 	\mathop{\sum}\limits_{k= r_n}^{s_n} a_k x_k}
			\stackrel{\text{Lem.\ref{equivalentconditionsail1}(iii)}}{\leq} 	\mathop{\sum}\limits_{n=1}^{m}  	 (1 +\varepsilon_{r_n})   |b_n| 	\mathop{\sum}\limits_{k= r_n}^{s_n} |  a_k|. 
		\end{align}
		Define $ \frac{1	}{\lambda_n }  :=
		\mathop{\sum}\limits_{k=r_n}^{s_n} |a_k|$ for $n\geq 1$. Then, we have 
		\begin{align*}
			\mathop{\sum}\limits_{n=1}^{m}  |b_n| & \,\, =\mathop{\sum}\limits_{n=1}^{m} 	 |b_n  \lambda_n | 
			\mathop{\sum}\limits_{k= r_n}^{s_n} |  a_k|
			\stackrel{\eqref{4-2}}{\leq}	
			\norm{\mathop{\sum}\limits_{n=1}^{m}  b_n  \lambda_n  u_n
			}	\\ & 
			\stackrel{\eqref{4-2}}{\leq}	
			\mathop{\sum}\limits_{n=1}^{m}  	 (1 +\varepsilon_{r_n})   |b_n  \lambda_n  | 	\mathop{\sum}\limits_{k= r_n}^{s_n} |  a_k|
			=		\mathop{\sum}\limits_{n=1}^{m}  	 (1 +\varepsilon_{r_n})   |b_n|.  
		\end{align*} 	
		It follows that 
		$\{		\lambda_n u_n \}_{n=1}^\infty$ is an asymptotically isometric
		copy of $\ell_1$ which satisfies Lemma ~ \ref{equivalentconditionsail1} (ii).  
	\end{proof}

    % {\color{red}(Comment from YN to YN: This tells us that we can omit arbitrary many of $x_k$'s in the definition of $u_n$'s (by setting $r_1$ to be big enough), but still produce an asymptotically isometric copy of $\ell_1$. This sounds counterintuitive initially, but it is correct.)}

	%\section{An infinite version of Abramovich--Zaidenberg theorem}

	\bibliographystyle{amsalpha}

\begin{thebibliography}{99}
		
		%\bibitem{A91} Y.~A. Abramovich, New classes of spaces on which compact operators satisfy the Daugavet equation, J. Operator Theory {\bf 25} (1991), no.~2, 331--345.
		
		\bibitem{AAB1991}
		Y. Abramovich, C. Aliprantis,  O. Burkinshaw, 
		{\it 	The Daugavet equation in uniformly convex Banach spaces}, 
		J. Funct. Anal. 97(1) (1991), 215--230.
		%\bibitem{AK08}M.D. Acosta, A. Kami\'{n}ska, {\it Weak neighborhoods and the Daugavet property of the interpolation spaces $L^1 + L^{\infty} $ and $L^1\cap L^{\infty}$}, Indiana Univ. Math. J.  57 (2008), 77--96.
		
		
		\bibitem{AKM12}
		M. Acosta, A. Kami\'{n}ska, M. Masty\l o, {\it The Daugavet property and weak neighborhoods in Banach lattices}, J. Convex Anal. %Journal of Convex Analysis, 
		19(3) (2012), 875--912.
		
		
		\bibitem{AKM15}
		M. Acosta, A. Kami\'{n}ska, M. Masty\l o, {\it The Daugavet property in rearrangement invariant spaces}, 
		Trans. Amer. Math. Soc.  367 (2015), 4061--4078.
		
		\bibitem{AK06}
		F. Albiac, N. Kalton, {\it Topics in Banach space theory}, Springer,  New York,  2006. 
		
		\bibitem{AB06} C. Aliprantis, O. Burkinshaw, {\it  	Positive  operators},  Springer, Dordrecht, 2006. 
        	%\bibitem{AL75} J. Arazy, J. Lindenstrauss, {\it Some linear topological properties of the spaces $C_p$ of operators on Hilbert space}, Compos. Math. 30 (1975), 81--111.
		%\bibitem{AHS21}S. Astashkin, J. Huang, F. Sukochev,{\it 	Lack of isomorphic embeddings of symmetric function spaces into operator ideals},J.Funct. Anal.280(5) 	(2021),	108895. 
        
            
            
         %\bibitem{BM05}     J. Becerra-Guerrero, M. Mart´ın,    {\it The Daugavet property of $C^*$-algebras, $JB^∗$-triples, and of their isometric preduals}, J. Funct. Anal. 224 (2005), 316--337.
%\bibitem{B83}R.D. Bourgin, {\it Geometric aspects of convex sets with the Radon-Nikodym Property}, Springer-Verlag, Berlin, 1983.
		\bibitem{BS88}
		C. Bennett, R. Sharpley, {\it Interpolation of operators}, Academic Press, Inc., Boston, MA, 1988.
		\bibitem{B67}
		G. Birkhoff, {\it Lattice theory}, 
		American Mathematical Society, Providence, RI, 1967. 

\bibitem{CPS00} A. 
 Carey, J. Phillips,  F. 
 Sukochev, {\it On unbounded $p$-summable Fredholm modules}, Adv. Math.  151(2) (2000), 140--163. 
 
		\bibitem{CS94}
		V. Chilin, F. Sukochev, {\it Weak convergence in non-commutative  symmetric spaces},  J.  Operator Theory 31(1)  (1994), 35--65. 
		\bibitem{CDS97}
		V. Chilin, P. Dodds,  F. Sukochev, 
		{\it The Kadec--Klee property in symmetric
			spaces of measurable operators}, Israel J. Math. 97 (1997), 203--219. 

		%	\bibitem{C90} J. B. Conway, {\it 	A course in functional analysis}, 	Springer-Verlag, New York, 1990. 	
		\bibitem{C12}
		M. Czerw\'{i}nska, 
		{\it Complex uniform rotundity in symmetric spaces of measurable operators}, 
		J. Math. Anal. Appl. 395(2) (2012),  501--508.
		
		
		
		
		\bibitem{CK17}
		M. Czerw\'{i}nska,  A.  Kam\'{i}nska,  {\it 
			Geometric properties of noncommutative symmetric spaces of measurable operators and unitary matrix ideals}, Comment. Math. 57(1) (2017),
		45--122.
		
		
		%\bibitem{Dan25}S. Dantas, M. Mart\'{i}n, Y. Perreau, 	{\it The Daugavet property is equivalent to the polynomial Daugavet property},	arXiv: 2507.06143.  
		
		\bibitem{D63}
		I. Daugavet,  
		{\it A property of completely continuous operators in the space C},   
		Uspehi Mat. Nauk   18(5)  (1963), 157--158 (in Russian).
		\bibitem{DisU77}
		J. Diestel,   J. Uhl,  Jr., 
		{\it   Vector measures}, 
		American Mathematical Society, Providence, RI, 1977. 
		
		\bibitem{Dis84}
		J. Diestel, {\it  
			Sequences and series in Banach spaces}, 
		Springer-Verlag, New York, 1984. 
		\bibitem{DGH00}
		S. Dilworth, M. Girardi, J. Hagler, {\it  Dual Banach spaces which contain an isometric
			copy of $L_1$}, Bull. Pol. Acad. Sci. Math.  48 (2000), 1--12.
		
		%\bibitem{DDP891} P.G. Dodds, T. Dodds,  B. de Pagter, {\it Noncommutative Banach function spaces}, Math. Z. 201(4) (1989), 583--597.
		%\bibitem{DDP892} P.G. Dodds, T. Dodds,  B. de Pagter, {\it A general Markus inequality},Miniconference on Operators in Analysis (Sydney, 1989), Proc. Centre Math. Anal. Austral.Nat. Univ. 24, Austral. Nat. Univ. Canberra, 1990, 47--57.
		
		%\bibitem{DDP93} P.G. Dodds, T. Dodds, B. de Pagter, {\it Noncommutative K\"{o}the duality}, Trans. Amer. Math. Soc.  339 	(1993), 717--750.
		
        
        
        %\bibitem{DDP932}	P. Dodds, T. Dodds, {\it  Some aspects of the theory of symmetric operator spaces}, Quaest. Math. 18(1-3) (1995) 47--89. First International Conference in Abstract Algebra, Kruger Park, 1993.
		\bibitem{DDS97}
		P. Dodds, T. Dodds, F. Sukochev,
		{\it Lifting of Kadec--Klee properties to symmetric spaces of measurable operators}, Proc. Amer. Math. Soc. 125 (1997), 
		1457--1467.
		\bibitem{DDST05}
		P. Dodds,  T. Dodds, F. Sukochev, S. Tikhonov, {\it A non-commutative Yosida--Hewitt theorem and convex
			sets of measurable operators closed locally in measure}, Positivity   9 (2005), 457--484. 
		
		\bibitem{DD12} P. Dodds,  B. de Pagter, {\it The non-commutative Yosida--Hewitt decomposition revisited},
		Trans. Amer. Math. Soc.  364 (2012),  6425--6457. 
		
		\bibitem{DD14}
		P. Dodds, B. de Pagter, {\it  Normed K\"{o}the spaces: a non-commutative viewpoint}, Indag. Math.   25 (2014), 206--249.
		
		
		
		
		\bibitem{DPS}
		P. Dodds, B. de Pagter,
        F.  Sukochev, 
		{\it Noncommutative integration and operator theory},  Birkh\"{a}user/Springer, Cham, 2023.
		
		\bibitem{DPS16}
		P. Dodds, B. de Pagter,  F.  Sukochev, {\it Sets of uniformly absolutely continuous norm
			in symmetric spaces of measurable operators,} Trans. Amer. Math. Soc.  368(6) (2016), 
		4315--4355. 
		\bibitem{DR00}
		P. Dowling, N. Randrianantoanina, 
		{\it Asymptotically isometric copies of $\ell_{\infty}$ in Banach spaces and a theorem of Bessaga and Pe\l czy\'{n}ski},  Proc. Amer. Math. Soc.   128(11)  (2000), 3391--3397. 
		\bibitem{DR00two}
		P. Dowling, C. Lennard,  B. Turett, {\it  
			Some fixed point results in $l^1$ and $c_0$}, 
		Nonlinear Anal. 39(7) (2000),  
		929--936.
		
		
		\bibitem{CMSBook11}
		M. Fabian, P. Habala, P.  H\'{a}jek, V. Montesinos,  %Santaluc\'{i}a, 
		V. Zizler, 
		{\it Banach space theory.  
			The basis for linear and nonlinear analysis},  Springer, New York, 2011.  

\bibitem{FS65} C. Foias,  I. Singer, 
{\it Points of diffusion of linear operators and almost diffuse operators in spaces of continuous functions}, Math. Z.  87 (1965), 434--450.

  %\bibitem{G97}   T.A.   Gillespie, {\it Boundedness criteria for Boolean algebras of projections}, J. Funct. Anal. 148(1) (1997),  70--85.

		
   %\bibitem{GL74}     Y.    Gordon,  D.R. Lewis,   {\it Absolutely summing operators and local unconditional structures}, Acta Math. 133 (1974), 27--48.

\bibitem{HaglerThesis}
		J. Hagler, {\it  Embeddings of $L_1$ into conjugate Banach spaces}, Ph.D. Thesis, University of California, Berkeley, Calif.,  1972.
		
		\bibitem{H76}
		R. Haydon,
		{\it  Some more characterizations of Banach spaces containing $\ell_1$}, 
		Math. Proc. Cambridge Philos. Soc.  80(2) (1976), 269--276.
		
		
		\bibitem{HPS22}
		J. Huang, M. Pliev, F. Sukochev,
		{\it 
			(Non-)Dunford--Pettis operators on noncommutative symmetric spaces},
		J. Funct. Anal.  282(11) (2022), 29 pp.
		\bibitem{HNPS24Tran}
		J. Huang,  Y. Nessipbayev,   M. Pliev, F.  Sukochev, 
		{\it The Gelfand--Phillips and Dunford--Pettis type properties in bimodules of measurable operators}, 
		Trans. Amer. Math. Soc.  377(9) (2024), 6097--6149.
		
		
		
		%	\bibitem{HKM00}	H. Hudzik,  A. Kami\'{n}ska, M. Masty\l o, {\it		Monotonicity and rotundity properties in Banach lattices}, 	Rocky Mountain J. Math. 30(3) (2000), 933--950.
		
		
		
		
		
		
		
		%\bibitem{HSSZ25}J. Huang, O.  Sadovskaya, F. Sukochev, D. Zanin,{\it Lack of isomorphic embeddings of sequence space $\ell_{p,q}$ into $L_{p,q}(\cM,\tau)$ over a noncommutative probability space},Adv. Math. 482, Part A,(2025), 110571.
		
		\bibitem{KP62} 
		M. Kadec,  A. Pe\l czy\'{n}ski, 
		{\it  Bases, lacunary sequences and complemented subspaces
			in the spaces $L_p$}, Studia Math. 21 (1962), 161--176.
		
		\bibitem{KSSW00}
		V.  Kadets,   R. Shvidkoy,  G. Sirotkin, D. Werner, 
		{\it  Banach spaces with the Daugavet property}, 
		Trans. Amer. Math. Soc.  352(2) (2000), 855--873.
		
		%\bibitem{KMM07}V. Kadets,  M. Mart\'{i}n, J. Mer\'{i}, {\it  Norm equalities for operators on Banach spaces},  Indiana Univ. Math. J.  56(5) (2007),  2385–2411. 
		
		\bibitem{KMMW13}  V. Kadets, M. Mart\'{i}n, J. Mer\'{i},  D. Werner, 
		{\it  Lushness, numerical index
			1 and the Daugavet property in rearrangement invariant spaces}, Canad. J. Math.  65(2) (2013),
		331--348.
		
		\bibitem{KMZW25}  V. Kadets, M. Mart\'{i}n, A. Rueda Zoca,  D. Werner, 
		{\it Banach spaces
			with the Daugavet
			property}, %preprint,
		unpublished book, preliminary version, uploaded at DIGIBUG (official repository of the University of Granada). 
		
		
		
		\bibitem{K79}
		N.  Kalton, {\it  Embedding $L_1$ in a Banach lattice}, %Isr. J. Math.
		Israel J. Math.  32 (1979),  209--220.
		%\bibitem{Kalton80}N. Kalton, {\it Linear operators on $L_p$ for $0 < p < 1$}, Trans. Amer. Math. Soc.  259 (1980), 319--355.
		
		\bibitem{Kalton_S}
		N. Kalton, F.  Sukochev, 
		{\it Symmetric norms and spaces of operators}, 
		J. Reine Angew. Math. 621 (2008), 81--121.
		
		
		%\bibitem{KL04}	A. Kami\'{n}ska, H.J. Lee, {\it 	M-ideal properties in Marcinkiewicz spaces}, Comment. Math.  2004, Tomus specialis in Honorem Juliani Musielak, 123--144.	
		\bibitem{KPS}
		S. Kre\u{\i}n, Y.  Petunin, E. Sem\"{e}nov,
		{\it Interpolation of linear operators},
		American Mathematical Society, Providence, RI, 1982.
		%	\bibitem{L25}	H.J.  Lee,  {\it  Monotonicity and complex convexity in Banach lattices,} 	J. Math. Anal. Appl.	307(1) (2005),  86--101.
		%\bibitem{LT24}H.J.  Lee,   H. Tag, {\it	Remark on the Daugavet property for complex Banach spaces}, 	Demonstr. Math. 57(1) (2024),  21 pp.

	%\bibitem{Lewis75}  D.R.   Lewis, {\it An isomorphic characterization of the Schmidt class}, Compositio Math. 30(3)  (1975),  293--297.
		\bibitem{LT1} 
		J. Lindenstrauss,  L. Tzafriri, {\it 
			Classical Banach spaces. I.
			Sequence spaces},  
		Springer-Verlag, Berlin-New York, 1977.
		\bibitem{LT2} 
		J. Lindenstrauss, L. Tzafriri, 
		{\it Classical Banach spaces. II. Function spaces,} Springer-Verlag, Berlin-New York, 1979.
		%\bibitem{LRZ2021}G. L\'{o}opez-P\'{e}rez, A. Rueda Zoca, {\it  L-orthogonality, octahedrality and Daugavet property in Banach spaces}, Adv. Math.  383 (2021), 107719.
		
		%\bibitem{M98}R. E. Megginson, {\it An introduction to Banach space theory}, Springer-Verlag, New York, 1998. 
		%\bibitem{N55}H. Nakano, {\it Semi-ordered linear spaces}, Japan society for the promotion of science, Tokyo, 1955. 
		
		
		\bibitem{LSZ}
		S. Lord, F. Sukochev, D. Zanin,
		{\it Singular traces: theory and applications,} De Gruyter, Berlin, 2013.
		\bibitem{Lo66}G.J. Lozanovski\u{\i}, {\it   On almost integral operators in KB-spaces},  Vestnik Leningrad Univ. Mat. Mekh. Astr.  21(7) (1966), 35--44 (in Russian).
		%\bibitem{Lu63}W.A.J. Luxemburg, A.C.  Zaanen, {\it  Notes on Banach function spaces. II.},  Indag. Math. 25 (1963), 148--153.
		%\bibitem{LZ71}W.A.J. Luxemburg, A.C. Zaanen, {\it Riesz spaces. Vol. I}, North-Holland Publishing Co.,   Amsterdam-London;  American Elsevier Publishing Co. Inc.,  New York, 1971.
		
			%\bibitem{M67}Ch. McCarthy, {\it $c_p$}, Israel J. Math. 5 (1967), 249--271.
	\bibitem{MN91}
		P. Meyer-Nieberg, {\it 
			Banach lattices}, 
		Springer-Verlag, Berlin, 1991.
		
		
		
		\bibitem{N74}
		E. Nelson, {\it Notes on non-commutative integration}, J. Funct. Anal. 15 (1974), 103--116.
		
		
		\bibitem{O02}
		T. Oikhberg,  {\it  The Daugavet property of C$^*$-algebras and non-commutative $L_p$-spaces},
		Positivity  6(1) (2002),  59--73.

      
		\bibitem{P02}
		H. Pfitzner, {\it 
Perturbation of $l^1$-copies and measure convergence in preduals of von Neumann algebras},  
J. Operator Theory  47(1) (2002),  145--167.
    
		\bibitem{Ran03}
		N. Randrianantoanina, 
		{\it Embeddings of $\ell_p$ into non-commutative spaces}, 
		J. Aust. Math. Soc. 74(3) (2003),  331--350.
		
		%\bibitem{Ro73}H.P. Rosenthal,{\it  On subspaces of $L_p$}, Ann. Math. 97 (1973), 344--373.
		\bibitem{Ro74}
		H. Rosenthal, 
		{\it The heredity problem for weakly compactly generated Banach spaces},
		Compositio Math. 28 (1974), 83--111. 
       % \bibitem{R94} H.  Rosenthal, {\it  A characterization of Banach spaces containing $c_0$}, J. Amer. Math. Soc. 7(3) (1994),  707--748.
		\bibitem{R74}
		H. Rosenthal, {\it 
			A characterization of Banach spaces containing $\ell^1$}, 
		Proc. Nat. Acad. Sci. U.S.A. 71 (1974), 2411--2413.

\bibitem{Rudin}
W. Rudin, 
{\it Functional analysis,} second edition, McGraw-Hill, Inc. Singapore, 1991. 

\bibitem{RTV} A.~Rueda Zoca, P.~Tradacete,  I.~Villanueva, \emph{Daugavet property in tensor product spaces}, J. Inst. Math. Jussieu  20(4) (2021), 1409--1428.

\bibitem{SS14} E. Semenov,  F. Sukochev, \emph{ Sums and intersections of symmetric operator spaces}, J. Math. Anal. Appl.  414(2) (2014),  742--755.

		%\bibitem{Sh00}R.V. Shvydkoy,  {\it Geometric aspects of the Daugavet property}, 	J. Funct. Anal. 176(2) (2000),  198--212.

\bibitem{S71} S.B. Stechkin, {\it  The approximation of continuous periodic functions by Favard sums}, Trudy Mat. Inst. Steklov. 109 (1971), 26--34.

		%\bibitem{SC91} F.A. Sukochev,  V.I.  Chilin, {\it  Symmetric spaces over semi-finite von Neumann algebras}, Dokl. Akad. Nauk SSSR, 313(4) (1990),  811--815; 	English transl. Soviet Math.	Dokl. 42(1) (1991),  97--101.
     	\bibitem{S95}
		F.  Sukochev,  {\it    
On the uniform Kadec--Klee property with respect to convergence in measure}, 
J. Austral. Math. Soc. Ser. A    59(3) (1995), 343--352.

\bibitem{SC88}
      F.  Sukochev,  V. Chilin, 
      {\it The triangle inequality for operators that are measurable with respect to Hardy--Littlewood order},  Izv. Akad. Nauk UzSSR Ser. Fiz.-Mat. Nauk
     (4)   (1988),  44--50  (in Russian).

		\bibitem{S96}
		F.  Sukochev,  {\it 
			Non-isomorphism of $L_p$-spaces associated with finite and infinite von Neumann algebras}, 
		Proc. Amer. Math. Soc.  124(5) (1996),  1517--1527.
		\bibitem{S00}
		F. Sukochev, {\it Linear-topological classification of separable $L_p$-spaces associated with von
			Neumann algebras of type I}, Israel J. Math.  115 (2000), 137--156. 
		\bibitem{Sukochev16}F.  Sukochev, {\it H\"{o}lder inequality for symmetric operator spaces and trace property of K-cycles,} Bull. London Math. Soc. 48(4) (2016), 637--647.
	%	\bibitem{Veeorg23} T. Veeorg,     {\it Characterizations of Daugavet points and delta-points in Lipschitz-free spaces}, Studia Math.   268 (2023), 213--233.
		
		\bibitem{Takesakibook}	M. Takesaki, {\it Theory of operator algebras. I},	Springer-Verlag, Berlin, 2002.
        \bibitem{Werner01} D. Werner,
        {\it Recent progress on the Daugavet property}, Irish Math. Soc. Bull. (46)  (2001), 77--97.
		\bibitem{W99}
		W. Wnuk, 
		{\it  Banach lattices with order continuous norms}, Advanced Topics in Mathematics,
		Adam Mickiewicz Univ. Poznan, Polish Scientific Publishers PWN, 1999.
		\bibitem{W92}  P. Wojtaszczyk, {\it Some remarks on the Daugavet equation}, Proc. Amer. Math. Soc.  115(4) (1992), 1047--1052.
		\bibitem{W03}
		M. W\'{o}jtowicz,  {\it Contractive projections onto isometric copies of $L_1(\nu)$ in strictly
			monotone Banach lattices}, Bull. Pol. Acad. Sci. Math. 51(1) (2003),  1--12.
		
		
		
		
		
		
		
	\end{thebibliography}

\end{document}